\documentclass[11pt, twoside]{amsart}

\title{The Relative Monoidal Center and Tensor Products of Monoidal Categories}

\date{\today}
\author{Robert Laugwitz}
\address{Department of Mathematicss, Rutgers University,
Hill Center, 110 Frelinghuysen Road,
Piscataway, NJ 08854-8019}
\email{robert.laugwitz@rutgers.edu}
\urladdr{https://www.math.rutgers.edu/~rul2/}

\usepackage{comment}
\usepackage{microtype}
\usepackage{mathabx}
\usepackage{import}
\usepackage{amsmath}
\usepackage{amsfonts}
\usepackage{amsthm}
\usepackage{amssymb}
\usepackage{dsfont}
\usepackage[alphabetic, initials]{amsrefs}
\usepackage[english]{babel}
\usepackage{url}
\usepackage{graphicx}
\usepackage[
colorlinks=true,
linkcolor=black, 
anchorcolor=black,
citecolor=black,
urlcolor=black, 
]{hyperref}
\usepackage[all]{xy}
\usepackage{geometry}

\font\sc=rsfs10

\newcommand{\longrightleftarrows}{\resizebox{18pt}{7pt}{$\rightleftarrows$}}

\newcommand{\leftexp}[2]{{\vphantom{#2}}^{#1}{#2}}
\newcommand{\leftexpsub}[3]{{\vphantom{#3}}^{#1}_{#2}{#3}}
\newcommand{\adj}[4]{#1\colon #2~\longrightleftarrows~ #3\colon #4}

\newcommand{\op}[1]{\mathrm{#1}}
\newcommand{\oop}{\mathrm{op}}
\newcommand{\cop}{\mathrm{cop}}
\newcommand{\ov}[1]{\overline{#1}}

\newcommand{\lYD}[1]{\leftexpsub{#1}{#1}{\mathbf{YD}}}
\newcommand{\rYD}[1]{\mathbf{YD}^{#1}_{#1}}

\newcommand{\lmod}[1]{#1\text{--}\mathbf{Mod}}
\newcommand{\rmod}[1]{\mathbf{Mod}\text{--}#1}
\newcommand{\lrmod}[2]{#1\text{--}\mathbf{Mod}\text{--}#2}
\newcommand{\lrfmod}[2]{#1\text{--}\mathbf{Mod}^{\mathrm{fd}}\text{--}#2}
\newcommand{\llmod}[2]{#1\text{--}#2\text{--}\mathbf{Mod}}
\newcommand{\rrmod}[2]{\mathbf{Mod}\text{--}#1\text{--}#2}
\newcommand{\lcomod}[1]{#1\text{--}\mathbf{CoMod}}
\newcommand{\rcomod}[1]{\mathbf{CoMod}\text{--}#1}

\newcommand{\bal}[1]{#1\operatorname{-bal}}

\newcommand{\coev}{\operatorname{coev}}

\newcommand{\Cone}{\mathbf{coCone}_{\Bbbk}}

\newcommand{\Drin}{\operatorname{Drin}}
\newcommand{\ev}{\operatorname{ev}}

\newcommand{\Heis}{\operatorname{Heis}}

\newcommand{\Hopf}{\mathbf{Hopf}}
\newcommand{\Hom}{\operatorname{Hom}}
\newcommand{\cHom}{\mathbf{Hom}}
\newcommand{\ide}{\operatorname{Id}}

\newcommand{\Isom}{\mathbf{Isom}^{\otimes}}
\newcommand{\isomorph}{\stackrel{\sim}{\longrightarrow}}
\newcommand{\natisomorph}{\stackrel{\sim}{\Longrightarrow}}

\newcommand{\one}{\mathds{1}}

\newcommand{\reg}{\operatorname{reg}}

\newcommand{\triv}{\mathrm{triv}}

\newcommand{\Alg}{\mathbf{Alg}}
\newcommand{\BiAlg}{\mathbf{BiAlg}}
\newcommand{\BiMod}{\mathbf{BiMod}}

\newcommand{\Cat}{\mathbf{Cat}_{\Bbbk}}

\newcommand{\CoAlg}{\mathbf{CoAlg}}

\newcommand{\Fun}{\mathbf{Fun}_\Bbbk}
\newcommand{\Ends}{\mathbf{End}_\Bbbk}
\newcommand{\Homs}{\mathbf{Hom}_\Bbbk}

\newcommand{\MonCat}{\mathbf{MonCat}}

\newcommand{\Vect}{\mathbf{Vect}_\Bbbk}

\newcommand{\Ug}{U_q(\mathfrak{g})}

\providecommand{\fr}[1]{\mathfrak{#1}}

\providecommand{\op}[1]{\operatorname{#1}}

\newcommand{\mZ}{\mathbb{Z}}

\newcommand{\mF}{\mathbb{F}}

\newcommand{\cC}{\mathcal{C}}
\newcommand{\cD}{\mathcal{D}}
\newcommand{\cB}{\mathcal{B}}
\newcommand{\cE}{\mathcal{E}}

\newcommand{\cH}{\mathcal{H}}
\newcommand{\cO}{\mathcal{O}}
\newcommand{\cI}{\mathcal{I}}

\newcommand{\cT}{\mathcal{T}}
\newcommand{\cV}{\mathcal{V}}
\newcommand{\cW}{\mathcal{W}}
\newcommand{\cM}{\mathcal{C}}
\newcommand{\cN}{\mathcal{D}}

\newcommand{\cX}{\mathcal{X}}
\newcommand{\cZ}{\mathcal{Z}}

\newcommand{\rB}{\mathrm{B}}

\newcommand{\rF}{\mathrm{F}}
\newcommand{\rG}{\mathrm{G}}
\newcommand{\rH}{\mathrm{H}}
\newcommand{\rR}{\mathrm{R}}
\newcommand{\rT}{\mathrm{T}}

\newcommand{\bfD}{\mathbf{D}}

\newtheoremstyle{defstyle}
  {0.5cm}                   
  {0.5cm}                   
  {\normalfont}           
  {}     
  {\normalfont\bfseries}  
  {:}                     
  {0.3cm}              
  {\thmname{#1}\thmnumber{ #2}\thmnote{ (#3)}}

\numberwithin{equation}{section}

\newtheorem*{rep@theorem}{\rep@title}
\newcommand{\newreptheorem}[2]{%
\newenvironment{rep#1}[1]{%
 \def\rep@title{#2 \ref{##1}}%
 \begin{rep@theorem}}%
 {\end{rep@theorem}}}
\makeatother

\newtheorem{theorem}{Theorem}[section]

\newtheorem{proposition}[theorem]{Proposition}
\newreptheorem{proposition}{Proposition}
\newtheorem{corollary}[theorem]{Corollary}
\newreptheorem{corollary}{Corollary}
\newtheorem{lemma}[theorem]{Lemma}

\newtheorem{theorem*}{Theorem}
\newreptheorem{theorem}{Theorem}

\theoremstyle{definition}
\newtheorem{definition}[theorem]{Definition}

\newtheorem{assumption}[theorem]{Assumption}

\theoremstyle{remark}
\newtheorem{example}[theorem]{Example}
\newtheorem{remark}[theorem]{Remark}

\begin{document}
\begin{abstract}
This paper develops a theory of monoidal categories relative to a braided monoidal category, called \emph{augmented monoidal categories}. For such categories,  balanced bimodules are defined using the formalism of balanced functors. The two main constructions are a relative tensor product of monoidal categories as well as a relative version of the monoidal center, which are Morita dual constructions. A general existence statement for a relative tensor products is derived from the existence of pseudo-colimits.

In examples, a category of locally finite weight modules over a quantized enveloping algebra is equivalent to the relative monoidal center of modules over its Borel part. A similar result holds for small quantum groups, without restricting to locally finite weight modules. More generally, for modules over braided bialgebras, the relative center is shown to be equivalent to the category of braided Yetter--Drinfeld modules (or crossed modules). This category corresponds to modules over the braided Drinfeld double (or double bosonization) which are locally finite for the action of the dual.
\end{abstract}
\maketitle
\vspace{-10pt}
\tableofcontents


\section{Introduction}\label{introduction}

The subject of this paper is a categorical construction that gives modules over quantum groups (both small and generic) in specific examples --- the \emph{relative monoidal center} $\cZ_\cB(\cM)$. Here $\cM$ is what is called a \emph{$\cB$-augmented monoidal category} in this paper, for $\cB$ a braided monoidal category. Further, categorical modules over the relative center are studied. The categorical Morita dual of $\cZ_\cB(\cM)$ is shown to be the relative tensor product $\cM\boxtimes_\cB \cM^{\oop}$.

\subsection{Background}\label{background}

The \emph{Drinfeld double} (or \emph{quantum double}, \cite{Dri}) of a Hopf algebra is a fundamental construction in the field of quantum algebra. An important application is the algebraic construction of classes of three-dimensional topological field theories (TFTs): For example, Dijkgraaf--Witten and Reshetikhin--Turaev TFTs. In the former, the Drinfeld double of a finite group appears \cite{DPR} and the latter is obtained by studying a quotient of a category of modules over the quantum groups $U_q(\fr{g})$  \cite{RT}. This quantum group can either be constructed as a quotient of the Drinfeld double of its Borel Hopf subalgebra \cite{Dri}, or using a version of the Drinfeld double construction \cites{Maj1,Maj3}, called \emph{double bosonization} or the \emph{braided Drinfeld double}.

The center $\cZ(\cM)$ of a monoidal category $\cM$ \cites{Maj2,JS} provides a interpretation of the Drinfeld double construction. It is a braided monoidal category equivalent to the endomorphism category of the regular bimodule over $\cM$, i.e.
$$\cZ(\cM)\simeq \cHom_{\cM\text{--}\cM}(\cM,\cM),$$
cf. \cite{EGNO}*{Section 7.13}.
A natural source of monoidal categories are module categories over Hopf algebras $H$, $\cM=\lmod{H}$. For such categories, an equivalent description of the monoidal center is given by the category $\lYD{H}$ of \emph{crossed modules} (also called  \emph{Yetter--Drinfeld modules}) over $H$, see e.g. \cite{Mon}*{Section 10.6}.

In the vision of Crane--Frenkel \cite{CF}, it is suggested that  algebras (and their representation theory) should ultimately be replaced by categories (and their categorical modules) in order to algebraically construct four-dimensional invariants of manifolds. Studying categorical modules over monoidal categories can be seen as a step into this direction. 

It was shown in \cites{Ost,EGNO} that categorical modules over the monoidal center $\cZ(\cM)$
 can be constructed using categorical bimodules over $\cM$. In the case where $\cM$ is a finite $\Bbbk$-multitensor category, this construction gives all categorical modules over $\cZ(\cM)$. This finiteness condition holds, for example, for the category $\lmod{H}^{\mathrm{fd}}$ of finite-dimensional modules over a finite-dimensional Hopf algebra $H$ \cite{EGNO}*{Section 5.3}. Hence, all categorical modules over $\lmod{\Drin(H)}^{\mathrm{fd}}$ are described through this construction. 
 
\subsection{Quantum Groups and the Relative Center}\label{motivation}

The main motivation for this paper stems from the representation theory of the quantum group $U_q(\fr{g})$ \cites{Dri,Lus} and its relation to the monoidal center.
The category of Yetter--Drinfeld modules, describing the center of the monoidal category $\lmod{B}$ of modules over a bialgebra $B$, can --- more generally --- be defined for a bialgebra object $B$ in a braided monoidal category $\cB$, giving the braided monoidal category $\lYD{B}(\cB)$, see \cite{BD}. The negative nilpotent part $B=U_q(\fr{n}_-)$ of $U_q(\fr{g})$ is naturally such a bialgebra object (in fact, a Hopf algebra object) in the braided monoidal category $\cB=\lmod{U_q(\fr{t})}$, see e.g. \cites{AS,Lus}. Via the duality between $U_q(\fr{n}_-)$ and $U_q(\fr{n}_+)$ and a result of Majid  in \cite{Maj3}, we can view $\lYD{B}(\cB)$ in this example as the category of $U_q(\fr{g})$-modules which have a locally finite $u_q(\fr{n}_+)$-action. This tensor category contains the non-monoidal category $\cO_q$ associated to the quantum group \cite{AM}, which is an analogue of an important construction in Lie theory \cite{BGG}.

Motivated by the quantum group example, it is natural to expect the center of the monoidal category $\cM=\lmod{B}(\cB)$ to be equivalent to the above category $\lYD{B}(\cB)$. However, in this example, $\cZ(\cM)$ is a larger category, consisting of certain modules over $\Drin(U_q(\fr{b}_-))$, which is a Hopf algebra defined on the vector space $U_q(\fr{n}_+)\otimes U_q(\fr{t})^*\otimes U_q(\fr{t})\otimes U_q(\fr{n}_-)$. This algebra contains an additional copy of the dual of the Cartan part $U_q(\fr{t})$. Drinfeld takes the quotient by the Hopf ideal identifying $U_q(\fr{t})$ and $U_q(\fr{t})^*$ in \cite{Dri}.
In general, for a finite-dimensional braided bialgebra $B$  in a braided monoidal category $\cB=\lmod{H}$, for $H$ a quasi-triangular Hopf algebra, one can use the Radford biproduct  (or \emph{bosonization}) $B\rtimes H$ \cite{Rad} and find that
$$\cZ(\cM)\simeq \lYD{B\rtimes H}\simeq \lmod{\Drin(B\rtimes H)}.$$ 
One main construction of the present paper is a version $\cZ_{\cB}(\cM)$ of the monoidal center \emph{relative} to a braided monoidal category $\cB$. In the case where $\cM=\lmod{B}(\cB)$ studied above, we find that
$$\cZ_{\cB}(\cM)\simeq \lYD{B}(\cB)\simeq\lmod{\Drin_H(B^*,B)},$$
where $\Drin_H(B^*,B)$ is the double bosonization (or braided Drinfeld double) of Majid. Hence the relative monoidal center gives the category of $U_q(n_+)$-locally finite weight modules in the example of $U_q(\fr{g})$, and thus a refinement of the center construction. For an even root of unity $\epsilon$, the relative monoidal center of $\lmod{u_{\epsilon} (\fr{b}_-)}$ is equivalent to the category of modules over the small quantum group $u_\epsilon(\fr{g})$. Note that in the case $\cB=\Vect$ we recover the monoidal center $\cZ(\cM)$. 




\subsection{Summary of Results}
In the spirit of moving from the representation theory of algebras to modules of categories (in this case, monoidal categories), this paper studies the representation theory of the relative monoidal center  in the framework of \cite{EGNO}. 
After establishing the categorical setup and recalling generalities on categorical modules (Sections \ref{setup} \& \ref{catmodules}), the relative tensor product of categorical bimodules is reviewed in the  generality of finitely cocomplete $\Bbbk$-linear categories (Section \ref{reltensorsect} \& Appendix \ref{appendixA}).

The concept of the relative monoidal center from \cite{Lau} is refined using the language of $\cB$-balanced functors to allow a better study of its categorical modules. For this, $\cB$-augmented monoidal categories are introduced in Section \ref{augmentedcats}.
These can be thought of as categorical analogues of $C$-augmented $C$-algebras $R$ over a commutative $\Bbbk$-algebra $C$. 
Over a $\cB$-augmented monoidal category, we can study $\cB$-balanced bimodules (Section \ref{balancedbimodsect}). This construction is a categorical analogue of $R\otimes_C R^{\oop}$-modules, over a $C$-algebra $R$. In other words, $R$-bimodules for which the left and right $C$-action coincide.
Based on this concept, we present two constructions for a $\cB$-augmented monoidal category $\cM$: 
\begin{enumerate}
\item the \emph{relative monoidal center} $\cZ_{\cB}(\cM)$, which is defined as $\cZ_{\cB}(\cM)=\cHom_{\cM\text{--}\cM}^{\cB}(\cM,\cM),$
i.e. the category of endofunctors of $\cB$-balanced bimodule functors of the regular $\cM$-bimodule (Section \ref{monoidalcentersect});
\item 
the \emph{relative tensor product} $\cM\boxtimes_{\cB}\cM^{\oop}$ (Theorem \ref{relativeproductproperties}) which is a monoidal category, generalizing the tensor product of categories of  Kelly \cite{Kel}.
\end{enumerate}
The monoidal category $\cM\boxtimes_{\cB}\cM^{\oop}$ has the universal property that its categorical modules give all $\cB$-balanced bimodules over $\cM$ (Theorem \ref{bimoduletheorem}). We show that there is a natural construction to turn a $\cM$-bimodule into a $\cB$-balanced $\cM$-bimodule.
\begin{reptheorem}{adjunctionthm}
Restriction along the monoidal functor $\cM\boxtimes\cM^{\oop}\to \cM\boxtimes_{\cB}\cM^{\oop}$ has a left $2$-adjoint
$$(-)_{\cB}\colon \lmod{\cM\boxtimes\cM^{\oop}}\longrightarrow \lmod{\cM\boxtimes_{\cB}\cM^{\oop}}, \quad \cV\longmapsto\cV_\cB.$$
\end{reptheorem}
This construction is a categorical analogue of restricting a $R$-$R$-bimodule to the subspace on which the left and right $C$-action coincide. Moreover, for any $\cB$-augmented monoidal categories $\cM$ we prove that there is a bifunctor (see Section \ref{Moritasection})
$$\lmod{\cM\boxtimes_{\cB}\cM^{\oop}}\longrightarrow \lmod{\cZ_\cB(\cM)}.$$
This gives a way to produce categorical modules over the relative monoidal center from $\cB$-balanced $\cM$-bimodules. In the finite case, we generalize the result of \cites{EGNO,Ost} that $\cZ(\cM)$ and $\cM\boxtimes\cM^{\oop}$ are categorically Morita equivalent to the relative case.
\begin{reptheorem}{centermorita}
In the case where $\cM$ and $\cB$ are finite multitensor categories, the monoidal categories $\cZ_\cB(\cM)$ and $\cM\boxtimes_{\cB}\cM^{\oop}$ are \emph{categorically Morita equivalent}.
\end{reptheorem}
We can specialize these categorical constructions to categories of modules over bialgebras (or Hopf algebras) in a braided monoidal category. 
\begin{repproposition}{YDprop}[See also Example \ref{tensorBBmod}]
If $\cM=\lmod{B}(\cB)$ is the monoidal category of $B$-modules over a Hopf algebra object $B$ in a braided monoidal category $\cB$, then 
\begin{align*}
\cZ_{\cB}(\cM)\simeq \lYD{B}(\cB), && \cM\boxtimes_{\cB}\cM^{\oop}\simeq \lrmod{B}{B}(\cB).
\end{align*}
\end{repproposition}
The former category is that of \emph{Yetter--Drinfeld modules} (or \emph{crossed modules}) in $\cB$ of \cite{BD}, while the latter category consists of objects in $\cB$ that have commuting left and right $B$-module structures. We further show that for algebra (or coalgebra) objects $A$ in $\cM$, the categories $\lmod{A}(\cM)$ (respectively, $\lcomod{A}(\cM)$) give a large supply of categorical modules over the relative monoidal center in Section \ref{YDactionsect}.  

In Section \ref{drinapp}, we specialize the setup further, assuming $\cB=\lmod{H}$, for $H$ a quasi-triangular Hopf algebra over a field $\Bbbk$. 
\begin{repcorollary}{drincor}
For Hopf algebras $B,C$ in $\cB=\lmod{H}$ with a non-degenerate duality pairing
\begin{align*}
\cM\boxtimes_{\cB}\cM^{\oop}\simeq \lmod{(B\otimes \leftexp{\cop}{B})\rtimes H}, && \cZ_{\cB}(\cM)\simeq \lmod{\Drin_H(C,B)}^{C\mathrm{-lf}},
\end{align*}
where $C\mathrm{-lf}$ denotes the full subcategory where the $C$-action is locally finite.
For $B$ finite-dimensional, $\cZ_{\cB}(\cM)\simeq \lmod{\Drin_H(C,B)}$.
\end{repcorollary}
The Hopf algebra $\Drin_H(C,B)$ is the \emph{double bosonization} of \cites{Maj1,Maj2} and referred to as the \emph{braided Drinfeld double} here. For the Hopf algebra $(B\otimes \leftexp{\cop}{B})\rtimes H$ see Definition \ref{BBHdef}.

Some results of Section \ref{reptheorysect} relate to results that were obtained in \cite{Lau2} by direct computations. In particular, the Radford--Majid biproduct $A\rtimes B^*$ is naturally a comodule algebra over the (braided) Drinfeld double of $B$. Further, these constructions give a natural map inducing $2$-cocycles over a (braided) Hopf algebra $B$ to $2$-cocycles over its (braided) Drinfeld double.

Applications to the representation theory of quantum groups are included in Sections \ref{qgroupssect} and \ref{smallqgroupssect}. Fix a Cartan datum $(I,\cdot)$, denote by $\fr{g}=\fr{n}_-\oplus \fr{t}\oplus \fr{n}_+$ the associated Lie algebra, and let $\epsilon$ be a primitive even root of unity. Then we derived the following results for the small quantum group.
\begin{reptheorem}{smallthm}
There is an isomorphism of Hopf algebras between $\Drin_{u_\epsilon(\fr{t})}(u_\epsilon(\fr{n}_+),u_\epsilon(\fr{n}_-))$ and $u_\epsilon(\fr{g})$. Thus, there is an equivalence of monoidal categories 
$$\cZ_{\lmod{u_\epsilon(\fr{t})}}(\lmod{u_\epsilon(\fr{b}_-)})\simeq \lmod{u_{\epsilon}(\fr{g})}.$$
\end{reptheorem}
\begin{repcorollary}{smallqcor}Let $\cM=\lmod{u_\epsilon(\fr{b_-})}^{\mathrm{fd}}$ and $\cB=\lmod{u_\epsilon(\fr{t})}^{\mathrm{fd}}$.
Then $\cM$ is $\cB$-augmented, and the monoidal categories $\lmod{u_\epsilon(\fr{g})}^{\mathrm{fd}}$ and $\cM\boxtimes_{\cB}\cM^{\oop}$ are categorically Morita equivalent. The latter is equivalent to finite-dimensional modules over a Hopf algebra $t_\epsilon (\fr{g})$ (see Definition \ref{tepdef}).
\end{repcorollary}
In \cite{EG} another approach displays the small quantum group as a Drinfeld double of a quasi-Hopf algebra.
In the case were $q$ is a generic parameter, it was shown in \cite{Maj3} that $U_q(\fr{g})$ is isomorphic as a Hopf algebra to the braided Drinfeld double (or \emph{double bosonization}) of  $U_q(\fr{n_-})$. 
\begin{reptheorem}{quantumcenter}
Let $\cB=\lcomod{L}_q$, for $L$ the root lattice, and $\cM=\lmod{U_q(\mathfrak{n}_-)}(\cB)$. Then the category $\cZ_{\cB}(\cM)$ is equivalent the category  $\lmod{U_q(\fr{g})}^{U_q(\fr{n}_+)\mathrm{-lfw}}$ of  $U_q(\fr{n}_+)$-locally finite weight modules over $U_q(\fr{g})$.
\end{reptheorem}
The categorical results of this paper show that there exists a bifunctor from $\lmod{\cM\boxtimes_{\cB}\cM^{\oop}}\simeq \lmod{T_q(\fr{g})}^{\op{w}}$ to categorical modules over $\lmod{U_q(\fr{g})}^{U_q(\fr{n}_+)\mathrm{-lfw}}$. See Definition \ref{Tqdef} for the Hopf algebra $T_q(\fr{g})$. This result provides a way to construct categorical modules over the latter category.
The representation-theoretic applications of these results will be explored further elsewhere.

\subsection{Applications to TFT}

Monoidal categories are closely connected to the construction of TFTs \cite{TV}. Working with finite tensor categories and their exact modules, relative tensor products of bimodules appear in the construction of $3$-dimensional TFTs with defects \cite{FSS}. In this context, the relative tensor product of bimodule categories, and (twisted versions of) the monoidal center are studied. This paper generalizes some of the categorical constructions used in this work, motivated by the representation theory of quantum groups, by allowing for more general tensor categories, and using the idea of viewing a monoidal category relative to a braided monoidal category $\cB$. Applications to TFTs are expected to be the subject of future work.

\subsection{Acknowledgments}

The author thanks Arkady Berenstein, Kobi Kremnizer, Shahn Majid, Ehud Meir, Vanessa Miemietz, and Chelsea Walton for interesting and helpful discussions related to the subject of this paper. Early parts of this research were supported by an EPSRC Doctoral Prize at the University of East Anglia. The author is supported by the Simons Foundation.

\section{Categorical Setup}

\subsection{The 2-Category of Finitely Cocomplete Categories}\label{setup}

Let $\Bbbk$ be a field. We say that a category $\cC$ is \emph{$\Bbbk$-linear} if it is pre-additive and $\Bbbk$-enriched \cite{EGNO}*{Definition 1.2.2}. We shall work with $\Bbbk$-linear categories that are  equivalent to small categories and in addition have all finite colimits. This assumption implies that finite biproducts exist and hence $\cC$ is additive. The existence of finite colimits can, for example, be ensured if $\cC$, in addition to finite coproducts, has coequalizers. As primary examples, we may consider the category $\lmod{A}$ of $A$-modules, where $A$ is a $\Bbbk$-algebra.

The collection of all such $\Bbbk$-linear categories with finite colimits forms a $2$-category together with $\Bbbk$-linear functors, which will automatically preserve all finite biproducts and are required to preserve all finite colimits, together with natural transformations of such functors. The resulting structure is that of a $\Bbbk$-linear $2$-category, which is denoted by $\Cat^c$. Between two categories $\cC$, $\cD$ of $\Cat^c$, we denote the category of $\Bbbk$-linear finite colimit preserving functors by $\Fun^c(\cC,\cD)$, which again lies in $\Cat^c$ (finite colimits are computed pointwise, and commute with other colimits). In contrast, we will denote the category of all $\Bbbk$-linear functors between such categories (not necessarily preserving finite biproducts or colimits) by $\Fun(\cC,\cD)$ and the corresponding larger $2$-category by $\Cat$.

There is a naive tensor product for $\cC$, $\cD$ denoted by $\otimes$ (cf. \cite{Kel}*{Section 1.4}), giving a $\Bbbk$-linear category $\cC\otimes \cD$. Its objects are denoted by $X\otimes Y$, were $X$ is an object of $\cC$, and $Y$ is an object of $\cD$. The spaces of morphisms are defined as $$\Hom_{\cC\otimes \cD}\left(X\otimes Y,X'\otimes Y'\right)=\Hom_{\cC}(X,X')\otimes_\Bbbk \Hom_{\cC}(Y,Y').$$
Note that in $\cC\otimes \cD$ biproducts are not ensured to exist. A proper cartesian closed structure on $\Cat$ is given by the tensor product $\cC\boxtimes \cD$, see \cite{Kel}*{Section 6.5}. Moreover, if finite colimits exist in $\cC$ and $\cD$, then they exist in $\cC\boxtimes \cD$, giving a cartesian closed structure on $\Cat^c$. It satisfies the universal property of being a finitely cocomplete $\Bbbk$-linear category such that there is a $\Bbbk$-linear functor $\cC\otimes \cD\to \cC\boxtimes \cD$ which is colimit preserving in both components, inducing natural equivalences 
$$\Fun^c(\cC\boxtimes \cD, \cE)\simeq\Fun^c(\cC,\Fun^c(\cD, \cE)),$$
in $\Cat^c$,
for $\cE$ in $\Cat^c$. The tensor product $\cC\boxtimes \cD$ is the closure of $\cC\otimes \cD$ under finite colimits \cite{Kel}*{(6.27)}. Hence, it follows that finite colimit preserving $\Bbbk$-linear functors $\cC\boxtimes \cD\to \cE$ also correspond to $\Bbbk$-linear functors $\cC\otimes \cD\to \cE$ which are finite colimit preserving in both components. Note that if $\cC$ and $\cD$ are abelian categories, $\cC\boxtimes \cD$ is not necessarily abelian. For a clarification of the relationship to Deligne's tensor product of abelian categories, see \cite{LF}. In the abelian case, the finite colimit preserving functors are precisely the right exact functors.

Unless otherwise stated, all categories and functors considered live in the $2$-category $\Cat^c$. We shall use properties of the $2$-category $\Cat^c$ such as its closure under colimits and pseudo-colimits \cites{BKP,Kel3} in order to obtain existence results for a relative version of the tensor products mentioned in this section (see Appendix~\ref{appendixA}). The cocompleteness of $\Cat^c$ is obtained by the construction of a $2$-comonad $T$ on $\Cat$ such that its coalgebras are finitely cocomplete $\Bbbk$-linear categories, with coalgebra morphisms preserving the finite colimits \cite{BKP}. Dually to \cite{BKP}*{Theorem 2.6}, it follows that $\Cat^c$ admits pseudo-colimits, taking the $2$-category ${\sc\mbox{K}}$ appearing therein to be the complete and cocomplete $2$-category of small $\Bbbk$-linear categories $\Cat$ \cite{Kel3}, and $T$ is a $2$-comonad such that $T$-coalgebras are finite cocomplete $\Bbbk$-linear categories.

\subsection{Notations and Conventions}

Throughout this paper, $\cM$ denotes a monoidal category in $\Cat^c$, with tensor product $\otimes=\otimes^{\cM}\colon \cM\boxtimes \cM\to \cM$, associativity isomorphism $\alpha\colon \otimes (\otimes \boxtimes \ide_\cM)\natisomorph \otimes (\ide_\cM \boxtimes \,\otimes)$, and unit $\one$ together with an isomorphism $\iota\colon \one \otimes \one \to \one$ such that left and right tensoring by $\one$ are auto-equivalences of $\cM$, see \cite{EGNO}*{Definition 2.1.1}. 

\begin{assumption}\label{strict} For simplification of the exposition, we will assume that $\cM$ is strictly associative and unital. That is, $(X\otimes Y)\otimes Z=X\otimes (Y\otimes Z)$ and $\alpha_{X,Y,Z}=\ide$ for all objects $X,Y,Z$ of $\cM$. Further $\one\otimes M=M=M\otimes \one$, and left and right unitary isomorphisms are given by identities. 
\end{assumption}

The choice to work with a strict monoidal category is justified by Mac Lane's coherence theorem. That is, whenever two objects constructed by applying the tensor product and tensor unit can be identified via two isomorphisms which are combinations of the natural isomorphisms present in the structure of the monoidal category, then these isomorphisms are equal \cite{Mac}*{Section VII.2}. If needed, one can insert the associativity (or unitary) isomorphisms in the diagrams in the appropriate places and the statements will remain valid. Note further that by Mac Lane's strictness theorem for monoidal categories \cite{Mac}*{Section XI.3}, any monoidal category is equivalent (via strong monoidal functors) to a strict monoidal category.

The symbol $\cB$ will denote a braided monoidal category in $\Cat^c$ with braiding $\Psi$. We also omit the associativity and unitary isomorphisms of $\cB$ from the notation, treating $\cB$ as strict.  The category $\overline{\cB}$ is the same as $\cB$ but declaring the inverse braiding $\Psi^{-1}$ to be the braiding.

All functors of monoidal categories are assumed to be strong monoidal. A strong monoidal functor $\rG\colon \cM \to \cN$ comes with a natural isomorphism $\mu^{\rG}\colon \otimes(\rG\boxtimes \rG)\stackrel{\sim}{\Longrightarrow}\rG\otimes$ satisfying a commutative diagram as in \cite{EGNO}*{(2.23)} relating the associativity isomorphisms of $\cM$ and $\cN$. In addition, we assume that $\rG$ is strictly unital. That is, $\rG(\one)=\one$, and $\mu^{\rG}_{\one, M}$, $\mu^{\rG}_{M,\one}$ are both given by identities.

We summarize here some notational conventions adapted in this paper for monoidal categories.
The symbol $\otimes^{\oop}$ is used to denote the opposite tensor product given by $X\otimes^{\oop}Y=Y\otimes X$. With relation to Kelly's tensor product of categories $\boxtimes$, we denote pure tensors of pairs of objects $X,Y$ by $X\otimes Y$. When brackets are missing, we adapt the convention $\otimes^{\cM}$ before $\boxtimes$. For a morphism $f$ and an object $X$ in $\cM$ we write $f\otimes X$ to denote $f\otimes \ide_X$. Similarly, when considering natural transformations, we may denote by  $\phi\otimes \rG$ the natural transformation $\phi\otimes \ide_{\rG}$.

A main source of examples of monoidal categories will be obtained from a braided bialgebra or Hopf algebra $B$ (see e.g. \cite{Maj1}*{Definition 9.4.5}).  That is, $B$ is a bialgebra or Hopf algebra object in a braided monoidal category $\cB$. More explicitly, $B$ comes equipped with the following morphisms in $\cB$: a product map $m = m_B\colon B\otimes B\to B$, with unit $1\colon I \to B$, and a coproduct map $\Delta=\Delta_B\colon B\to B\otimes B$, with counit $\varepsilon\colon B\to I$, satisfying the axioms (\ref{bialgebraaxiom1})--(\ref{bialgebraaxiomlast}).
\begin{gather}\label{bialgebraaxiom1}
m(m\otimes \ide_B) = m (\ide_B\otimes m ),\\
m( 1\otimes \ide_B) = m(\ide_B\otimes 1)=\ide_B,\\
(\Delta \otimes \ide_B) \Delta = (\ide_B\otimes \Delta )\Delta ,\\
(\varepsilon\otimes \ide_B)\Delta  =(\ide \otimes \varepsilon)\Delta=\ide,\\
\Delta m = m\otimes m(\ide_B\otimes \Psi_{B,B}\otimes \ide_B)(\Delta\otimes \Delta),\label{bialgebraaxiom}\\
\Delta 1 = 1\otimes 1,\\
\varepsilon m = \varepsilon\otimes \varepsilon.\label{bialgebraaxiomlast}
\end{gather}
The first two axioms state that $B$ is algebra algebra object in $\cB$, and the next two axioms state that $B$ is also a coalgebra object in $\cB$.
The last three axioms may be rephrased by saying that $\Delta$, and  $\varepsilon$ are morphisms of algebras in $\cB$, or, equivalently, that $m$, $1$ are morphisms of coalgebras in $\cB$. Morphisms of bialgebras are morphisms in $\cB$ which commute with all of these structural maps. If, in addition, there exists an invertible morphism $S\colon B\to B$ in $\cB$ which is a two-sided convolution inverse to $\ide_B$, i.e.
\begin{align}
m(\ide_B\otimes S)\Delta&=m(S\otimes \ide_B)\Delta=1\varepsilon,
\end{align}
then we say $B$ is a \emph{Hopf algebra} in $\cB$. Morphisms of Hopf algebras are required to commute with this antipode $S$. It is a consequence that $S$ is an anti-algebra morphism in $\cB$ (see \cite{Maj1}*{Fig. 9.14} for a proof).

We denote the category of algebra objects in $\cB$ by $\Alg(\cB)$, and write $\CoAlg(\cB)$ for coalgebras in $\cB$, $\BiAlg(\cB)$ for bialgebras, and $\Hopf(\cB)$ for Hopf algebras. Note that while $\Alg(\cB)$ and $\CoAlg(\cB)$ are monoidal categories with the underlying tensor product in $\cB$ (under use of the braiding), $\BiAlg(\cB)$ and $\Hopf(\cB)$ are not. Given an algebra $A$ in $\cB$, we denote the category of left $A$-modules by $\lmod{A}(\cB)$, and right $A$-modules by $\rmod{A}(\cB)$. Similarly, given a coalgebra $C$, we write $\lcomod{C}(\cB)$ and $\rcomod{C}(\cB)$ for left and right comodules over $C$ in $\cB$.


\section{Balanced Bimodules and the Relative Monoidal Center}\label{catsect}

This sections contains this paper's main constructions on the level of monoidal categories and their (bi)modules. The main object is the \emph{relative monoidal center} $\cZ_\cB(\cM)$ which is defined here using the concept of $\cB$-balanced bimodules, refining the construction from \cite{Lau}. The Morita dual to this center construction is the monoidal category $\cM\boxtimes_{\cB}\,\cM^{\oop}$, for which an existence statement is provided. All constructions are carried out in the $2$-category $\Cat^c$.

\subsection{Categorical Modules}\label{catmodules}

Following \cite{EGNO}*{Definition 7.1.3}, a \emph{left module} $\cV$ over a monoidal category $\cM$ is a category $\cV$ together with a monoidal functor
\begin{align*}
\triangleright\colon \cM\longrightarrow \Ends^c(\cV),
\end{align*}
where $\Ends^c(\cV)$ is the monoidal category of endofunctors on $\cV$ in $\Cat^c$ with respect to composition. 
It is often helpful to spell out the definition of a categorical module in terms of a functor $\triangleright \colon \cM \boxtimes \cV \longrightarrow \cV$ together with a natural isomorphism $\chi \colon \triangleright (\otimes \boxtimes \ide_\cV)\natisomorph \triangleright(\ide_\cM \boxtimes \triangleright)$ satisfying compatibility conditions with the monoidal structure of $\cM$. Recalling that we shall treat $\cM$ as a \emph{strict monoidal category}, cf. Assumption 
\ref{strict}, this amounts to the conditions that $\one \triangleright V=V$ for any object $V$ of $\cV$, and the diagram
\begin{equation}\label{modulecoherence}
\vcenter{\hbox{\xymatrix{
(M\otimes N\otimes P)\triangleright V\ar[rr]^{\chi_{M\otimes N, P,V}}\ar[d]^{\chi_{M,N\otimes P,V}}&& (M\otimes N)\triangleright (P\triangleright V)\ar[d]^{\chi_{M,N,P\triangleright V}}\\
M\triangleright ((N\otimes P)\triangleright V))\ar[rr]^{M\triangleright \chi_{N,P,V}}&&M\triangleright(N\triangleright(P\triangleright V))
}}}
\end{equation}
commutes for any choice of objects $M,N,P\in \cM$ and $V\in\cV$ (cf. \cite{EGNO}*{Definition 7.1.2}).
Further, note that $\chi_{M,\one,V}=\chi_{\one,M,V}=\ide_{M\triangleright V}$.

A \emph{morphism} $(\rG,\lambda)\colon (\cV,\triangleright_\cV)\longrightarrow (\cW,\triangleright_\cW)$ of left modules over $\cM$ is a functor $\rG\colon \cV\to\cW$ together with a system of natural isomorphisms
\begin{align*}
\lambda_X \colon  \rG(X\triangleright_{\cV}\ide_\cW)&\natisomorph X\triangleright_{\cW}\rG.
\end{align*}
which is natural in $X$ and compatible with the monoidal structure in $\cM$ as in \cite{EGNO}*{Eq. (7.6)}. We write $\lambda_{X,V}=(\lambda_X)_V$. The compatibility conditions amount to  $\lambda_{\one,V}=\ide_{\rF(V)}$, and that the diagram
\begin{align}\label{modulemorph1}
\vcenter{\hbox{
\xymatrix{
&\rG(M\otimes N\triangleright_{\cV}V)\ar[ld]_{\rG(\chi^{\cV}_{M,N,V})}\ar[rd]^{\lambda_{M\otimes N,V}}&\\
\rG(M\triangleright_{\cV}(N\triangleright_{\cV}V))\ar[d]_{\lambda_{M,N\triangleright_{\cV} V}}&&M\otimes N\triangleright_{\cW}\rG(V)\ar[d]^{\chi^{\cW}_{M,N,\rG(V)}}\\
M\triangleright_{\cW}\rG(N\triangleright_{\cV}V)\ar[rr]^{M \triangleright_{\cW} \lambda_{N,V}}&&M\triangleright_{\cW}( N\triangleright_{\cW}\rG (V ))
}}}
\end{align}
commutes for any objects $M,N$ in $\cM$, and $V$ in $\cV$.

In other words, the coherent  isomorphism $\lambda$ defines a weakly commuting diagram of functors
\begin{align*}
\xymatrix{\cM\ar[rr]^{\triangleright_{\cV}}\ar[d]^{\triangleright_{\cW}} &&\Ends^c(\cV)\ar[d]^{\rG(-)}\ar@{=>}[lld]^{\lambda}\\
\Ends^c(\cW)\ar[rr]_{(-)\rG} &&\Fun^c(\cV,\cW)\,.
}
\end{align*}

Given two morphisms $\rG, \rH\colon \cV\to \cW$ or $\cM$-modules, we can define a $2$-morphism $\tau\colon\rG\Longrightarrow\rH$ to be a natural transformation such that
\begin{equation}
(\ide_{\cM}\triangleright_{\cW} \tau)\lambda^{\rG}=\lambda^{\rH}\tau_{\triangleright_{\cV}},
\end{equation}
i.e. the diagram
\begin{align}\label{2morphism}
\vcenter{\hbox{
\xymatrix{
\rG(M\triangleright_{\cV} V)\ar[rr]^-{\lambda^{\rG}_{M,V}}\ar[d]_{\tau_{M\triangleright_{\cV} V}}&&M\triangleright_{\cW} \rG(V)\ar[d]^{M\triangleright_{\cW} \tau_{V}}\\
\rH(M\triangleright_{\cV} V)\ar[rr]^-{\lambda^{\rH}_{M,V}}&&M\triangleright_{\cW} \rH(V)
}}}
\end{align}
commutes for any objects $M$ of $\cM$ and $V$ of $\cV$. Hence, we obtain a $\Bbbk$-linear $2$-category $\lmod{\cM}$.
Indeed, given two morphisms of left $\cM$-modules $\rG\colon \cV\to \cW$, $\rH\colon \cW\to \cX$, the composition $\rH\rG$ is a morphism of $\cM$-modules with $\lambda^{\rH\rG}$ defined by
\begin{align}\label{lambdacompose}
\lambda^{\rH\rG}_{M,V}&=\lambda_{M,\rG(V)}^{\rH}\rH(\lambda_{M,V}^{\rG}).
\end{align}
The category of $\cM$-module morphisms from $\cV$ to $\cW$, denoted by $\cHom_{\cM}(\cV,\cW)$, contains finite colimits and hence $\lmod{\cM}$ is enriched over $\Cat^c$.

\begin{remark}
This information can be rephrased as follows: A left module $\cV$ over $\cM$ is a $\Bbbk$-linear bifunctor
$
\triangleright_{\cV} \colon \cM\longrightarrow \Cat^c,
$ where $\cM$ is viewed as a bicategory with one object, and $\cV$ is the image of this object. A morphism of left modules over $\cM$ is now a morphism of bifunctors $\triangleright_{\cV}\Longrightarrow \triangleright_\cW$.
This way, $2$-morphisms of $\cM$-modules correspond to modifications of morphisms of bifunctors.
\end{remark}

From the monoidal functor point of view, a \emph{right} $\cM$-module structure on $\cV$ is a monoidal functor $\cM^{\oop}\longrightarrow \Ends^c(\cV)$, and a $\cM$-\emph{bimodule} is a monoidal functor $\cM\boxtimes \cM^{\oop}\longrightarrow \Ends^c(\cV)$. More generally, a $\cM$-$\cN$-bimodule is a monoidal functor $\cM\boxtimes \cN^{\oop}\longrightarrow \Ends^c(\cV)$. We may equivalently express the data of a bimodule category as follows (cf.  \cite{FSS}*{Definition 2.5}):

\begin{proposition}\label{bimoduledata}
An $\cM$-$\cN$-bimodule structure on $\cV$ is equivalent to the data of two functors
\begin{align*}
\triangleright &\colon \cM\boxtimes \cV\longrightarrow \cV,& \triangleleft & \colon \cV\boxtimes \cN\longrightarrow \cV,
\end{align*}
together with natural isomorphisms
\begin{gather*}
\chi\colon \triangleright (\otimes \boxtimes\ide_\cV) \natisomorph\triangleright (\ide_\cM\boxtimes ~\triangleright),\qquad
\xi \colon \triangleleft(\ide_\cV\boxtimes ~\otimes)\natisomorph \triangleleft(\triangleleft \boxtimes \ide_\cN),\\
\zeta\colon \triangleleft(\triangleright \boxtimes \ide_\cN)\natisomorph \triangleright (\ide_\cM\boxtimes ~\triangleleft).
\end{gather*}
These natural isomorphisms satisfy  the coherence condition given by Eq. (\ref{modulecoherence}), as well as the following Eqs. (\ref{modulecoherence2})--(\ref{bimodcoherence2}) below. 
\end{proposition}
\begin{gather}\label{modulecoherence2}
\vcenter{\hbox{\xymatrix{
V\triangleleft (M\otimes N\otimes P)\ar[rr]^{\xi_{V,M,N\otimes P}}\ar[d]^{\xi_{V,M\otimes N,P}}&& (V\triangleleft M)\triangleleft(N\otimes P) \ar[d]^{\xi_{V\triangleleft M,N,P}}\\
(V\triangleleft (M\otimes N))\triangleleft P\ar[rr]^{\xi_{V,M,N}\triangleleft P}&&((V\triangleleft M)\triangleleft N)\triangleleft P,
}}}
\end{gather}
\begin{gather}
\label{bimodcoherence1}
\vcenter{\hbox{
\xymatrix{
&(X\otimes Y\triangleright V)\triangleleft M\ar[ld]_{\chi_{X,Y,V}\triangleleft M}\ar[rd]^{\zeta_{X\otimes Y,V,M}}&\\
(X\triangleright (Y\triangleright V))\triangleleft M\ar[d]^{\zeta_{X,Y\triangleright V,M}}&&X\otimes Y \triangleright (V\triangleleft M)\ar[d]^{\chi_{X,Y,V\triangleleft M}}\\
X\triangleright ((Y\triangleright V)\triangleleft M)\ar[rr]^{X\triangleright \zeta_{Y,V,M}}&&X\triangleright (Y \triangleright (V\triangleleft M)),
}}}
\end{gather}
\begin{gather}
\label{bimodcoherence2}
\vcenter{\hbox{
\xymatrix{
&X\triangleright (V\triangleleft M\otimes N)\ar[ld]_{X\triangleright \xi_{V,M,N}}\ar[rd]^{\zeta^{-1}_{X,V,M\otimes N}}&\\
X\triangleright ((V\triangleleft M)\triangleleft N)  \ar[d]^{\zeta^{-1}_{X,V\triangleleft M,N}}&&(X \triangleright V)\triangleleft M\otimes N\ar[d]^{\xi_{X\triangleright V,M,N}}\\
(X\triangleright (V\triangleleft M))\triangleleft N\ar[rr]^{\zeta^{-1}_{X,V,M}\triangleleft N}&&((X\triangleright V) \triangleleft M)\triangleleft N,
}}}
\end{gather}
for objects $X,Y$ of $\cM$, $V$ of $\cV$, and $M,N,P$ of $\cN$.
In addition, we require $\one \triangleright V=V$, $V\triangleleft \one =V$, and that
\begin{align}
\chi_{\one,X,V}&=\chi_{X,\one,V}=\ide_{X\triangleright V}, &
\xi_{V,\one,N}&=\xi_{V,N,\one}=\ide_{V\triangleleft N}, \\
\zeta_{X, V,\one}&=\ide_{X\triangleright V}, &\zeta_{\one, V,N}&=\ide_{V\triangleleft N},
\end{align}
for objects $X$ of $\cM$, $N$ of $\cN$, and $V$ of $\cV$.

\begin{definition}
We define the $2$-category $\BiMod_{\cM\text{--}\cN}$ as the $2$-category of $\cM$-$\cN$-bimodules  in $\Cat^c$ using the equivalent description of the data from Proposition \ref{bimoduledata}. This amounts to having as objects $\cM$-$\cN$-bimodules, $1$-morphisms in $\BiMod_{\cM\text{--}\cN}$ are given by functors $\rG\colon \cV\to\cW$ together with natural isomorphisms
\begin{align*}
\lambda \colon  \rG(X\triangleright_{\cV}\ide_\cV)&\natisomorph X\triangleright_{\cW}\rG,\\
\rho \colon  \rG(\ide_\cV\triangleleft_{\cV}N)&\natisomorph \rG\triangleleft_{\cW}N,
\end{align*}
for objects $X$ of $\cM$, $N$ of $\cN$, which are coherent with the structural isomorphisms of the modules. That is, diagram (\ref{modulemorph1}) commutes, in addition to a similar diagram for $\rho$ and $\xi$:
\begin{align}\label{modulemorph2}
\vcenter{\hbox{
\xymatrix{
&\rG(V\triangleleft_{\cV}M\otimes N)\ar[ld]_{\rG(\xi^{\cV}_{V,M,N})}\ar[rd]^{\rho_{V,M\otimes N}}&\\
\rG((V\triangleleft_{\cV}M)\triangleleft_{\cV} N)\ar[d]_{\rho_{V\triangleleft_{\cV}M,N}}&&\rG(V)\triangleleft_{\cW}M\otimes N\ar[d]^{\xi^{\cW}_{\rG(V),M,N}}\\
\rG(V\triangleleft_{\cV}M)\triangleleft_{\cW}N\ar[rr]^{\rho_{V,M}\triangleleft_{\cW} N}&&(\rG(V)\triangleleft_{\cW}M)\triangleleft_{\cW}N,
}}}
\end{align}
for objects $M,N$ of $\cN$; further, the diagram
\begin{align}\label{modulemorph3}
\vcenter{\hbox{
\xymatrix{
\rG((X\triangleright_{\cV} V)\triangleleft_{\cV} N)\ar[rr]^{\rG(\zeta_{X,V,N}^{\cV})}\ar[d]^{\rho_{X\triangleright_{\cV} V,N}}&&\rG(X\triangleright_{\cV} (V\triangleleft_{\cV} N))\ar[d]^{\lambda_{X,V\triangleleft_{\cV} N}}\\
\rG(X\triangleright_{\cV} V)\ar[d]^{\lambda_{X,V}\triangleleft_{\cW} N}\triangleleft_{\cW} N&&X\triangleright_{\cW} \rG(V\triangleleft_{\cV} N)\ar[d]^{X\triangleright_{\cW} \rho_{V, N}}\\
(X\triangleright_{\cW} \rG(V))\triangleleft_{\cW} N\ar[rr]^{\zeta_{X,\rG(V),N}^{\cW}}&&X\triangleright _{\cW}(\rG(V)\triangleleft_{\cW} N),
}}}
\end{align}
is required to commute for objects $X$ of $\cM$, $M,N$ of $\cN$, $V$ of $\cV$. Moreover, we require that 
\begin{align}
\lambda_{\one, V}=\ide_{\rG(V)}, && \rho_{V,\one}=\ide_{\rG(V)},&&\forall V\in \cV.
\end{align}

Note that given two such morphisms $(\rG,\lambda^{\rG}, \rho^{\rG})\colon \cV\to \cW$ and $(\rH,\lambda^{\rH}, \rho^{\rH})\colon \cW\to \cX$ the composition is $(\rH\rG,\lambda^{\rH\rG}, \rho^{\rH\rG})$ with $\lambda^{\rH\rG}$ defined in Eq. (\ref{lambdacompose}) and
\begin{align}\label{rhocompose}
\rho^{\rH\rG}_{V,N}&=\rho^{\rH}_{\rG(V),N}\rH(\rho_{V,N}^{\rG}).
\end{align}
In $\BiMod_{\cM\text{--}\cN}$, a $2$-morphism $\tau\colon (\rG,\lambda^{\rG},\rho^{\rG}) \Longrightarrow (\rH,\lambda^{\rH},\rho^{\rH})$ is a natural transformations commuting with this structure. That is, diagram (\ref{2morphism}) and the diagram
\begin{align}\label{2morphism2}
\vcenter{\hbox{
\xymatrix{
\rG(V\triangleleft_{\cV} N)\ar[rr]^-{\rho^{\rG}_{V,N}}\ar[d]_{\tau_{V\triangleleft_{\cV} N}}&&\rG(V)\triangleleft_{\cW} N\ar[d]^{\tau_V\triangleleft_{\cW}N}\\
\rH(V\triangleleft_{\cV} N)\ar[rr]^-{\rho^{\rH}_{V,N}}&&\rH(V)\triangleleft_{\cW} N
}}}
\end{align}
commute for objects $N$ of $\cN$, $V$ of $\cV$. Given two $\cM$-$\cN$-bimodules, we denote the category of $\cM$-$\cN$-bimodule morphism from $\cV$ to $\cW$ by $\cHom_{\cM\text{--}\cN}(\cV,\cW)$.
\end{definition}
The discussion can be summarized by saying that there is a biequivalence of $\Bbbk$-linear $2$-categories
$\BiMod_{\cM\text{--}\cN}\simeq \lmod{\cM\boxtimes\cN^{\oop}}.$


\subsection{Balanced Functors and Relative Tensor Products}\label{reltensorsect}

We recall the following definition of balanced functors from \cite{ENO}, \cite{FSS}:

\begin{definition}\label{reltensordef}
Let $\cV$ be a right $\cC$-module and $\cW$ a left $\cC$-module category. A functor
$$\rF\colon \cV\boxtimes \cW\to \cT, \qquad V\otimes W\longmapsto \rF(V, W)$$
is called \emph{$\cC$-balanced} if it comes equipped with a system of isomorphisms 
\begin{align*}
\eta_{V,C,W}\colon \rF(V\triangleleft C, W)\stackrel{\sim}{\longrightarrow}\rF(V,C\triangleright W),&&\forall C\in \cC, V\in \cV, W\in \cW,
\end{align*}
natural in $V,C,W$,
which is compatible with the natural isomorphisms of the module structures (as in \cite{ENO}*{Definition 3.1}). That is, the diagram
\begin{align}\label{balancedcoh}
\vcenter{\hbox{
\xymatrix{
\rF(V\triangleleft (C\otimes D),W)\ar[rr]^{\rF(\xi_{V,C,D}, W)}\ar[d]_{\eta_{V, C\otimes D, W}}&&\rF((V\triangleleft C)\triangleleft D,W)\ar[d]^{\eta_{V\triangleleft C,D, W}}\\
\rF(V,(C\otimes D)\triangleright W)\ar[rd]_{\rF(V, \chi_{C,D,W})}&&\rF(V\triangleleft C, D\triangleright W)\ar[ld]^{\eta_{V,C,D\triangleright W}}\\
&\rF(V,C\triangleright(D\triangleright W))&
}}}
\end{align}
commutes for any objects $V$ in $\cV$, $W$ in $\cW$, $C,D$ in $\cC$. Further, we require
\begin{align}\label{balancedcohone}
\eta_{V,\one,W}=\ide_{\rF(V,W)}.
\end{align}

A \emph{morphism of $\cC$-balanced functors} $\phi\colon (\rF, \eta^{\rF})\to (\rG,\eta^{\rG})$ is a natural transformation $\phi\colon \rF\Rightarrow \rG$ such that the diagram
\begin{align}\label{balancedmor}
\vcenter{\hbox{
\xymatrix{\rF(V\triangleleft C, W)\ar^-{\eta^{\rF}_{V,C,W}}[rr]\ar_-{\phi_{V\triangleleft C, W}}[d]&&\rF(V,C\triangleright W)\ar^-{\phi_{V, C\triangleright W}}[d]\\
\rG(V\triangleleft C, W)\ar^-{\eta^{\rG}_{V,C,W}}[rr]&&\rG(V,C\triangleright W)
}}}
\end{align}
commutes. The category of $\cC$-balanced functors is denoted by
${\Fun}^{\bal{\cC}}(\cV\boxtimes \cW, \cT)$ where, again, all categories and functors are in $\Cat^c$.
\end{definition}

The following definition is a adapted from \cite{Scha}*{Definition 3.1} to working in $\Cat^c$.

\begin{definition}\label{relativetensordef}
Let $\cC$ be a monoidal category with a right $\cC$-module $\cV$, and a left $\cC$-module $\cW$ in $\Cat^c$. A category $\cT$ in $\Cat^c$, together with a $\cC$-balanced functor $\rT_{\cV,\cW}\colon \cV\boxtimes \cW\to \cT$, is a \emph{relative tensor product} of $\cV$ and $\cW$ over $\cC$ if for any category $\cD$, $\rT_{\cV,\cW}$ induces an equivalence of $\Bbbk$-linear categories
\begin{align*}
\Psi_{\cV,\cW}\colon {\Fun^c}(\cT, \cD)\isomorph \Fun^{\bal{\cC}}(\cV\boxtimes \cW, \cD), && \rG\mapsto \rG\rT_{\cV,\cW}.
\end{align*}
Further, part of the data is a fixed choice of a $\Bbbk$-linear functor
\[
\Phi_{\cV,\cW}\colon {\Fun^c}(\cT, \cD)\isomorph \Fun^{\bal{\cC}}(\cV\boxtimes \cW, \cD)
\]
such that there are natural isomorphisms
\begin{align*}
\theta_{\cV,\cW}\colon \ide \natisomorph \Phi_{\cV,\cW}\Psi_{\cV,\cW}, &&\tau_{\cV,\cW}\colon \Psi_{\cV,\cW}\Phi_{\cV,\cW}\natisomorph \ide
\end{align*}
satisfying the adjunction axioms (so that $\Psi_{\cV,\cW},\Phi_{\cV,\cW}$ form an \emph{adjoint equivalence} of $\Bbbk$-linear categories).

It follows that if such a category $\cT$ exists, then it is unique, up to unique adjoint equivalence of categories. We hence use the notation $\cT=\cV\boxtimes_\cC \cW$. 
\end{definition}

It follows from \cite{Scha}*{Proposition 3.4} that the tensor product $\boxtimes_{\cC}$ (if existent in this generality) has the following naturality properties. Given a functor $\rF\colon \cV\to \cV'$ of right $\cC$-module categories, and  $\rG\colon \cW\to \cW'$ of left $\cC$-module categories, there exists an induced factorization functor $$\rF\boxtimes_{\cC}\rG \colon \cV\boxtimes_{\cC}\cW\longrightarrow \cV'\boxtimes_{\cC}\cW',$$
which is unique up to unique isomorphism of functors.

This is proved by first observing that the composition functor
$$\rT_{\cV',\cW'} (\rF\boxtimes \rG)\colon \cV\boxtimes \cW\xrightarrow{\rF\boxtimes \rG} \cV'\boxtimes \cW'\xrightarrow{\rT_{\cV',\cW'}} \cV'\boxtimes_{\cC} \cW'$$
is $\cC$-balanced, using the $\cC$-balancing isomorphism determined by $\beta_{\rF(V),C,\rG(W)}$ for $\beta$ the $\cC$-balancing isomorphism of $\rT'$. Further, \cite{Scha}*{Proposition 3.4 (1)} gives an isomorphism of morphisms of balanced bimodules   $\phi_{\rF,\rG}\colon \rT_{\cV',\cW'}(\rF\boxtimes\rG)\natisomorph(\rF\boxtimes_{\cC}\rG) \rT_{\cV,\cW}$. 

Further, assume given pairs of right (respectively, left) $\cC$-module functors $\rF,\rF'\colon \cV\to \cV'$ (respectively,  $\rG, \rG'\colon \cW\to \cW'$) and natural transformations 
$\alpha\colon \rF\Longrightarrow \rF'$, $\beta \colon \rG\Longrightarrow \rG'$ of such functors, there exists a \emph{unique} natural transformation
$$\alpha\boxtimes_{\cC}\beta\colon \rF\boxtimes_{\cC}\rG \Longrightarrow \rF'\boxtimes_{\cC}\rG'$$
such that the following diagram commutes:
\begin{align}
\vcenter{\hbox{
\xymatrix{
\rT_{\cV',\cW'}(\rF\boxtimes\rG)\ar@{=>}[d]^{\phi_{\rF,\rG}}\ar@{=>}[rr]^{\alpha\boxtimes\beta}&& \rT_{\cV',\cW'}(\rF'\boxtimes\rG')\ar@{=>}[d]^{\phi_{\rF',\rG'}}\\
(\rF\boxtimes_{\cC}\rG)\rT_{\cV,\cW}\ar@{=>}[rr]^{\alpha\boxtimes_{\cC}\beta}&& (\rF'\boxtimes_{\cC}\rG')\rT_{\cV,\cW}.
}}}
\end{align}
These assignments are functorial to give a bifunctor
\begin{align*}
\boxtimes_{\cC}\colon \rmod{\cC}\times\lmod{\cC}\longrightarrow \Cat^c,&& (\cV,\cW)\mapsto \cV\boxtimes_{\cC}\cW.
\end{align*}
The following existence statement is proved in Appendix \ref{appendixA}.

\begin{theorem}\label{reltensorexistence}
Let $\cC$ be a monoidal category with a right $\cC$-module $\cV$ and a left $\cC$-module $\cW$ in $\Cat^c$. Then $\cV\boxtimes_\cC \cW$ exists in $\Cat^c$.
\end{theorem}

\begin{remark}
The existence statement of $\cV\boxtimes_{\cC}\cW$ can, in principle, be proved in any cocomplete $2$-category which has finite pseudo-colimits. For the $2$-category $\Cat^c$ of cocomplete $\Bbbk$-linear categories, pseudo-colimits exist by \cite{BKP}.\end{remark}

\begin{corollary}\label{bimoduletensor}
If $\cV$ is an $\cM$-$\cB$-bimodule, and $\cW$ is a $\cB$-$\cN$-bimodule, then $\cV\boxtimes_{\cB}\cW$ can be given the structure of an $\cM$-$\cN$-bimodule.
\end{corollary}
\begin{proof}
Based on the existence statement of $\cV\boxtimes_{\cC}\cW$ in this setup, we can now use \cite{Scha}*{Proposition 3.20} to obtain the result. Note that the full structure of the collection of categorical bimodules with relative tensor products is that of a tricategory as detailed in \cite{Scha}*{Section 3.2}.
\end{proof}

We now provide some examples that will be important in the paper.

\begin{definition}
Let $A$ and $B$ be algebra objects in $\cB$. We denote by $\llmod{A}{B}(\cB)$ the category of left $A$-$B$ modules in $\cB$. That is, an object $V$ of $\llmod{A}{B}(\cB)$ is an object of $\cB$ together with a left $A$-action $\triangleright_A\colon A\otimes V\to V$ and left $B$-action $\triangleright_B\colon B \otimes V\to V$ which commute in the sense that
\begin{equation}
\triangleright_A(\ide_A\otimes \triangleright_B)=\triangleright_B(\ide_B\otimes \,\triangleright_A)(\Psi_{A,B}\otimes \ide_V).
\end{equation}
A morphism in $\llmod{A}{B}(\cB)$ commutes with both the left $A$-action and $B$-action. 

Similarly, we define the category $\lrmod{A}{B}(\cB)$ and $\rrmod{A}{B}(\cB)$. In the former, we have a left $A$-action and a right $B$-action which commute, where in the later both are right actions.
\end{definition}

It in fact follows that $\llmod{A}{B}(\cB)$ is equivalent to $\lmod{A\otimes B}(\cB)$ in a natural way. However, when  working with bialgebras we shall see that $\llmod{A}{B}(\cB)$ can still be monoidal even though $A\otimes B$ is not a bialgebra object in $\cB$. 

\begin{lemma}
The category $\llmod{A}{B}(\cB)$ (or $\lrmod{A}{B}(\cB)$, $\rrmod{A}{B}(\cB)$) is a bimodule over $\cB$, with left and right action given by the induced action on the tensor product in $\cB$. That is, given objects $(V,\triangleright_A,\triangleright_B)$ of $\llmod{A}{B}(\cB)$ and $X$ of $\cB$, $V\otimes X$ is an object in $\llmod{A}{B}(\cB)$ with actions given by
\begin{align*}
\triangleright_{A}\otimes \ide_X\colon A\otimes V\otimes X\longrightarrow V\otimes X, && \triangleright_{B}\otimes \ide_X\colon B\otimes V\otimes X\longrightarrow V\otimes X;
\end{align*}
and $X\otimes V$ is an object in $\llmod{A}{B}(\cB)$ with actions given by
\begin{align*}
(\ide_X\otimes \,\triangleright_{A})(\Psi_{A,X}\otimes \ide_V)&\colon A\otimes X\otimes V\longrightarrow X\otimes V, \\(\ide_X\otimes \,\triangleright_{B})(\Psi_{B,X}\otimes \ide_V)&\colon B\otimes X\otimes V\longrightarrow X\otimes V.
\end{align*}
\end{lemma}

\begin{proposition}\label{Bmodreltensor}
The category $\llmod{A}{B}(\cB)$ satisfies the universal property of $\lmod{A}(\cB)\boxtimes_{\cB}\lmod{B}(\cB)$, $\lrmod{A}{B}(\cB)$ satisfies the universal property of $\lmod{A}(\cB)\boxtimes_{\cB}\rmod{B}(\cB)$, and $\rrmod{A}{B}(\cB)$ satisfies the universal property of $\rmod{A}(\cB)\boxtimes_{\cB}\rmod{B}(\cB)$ as $\cB$-bimodules.
\end{proposition}
\begin{proof}
We will prove the statement for $\lrmod{A}{B}(\cB)$. For the other statements, note that the equivalence $\lmod{B}(\cB)\simeq \rmod{B^{\oop}}(\cB)$, where $B^{\oop}$ has the opposite product $m\Psi$, can be used. The equivalence is given by mapping a left $B$-module  $(V,\triangleright_B)$ to the right $B^{\oop}$-module $(V, \triangleright_B\Psi)$.

Consider the functor $\rT\colon \lmod{A}(\cB)\boxtimes_{\cB}\rmod{B}(\cB)\longrightarrow\lrmod{A}{B}(\cB)$ which sends $V\otimes W$ to $V\otimes^{\cB} W$, taking tensor products $\otimes^{\cB}$ in $\cB$. The left $A$-action is induced on the first tensorand, while the right $B$-action is induced on the second tensorand. Given a $\cB$-balanced functor $\rG\colon \lmod{A}(\cB)\boxtimes_{\cB}\rmod{B}(\cB)\to \cT$, we obtain a factorization $\rH\colon \lrmod{A}{B}(\cB)\to \cT$ by adapting an argument similar to \cite{DSS}*{Theorem 3.3}. Note that the argument used there does not need the functor $\rG$ to be exact, but rather to preserve colimits in both components. Let $X$ be an object in $\lrmod{A}{B}(\cB)$. Then $X$ is the coequalizer of the diagram
$$\xymatrix{
(X\otimes B)\otimes B\cong X\otimes (B\otimes B)\ar@/^/[rr]^-{\triangleright_{B}\otimes \ide_B}\ar@/_/[rr]_-{\ide_X\otimes m_B}&&X\otimes B,
}
$$
which consists of morphisms of left $A$-modules and right $B$-modules. This diagram exists in $\lmod{A}(\cB)\boxtimes \rmod{B}(\cB)$. That is, the (non-commuting) diagram
\begin{align}\label{preimagediag}
\xymatrix{
(X\triangleleft B)\boxtimes B\ar[d]_{\eta_{X,B,B}}\ar[rr]^{\triangleright_B\boxtimes \ide}&&X\boxtimes B\ar@{=}[d]\\
X\boxtimes (B\triangleright B)\ar[rr]^{\ide\boxtimes m_B}&& X\boxtimes B\\
}
\end{align}
is mapped to it under the canonical functor $\rT$ from above. Hence, we can define $\rH(X)$ to be the coequalizer of $\rT$ applied to the compositions in diagram (\ref{preimagediag}). As colimits commute with colimits, the functor $\rH$ obtained this way will preserve finite colimits.
\end{proof}

When more structure is present (for example, an abelian structure), the existence of a universal category describing $\cC$-balanced functors is not guaranteed as a category having such structure. For a finite $\Bbbk$-tensor category $\cC$, it was shown in \cite{DSS} that $\cV\boxtimes_{\cC}\cW$ exists as a finite $\Bbbk$-linear category, given that $\cV$ and $\cW$ are finite. This uses the characterization from \cite{EGNO}*{2.11.6} (and \cite{Ost2}*{Theorem~1} for the semisimple case) that finite exact modules are of the form $\lmod{A}(\cC)$ as right $\cC$-modules, for an algebra object $A$ in $\cC$.

\subsection{Augmented Monoidal Categories}\label{augmentedcats}

We want to work with monoidal categories that are augmented by a braided monoidal category $\cB$. This notion  is a categorical analogue of a $C$-algebra over a commutative ring $C$ which is also $C$-augmented.

\begin{definition}\label{augmentedmon}
A \emph{$\cB$-augmented monoidal category $\cM$} is a monoidal category $\cM$ in $\Cat^c$ equipped with the additional data of monoidal functors
$$\adj{\rF}{\cM}{\cB}{\rT}
$$
and natural isomorphisms
\begin{align*}
\tau\colon \rF\rT&\stackrel{\sim}{\Longrightarrow} \ide_{\cB},& 
\sigma\colon \otimes (\ide_{\cM}\boxtimes \rT)&\stackrel{\sim}{\Longrightarrow} \otimes^{\oop} (\ide_{\cM}\boxtimes \rT),
\end{align*}
such that for any objects $V$ of $\cB$ and $X$ of $\cM$,
\begin{equation}\label{sigmadesent}
\rF(\sigma_{X,V})=\mu_{\rT(V),X}^{\rF}(\tau_V^{-1}\otimes \ide_{\rF(X)})\Psi_{\rF(X),V}(\ide_{\rF(X)}\otimes \tau_V)\left(\mu_{X,\rT(V)}^{\rF}\right)^{-1},
\end{equation}
or, equivalently
\begin{equation}\label{sigmadesent2}
\rF(\sigma_{X,V})=\mu_{\rT(V),X}^{\rF}\Psi_{\rF(X),\rF\rT(V)}\left(\mu_{X,\rT(V)}^{\rF}\right)^{-1}.
\end{equation}
That is, $\sigma$ descents to $\Psi$ under $\rF$. Equivalently, the following diagrams commute for any objects $X$ or $\cM$ and $V$ of $\cB$
\begin{align}\label{augmentedbraiding}\vcenter{\hbox{
\xymatrix{\rF(X\otimes \rT V)\ar[rrr]^{\rF(\sigma_{X,V})}\ar[d]_{\left(\mu_{X,\rT V}^{\rF}\right)^{-1}}&&&\rF(\rT V \otimes X)\ar[d]^{\left(\mu_{\rT V ,X}^{\rF}\right)^{-1}}\\
\rF X\otimes \rF\rT V\ar[d]_{\ide_{\rF X}\otimes \tau_V}\ar[rrr]^{\Psi_{\rF X ,\rF\rT V }}&&&\rF \rT V\otimes \rF X\ar[d]^{\tau_V\otimes \ide_{\rF X }}\\
\rF X\otimes V\ar[rrr]^{\Psi_{\rF X,V}}&&&V\otimes \rF X.
}}}
\end{align}
The natural isomorphisms are required to be coherent with the structure of $\cM$ and $\cB$.
That is, the following diagrams commute for any objects $X,Y$ in $\cM$, $V,W$ in $\cB$
\begin{gather}\label{augmenteddiag1}
\vcenter{\hbox{
\xymatrix{
(X\otimes Y)\otimes \rT V \ar[rr]^{\sigma_{X\otimes Y,V}}\ar[rd]_{\ide_X\otimes \sigma_{X,V}}&&\rT V\otimes (X\otimes Y),\\
&X\otimes  \rT V\otimes Y\ar[ru]_{\sigma_{X,V}\otimes \ide_Y}&
}}}
\end{gather}
\begin{gather}
\label{augmenteddiag2}
\vcenter{\hbox{
\xymatrix{
X\otimes \rT(V\otimes W)\ar[rr]^-{\sigma_{X,V\otimes W}}\ar[d]_{\ide_X\otimes \left(\mu^{\rT}_{V,W}\right)^{-1}}&& \rT(V\otimes W)\otimes X\\
X\otimes \rT V \otimes \rT W\ar[rd]_{\sigma_{X,V}\otimes \ide_{\rT W}}&& \rT V\otimes \rT W\otimes X\ar[u]_{\mu^{\rT}_{V,W}\otimes \ide_X},\\
&\rT V\otimes X\otimes \rT W\ar[ru]_{\ide_{\rT V}\otimes \sigma_{X,W}}&
}}}
\end{gather}
\begin{gather}
\label{augmenteddiag3}
\vcenter{\hbox{
\xymatrix@C=40pt{
\rF\rT(V)\otimes \rF\rT(W)\ar[r]^{\mu^{\rF}_{\rT V,\rT W}}\ar[rd]_{\tau_V\otimes \tau_W}&\rF(\rT(V)\otimes \rT(W))\ar[r]^{\rF(\mu^{\rT}_{V,W})}&\rF\rT(V\otimes W)\ar[ld]^{\tau_{V\otimes W}}.\\
&V\otimes W&
}}}
\end{gather}
The last condition states that $\tau$ is a morphism of monoidal functors, cf. \cite{EGNO}*{Definition 2.4.8}.
Moreover, we require that for any objects $X$ of $\cM$ and $V$ of $\cB$,
\begin{align}
\sigma_{X,\one}&=\ide_{X}, &\sigma_{\one, V}&=\ide_{\rT(V)}, &\tau_{\one}=\ide_{\one}.
\end{align}

A \emph{monoidal functor of $\cB$-augmented} monoidal categories $\rG \colon \cM\to \cN$ is a monoidal functor such that the there exists isomorphisms of monoidal functors
\begin{align}
\theta=\theta^{\rG}\colon \rG\rT_{\cM}\natisomorph\rT_{\cN}, && \phi=\phi^{\rG}\colon \rF_{\cM}\natisomorph \rF_{\cN}\rG.
\end{align}
Further, the compatibility condition that the diagrams
\begin{gather}\vcenter{\hbox{
\xymatrix{\rG(X\otimes \rT_{\cM}(V))\ar_-{\left(\mu_{X,\rT_{\cM}(V)}^{\rG}\right)^{-1}}[d]  \ar^-{\rG(\sigma^{\cM}_{X,V})}[rr]&&\rG(\rT_{\cM}(V)\otimes X)\ar^-{\left(\mu_{\rT_{\cM}(V),X}^{\rG}\right)^{-1}}[d]\\
\rG(X)\otimes \rG\rT_{\cM}(V)\ar_-{\rG(X)\otimes \theta_V}[d]&&\rG\rT_{\cM}(V)\otimes \rG(X)\ar^-{\theta_V\otimes \rG(X)}[d]\\
\rG(X)\otimes \rT_{\cN}(V)\ar^-{\sigma^{\cN}_{\rG(X),V}}[rr]&&\rT_{\cN}(V)\otimes \rG(X),
}}}\label{augmentedfunctor1}\\
\vcenter{\hbox{
\xymatrix{
\rF_{\cM}\rT_{\cM}(V)\ar[rr]^{\phi_{\rT_{\cM}(V)}}\ar[drr]_{\tau^{\cM}_V}&&\rF_{\cN}\rG\rT_{\cM}(V)\ar[rr]^{\rF_{\cN}(\theta_V)}&&\rF_{\cN}\rT_{\cN}(V) \ar[dll]^{\tau^{\cN}_V}\\
&&V&&
}}}
\label{augmentedfunctor2}
\end{gather}
commute for all objects $X$ of $\cM$, $V$ of $\cB$ is required. In particular, $\theta_{\one}=\ide_{\one}, \phi_{\one}=\ide_{\one}$.
\end{definition}
We define the $\Bbbk$-linear $2$-category $\MonCat_\cB$ of $\cB$-augmented categories of $\cM$ as having 
 $\cB$-augmented monoidal functors as $1$-morphisms, and all monoidal natural transformations as $2$-morphisms.
Note that given two functors $\rG\colon \cM\to \cN$, $\rH\colon \cN\to \cO$ of $\cB$-augmented monoidal categories, the composition $\rH\rG$ is a functor of $\cB$-augmented monoidal categories with the natural isomorphisms given by 
\begin{align}\label{composedaugmented}
\theta^{\rH\rG}_V&=\theta_{V}^{\rH}\rH(\theta_V^{\rG}),&
\phi^{\rH\rG}_X&=\phi^{\rH}_{\rG (X)}\phi^{\rG}_X.
\end{align}

\begin{example}\label{trivexample}
Every monoidal category in $\Cat^c$ which is complete under arbitrary countable biproducts is $\Vect$-augmented for the category of $\Bbbk$-vector spaces. 
For example, for a bialgebra $B$ over $\Bbbk$, $\lmod{B}$ gives such a category.
\end{example}

\begin{example}
If $\cM$ is a $\cB$-augmented monoidal category, then  $\cM^{\oop}$ is $\ov{\cB}$-augmented. The isomorphism $\sigma^{\oop}\colon \otimes ^{\oop}(\ide_{\cM}\boxtimes \rT)\natisomorph (\otimes ^{\oop})^{\oop}(\ide_{\cM}\boxtimes \rT)$ is taken to be $\sigma^{-1}$.
\end{example}

\begin{example}\label{braidedaugmented}
Let $\cB$ be a braided monoidal category. Then $\cB$ is $\cB$-augmented, with $\rF=\ide_{\cB}$, $\rT=\ide_{\cB}$, and $\tau=\ide$, $\sigma=\Psi$, the braiding. The next lemma generalizes this example and clarifies the relation to braided monoidal categories.
\end{example}

\begin{lemma}\label{augmentedbraided}
If $\cM$ is braided monoidal and both $\rF$, $\rT$ are functors of braided monoidal categories, such that $\tau$ is an isomorphism of monoidal functors, then  $\cM$ can be given the structure of a $\cB$-augmented monoidal category.
\end{lemma}
\begin{proof}
If $\cM$ is braided monoidal, and $\rF$, $\rT$ are compatible with the braided monoidal structure as in \cite{EGNO}*{Definition 8.1.7}, and $\tau$ is a morphism of monoidal functors, then Eq. (\ref{augmenteddiag3}) holds by assumption. Setting
\begin{align}
\sigma_{X,V}&:=\Psi^{\cM}_{X,\rT(V)}, &X\in \cM, V\in \cB,
\end{align}
then Eqs. (\ref{augmenteddiag1})--(\ref{augmenteddiag2}) hold by the hexagonal diagrams (cf. e.g. \cite{EGNO}*{8.1.1}), combined with $\rT$ being a functor of braided monoidal categories. The condition of Eq. (\ref{augmentedbraiding}) follows from the assumption that $\rF$ is a functor of braided monoidal categories.
\end{proof}

\begin{example}\label{bialgebraexpl}
Our main source of examples of $\cB$-augmented monoidal categories in $\Cat^c$ will be of the form $\lmod{B}(\cB)$, where $B$ is a bialgebra object in $\cB$ (cf. Section \ref{setup}). The category $\lmod{B}(\cB)$ is monoidal. The tensor product of two objects $(V,a_V)$, $(W,a_W)$, where $a_V\colon B\otimes V\to V$ and $a_W\colon B\otimes W\to W$ are the action morphisms,  is $\left(V\otimes W, a_{V\otimes W}\right)$ with
\begin{align}
a_{V\otimes W}:=(a_V\otimes a_W)(\ide_B\otimes \Psi_{B,V}\otimes \ide_V)(\Delta\otimes \ide_{V\otimes W}).\end{align}

Graphical calculus, as in \cite{Lau2}, which is inspired by that used in \cite{Maj6}, is helpful for computations in $\lmod{B}(\cB)$. For example, the tensor product of modules is depictured as
\begin{equation}a_{V\otimes W}=\vcenter{\hbox{
\begingroup%
  \makeatletter%
  \providecommand\color[2][]{%
    \errmessage{(Inkscape) Color is used for the text in Inkscape, but the package 'color.sty' is not loaded}%
    \renewcommand\color[2][]{}%
  }%
  \providecommand\transparent[1]{%
    \errmessage{(Inkscape) Transparency is used (non-zero) for the text in Inkscape, but the package 'transparent.sty' is not loaded}%
    \renewcommand\transparent[1]{}%
  }%
  \providecommand\rotatebox[2]{#2}%
  \ifx\svgwidth\undefined%
    \setlength{\unitlength}{30.44988496bp}%
    \ifx\svgscale\undefined%
      \relax%
    \else%
      \setlength{\unitlength}{\unitlength * \real{\svgscale}}%
    \fi%
  \else%
    \setlength{\unitlength}{\svgwidth}%
  \fi%
  \global\let\svgwidth\undefined%
  \global\let\svgscale\undefined%
  \makeatother%
  \begin{picture}(1,0.88991524)%
    \put(0,0){\includegraphics[width=\unitlength,page=1]{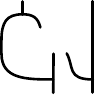}}%
    \put(0.5005321,0.47717664){\color[rgb]{0,0,0}\makebox(0,0)[lb]{\smash{$~$}}}%
    \put(0,0){\includegraphics[width=\unitlength,page=2]{tensoraction.pdf}}%
  \end{picture}%
\endgroup%
}}.\end{equation}
Here, $\Psi=\vcenter{\hbox{
\begingroup%
  \makeatletter%
  \providecommand\color[2][]{%
    \errmessage{(Inkscape) Color is used for the text in Inkscape, but the package 'color.sty' is not loaded}%
    \renewcommand\color[2][]{}%
  }%
  \providecommand\transparent[1]{%
    \errmessage{(Inkscape) Transparency is used (non-zero) for the text in Inkscape, but the package 'transparent.sty' is not loaded}%
    \renewcommand\transparent[1]{}%
  }%
  \providecommand\rotatebox[2]{#2}%
  \ifx\svgwidth\undefined%
    \setlength{\unitlength}{8.22109334bp}%
    \ifx\svgscale\undefined%
      \relax%
    \else%
      \setlength{\unitlength}{\unitlength * \real{\svgscale}}%
    \fi%
  \else%
    \setlength{\unitlength}{\svgwidth}%
  \fi%
  \global\let\svgwidth\undefined%
  \global\let\svgscale\undefined%
  \makeatother%
  \begin{picture}(1,1.0011064)%
    \put(0,0){\includegraphics[width=\unitlength,page=1]{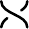}}%
  \end{picture}%
\endgroup%
}}$ denotes the braiding in $\cB$, $\Delta=\vcenter{\hbox{
\begingroup%
  \makeatletter%
  \providecommand\color[2][]{%
    \errmessage{(Inkscape) Color is used for the text in Inkscape, but the package 'color.sty' is not loaded}%
    \renewcommand\color[2][]{}%
  }%
  \providecommand\transparent[1]{%
    \errmessage{(Inkscape) Transparency is used (non-zero) for the text in Inkscape, but the package 'transparent.sty' is not loaded}%
    \renewcommand\transparent[1]{}%
  }%
  \providecommand\rotatebox[2]{#2}%
  \ifx\svgwidth\undefined%
    \setlength{\unitlength}{16.80101968bp}%
    \ifx\svgscale\undefined%
      \relax%
    \else%
      \setlength{\unitlength}{\unitlength * \real{\svgscale}}%
    \fi%
  \else%
    \setlength{\unitlength}{\svgwidth}%
  \fi%
  \global\let\svgwidth\undefined%
  \global\let\svgscale\undefined%
  \makeatother%
  \begin{picture}(1,0.76387957)%
    \put(0,0){\includegraphics[width=\unitlength]{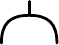}}%
  \end{picture}%
\endgroup%
}}$ denotes the coproduct of $B$, and $\triangleright=\vcenter{\hbox{
\begingroup%
  \makeatletter%
  \providecommand\color[2][]{%
    \errmessage{(Inkscape) Color is used for the text in Inkscape, but the package 'color.sty' is not loaded}%
    \renewcommand\color[2][]{}%
  }%
  \providecommand\transparent[1]{%
    \errmessage{(Inkscape) Transparency is used (non-zero) for the text in Inkscape, but the package 'transparent.sty' is not loaded}%
    \renewcommand\transparent[1]{}%
  }%
  \providecommand\rotatebox[2]{#2}%
  \ifx\svgwidth\undefined%
    \setlength{\unitlength}{9.12051093bp}%
    \ifx\svgscale\undefined%
      \relax%
    \else%
      \setlength{\unitlength}{\unitlength * \real{\svgscale}}%
    \fi%
  \else%
    \setlength{\unitlength}{\svgwidth}%
  \fi%
  \global\let\svgwidth\undefined%
  \global\let\svgscale\undefined%
  \makeatother%
  \begin{picture}(1,1.40715315)%
    \put(0,0){\includegraphics[width=\unitlength]{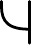}}%
    \put(0.13167899,0.48417802){\color[rgb]{0,0,0}\makebox(0,0)[lb]{\smash{$~$
}}}%
  \end{picture}%
\endgroup%
}}\colon B\otimes V\to V$ denotes a left coaction of $B$ on $V$.

The monoidal category $\lmod{B}(\cB)$ is $\cB$-augmented, using
the forgetful functor $\rF$ mapping a left $B$-module to the underlying object in $\cB$, and defining $\rT(V)$ to be the trivial $B$-module $V^{\op{triv}}$, using the trivial action $a^{\op{triv}}=\varepsilon\otimes \ide_V$ via the counit $\varepsilon$. Then $\sigma_{X,V}:=\rT(\Psi_{\rF(X),V})$ is a natural isomorphism as required, for $X$ in $\lmod{B}(\cB)$ and $V$ in $\cB$. For this, we check that $\rT(\Psi_{\rF(X),V})$ in fact is a morphism $X\otimes V^{\op{triv}}\to V^{\op{triv}}\otimes X$ of left $B$-modules in $\cB$. Indeed, this follows from
\begin{equation}\vcenter{\hbox{
\begingroup%
  \makeatletter%
  \providecommand\color[2][]{%
    \errmessage{(Inkscape) Color is used for the text in Inkscape, but the package 'color.sty' is not loaded}%
    \renewcommand\color[2][]{}%
  }%
  \providecommand\transparent[1]{%
    \errmessage{(Inkscape) Transparency is used (non-zero) for the text in Inkscape, but the package 'transparent.sty' is not loaded}%
    \renewcommand\transparent[1]{}%
  }%
  \providecommand\rotatebox[2]{#2}%
  \newcommand*\fsize{\dimexpr\f@size pt\relax}%
  \newcommand*\lineheight[1]{\fontsize{\fsize}{#1\fsize}\selectfont}%
  \ifx\svgwidth\undefined%
    \setlength{\unitlength}{173.1964275bp}%
    \ifx\svgscale\undefined%
      \relax%
    \else%
      \setlength{\unitlength}{\unitlength * \real{\svgscale}}%
    \fi%
  \else%
    \setlength{\unitlength}{\svgwidth}%
  \fi%
  \global\let\svgwidth\undefined%
  \global\let\svgscale\undefined%
  \makeatother%
  \begin{picture}(1,0.19995375)%
    \lineheight{1}%
    \setlength\tabcolsep{0pt}%
    \put(0,0){\includegraphics[width=\unitlength,page=1]{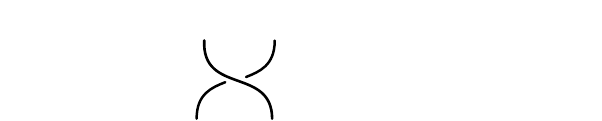}}%
    \put(0.47772942,0.08941879){\color[rgb]{0,0,0}\makebox(0,0)[lt]{\lineheight{0}\smash{\begin{tabular}[t]{l}$=$\end{tabular}}}}%
    \put(0,0){\includegraphics[width=\unitlength,page=2]{modulebalancing.pdf}}%
    \put(0.24002413,0.08856853){\color[rgb]{0,0,0}\makebox(0,0)[lt]{\lineheight{0}\smash{\begin{tabular}[t]{l}$=$\end{tabular}}}}%
    \put(0,0){\includegraphics[width=\unitlength,page=3]{modulebalancing.pdf}}%
    \put(0.71589831,0.08988238){\color[rgb]{0,0,0}\makebox(0,0)[lt]{\lineheight{0}\smash{\begin{tabular}[t]{l}$=$\end{tabular}}}}%
  \end{picture}%
\endgroup%
}}.\end{equation}
Here, the isomorphism $\tau\colon \rF\rT\natisomorph \ide$ consists of identity morphisms.

A morphism of bialgebras $B\to C$ in $\cB$ induces a functor of $\cB$-augmented monoidal categories $\lmod{C}(\cB)\to \lmod{B}(\cB)$, where the natural isomorphisms $\theta$ and $\phi$ both consist of identities.

Note the subtlety that the category of \emph{right} $B$-modules $\rmod{B}(\cB)$ is $\overline{\cB}$-augmented, rather than $\cB$-augmented.
\end{example}

\begin{example}\label{comoduleaugmented}
Dually, we can define $\lcomod{B}(\cB)$ using left comodules instead of module. This category is $\overline{\cB}$-augmented, while right $B$-comodules $\rcomod{B}(\cB)$ yield a $\cB$-augmented monoidal category.
\end{example}

In the case of $\cM=\lmod{H}(\cB)$ for a Hopf algebra in $\cB$ we will require the following equivalence for the $\overline{\cB}$-augmented monoidal category with opposite tensor product.

\begin{lemma}\label{oppositeaugmented}
Let $H$ be a Hopf algebra object in $\cB$ with invertible antipode.
There is an equivalence of $\overline{\cB}$-augmented monoidal categories $\left(\lmod{H}(\cB)\right)^{\oop}\simeq \rmod{H}(\cB)$.
\end{lemma}
\begin{proof}
Consider the equivalence of categories given by the functors
\begin{align*}
\Phi&\colon \lmod{H}(\cB)\longrightarrow \rmod{H}(\cB),&
(V,\triangleright )\longmapsto (V, \triangleleft'), &&\triangleleft'=\triangleright\Psi_{V,H}(\ide_V\otimes S)\\
\Phi'&\colon \rmod{H}(\cB)\longrightarrow \lmod{H}(\cB),&
(V,\triangleleft )\longmapsto (V, \triangleright'), &&\triangleright'=\triangleleft\Psi^{-1}_{V,H}(S^{-1}\otimes \ide_V).
\end{align*}
Using graphical calculus as above, the associated right and left actions are
\begin{align}\triangleleft'=\vcenter{\hbox{
\begingroup%
  \makeatletter%
  \providecommand\color[2][]{%
    \errmessage{(Inkscape) Color is used for the text in Inkscape, but the package 'color.sty' is not loaded}%
    \renewcommand\color[2][]{}%
  }%
  \providecommand\transparent[1]{%
    \errmessage{(Inkscape) Transparency is used (non-zero) for the text in Inkscape, but the package 'transparent.sty' is not loaded}%
    \renewcommand\transparent[1]{}%
  }%
  \providecommand\rotatebox[2]{#2}%
  \ifx\svgwidth\undefined%
    \setlength{\unitlength}{30.18503552bp}%
    \ifx\svgscale\undefined%
      \relax%
    \else%
      \setlength{\unitlength}{\unitlength * \real{\svgscale}}%
    \fi%
  \else%
    \setlength{\unitlength}{\svgwidth}%
  \fi%
  \global\let\svgwidth\undefined%
  \global\let\svgscale\undefined%
  \makeatother%
  \begin{picture}(1,1.21952462)%
    \put(0,0){\includegraphics[width=\unitlength,page=1]{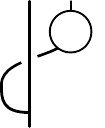}}%
    \put(0.54661962,0.78503177){\color[rgb]{0,0,0}\makebox(0,0)[lb]{\smash{$S$}}}%
  \end{picture}%
\endgroup%
}}, &&\triangleright'=\vcenter{\hbox{
\begingroup%
  \makeatletter%
  \providecommand\color[2][]{%
    \errmessage{(Inkscape) Color is used for the text in Inkscape, but the package 'color.sty' is not loaded}%
    \renewcommand\color[2][]{}%
  }%
  \providecommand\transparent[1]{%
    \errmessage{(Inkscape) Transparency is used (non-zero) for the text in Inkscape, but the package 'transparent.sty' is not loaded}%
    \renewcommand\transparent[1]{}%
  }%
  \providecommand\rotatebox[2]{#2}%
  \ifx\svgwidth\undefined%
    \setlength{\unitlength}{54.46342773bp}%
    \ifx\svgscale\undefined%
      \relax%
    \else%
      \setlength{\unitlength}{\unitlength * \real{\svgscale}}%
    \fi%
  \else%
    \setlength{\unitlength}{\svgwidth}%
  \fi%
  \global\let\svgwidth\undefined%
  \global\let\svgscale\undefined%
  \makeatother%
  \begin{picture}(1,0.95663033)%
    \put(0,0){\includegraphics[width=\unitlength]{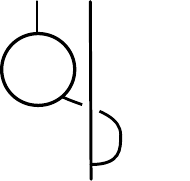}}%
    \put(0.07198333,0.4776608){\color[rgb]{0,0,0}\makebox(0,0)[lb]{\smash{$S^{\text{-1}}$}}}%
    \put(0.75953169,0.32175952){\color[rgb]{0,0,0}\makebox(0,0)[lb]{\smash{$.$}}}%
  \end{picture}%
\endgroup%
}}\end{align}
It is an exercise to check that the functors $\Phi$, $\Phi'$ give an equivalence of categories. 
We have to verify that $\Phi$ is a functor of monoidal categories $\left(\lmod{H}(\cB)\right)^{\oop}\to \rmod{H}(\cB)$. Indeed, we can use the isomorphism 
$\mu^{\Phi}_{V,W}:=\Psi_{W,V}^{-1}\colon \Phi(W\otimes V)\to \Phi(V)\otimes \Psi(W)$. That is, we need to verify the following functional equation
\begin{align*}
\left(\triangleleft_{V\otimes W}\right)'(\Psi_{W,V}\otimes H)=\vcenter{\hbox{
\begingroup%
  \makeatletter%
  \providecommand\color[2][]{%
    \errmessage{(Inkscape) Color is used for the text in Inkscape, but the package 'color.sty' is not loaded}%
    \renewcommand\color[2][]{}%
  }%
  \providecommand\transparent[1]{%
    \errmessage{(Inkscape) Transparency is used (non-zero) for the text in Inkscape, but the package 'transparent.sty' is not loaded}%
    \renewcommand\transparent[1]{}%
  }%
  \providecommand\rotatebox[2]{#2}%
  \ifx\svgwidth\undefined%
    \setlength{\unitlength}{146.33064181bp}%
    \ifx\svgscale\undefined%
      \relax%
    \else%
      \setlength{\unitlength}{\unitlength * \real{\svgscale}}%
    \fi%
  \else%
    \setlength{\unitlength}{\svgwidth}%
  \fi%
  \global\let\svgwidth\undefined%
  \global\let\svgscale\undefined%
  \makeatother%
  \begin{picture}(1,0.40576627)%
    \put(0,0){\includegraphics[width=\unitlength,page=1]{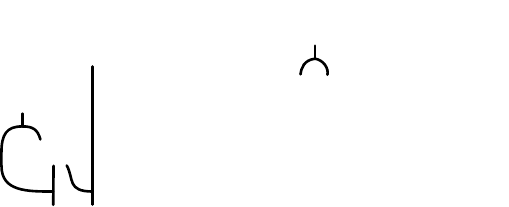}}%
    \put(0.10415519,0.09929557){\color[rgb]{0,0,0}\makebox(0,0)[lb]{\smash{$~$}}}%
    \put(0,0){\includegraphics[width=\unitlength,page=2]{leftrightproof.pdf}}%
    \put(0.27075347,0.2993768){\color[rgb]{0,0,0}\makebox(0,0)[lb]{\smash{$S$}}}%
    \put(0,0){\includegraphics[width=\unitlength,page=3]{leftrightproof.pdf}}%
    \put(0.35207819,0.18394972){\color[rgb]{0,0,0}\makebox(0,0)[lb]{\smash{$=$}}}%
    \put(0,0){\includegraphics[width=\unitlength,page=4]{leftrightproof.pdf}}%
    \put(0.59331073,0.32192828){\color[rgb]{0,0,0}\makebox(0,0)[lb]{\smash{$S$}}}%
    \put(0,0){\includegraphics[width=\unitlength,page=5]{leftrightproof.pdf}}%
    \put(0.681818,0.18350895){\color[rgb]{0,0,0}\makebox(0,0)[lb]{\smash{$=$}}}%
    \put(0,0){\includegraphics[width=\unitlength,page=6]{leftrightproof.pdf}}%
  \end{picture}%
\endgroup%
}}~=\Psi_{W,V}\triangleleft'_{W\otimes V}.
\end{align*}
Thus, we obtain an equivalence of monoidal categories. It is easily checked that $\tau$, $\sigma$ are compatible with the monoidal functors. The natural isomorphisms $\phi$, $\theta$ are all defined as identities in the underlying category $\cB$. 
\end{proof}

\begin{lemma}\label{augmentedtowers}
Let $\cM$ be a $\cB$-augmented monoidal category which is braided s.t. the structural functors $\rF_{\cM}, \rT_{\cM}$ are compatible with the braiding, and $\cN$ a $\cM$-augmented monoidal category. Then $\cN$ is also a $\cB$-augmented monoidal category.
\end{lemma}

\begin{theorem}\label{relativeproductproperties}
If $\cM$, $\cN$ are $\cB$-augmented (rigid) monoidal categories, then the $\cB$-balanced tensor product $\cM\boxtimes_{\cB}\cN$ is a  (rigid) monoidal category.
\end{theorem}
\begin{proof}
Note that $(\Cat^c, \boxtimes)$ is a monoidal category. In the proof, we work with different bracketing of iterated Kelly tensor products $\boxtimes$. For these, we chose a (coherent) set of natural isomorphisms between different ways of bracketing. We will however omit these from the notation for reasons of simplification of the exposition. 


Note that there exist a swap functor $\tau\colon \cM\boxtimes \cN\to \cN\boxtimes \cM$, and we observe that the composite 
$$\boxtimes_{\cB}(\otimes^{\cM}\boxtimes \otimes^{\cN})(\ide_{\cM}\boxtimes \tau \boxtimes \ide_{\cN})\colon \cM \boxtimes \cN\boxtimes \cM \boxtimes \cN\longrightarrow \cM \boxtimes_{\cB} \cN$$ factors (up to natural isomorphism) through a functor
$$\otimes^{\cM\boxtimes_{\cB}\cN}\colon(\cM\boxtimes_{\cB}\cN)\boxtimes (\cM\boxtimes_{\cB}\cN)\longrightarrow \cM\boxtimes_{\cB}\cN.$$

To demonstrate this, denote the balancing of the functor $\boxtimes_{\cC}\colon \cM\boxtimes \cN\to\cM\boxtimes_{\cB} \cN $ by $\eta$. 
Fix a pair of objects $M\otimes N\in \cM\boxtimes \cN$. Then the functor 
\begin{align*}
\cM\boxtimes\cN&\longrightarrow\cM\boxtimes_{\cB}\cN,&
X\otimes Y&\longmapsto (M\otimes^{\cM}X)\otimes(N\otimes^{\cN}Y)
\end{align*}
is $\cB$-balanced via the balancing isomorphism
$$\eta^{X\otimes Y}_{M,B,N}:=(\ide_{M\otimes X}\otimes \sigma_{N,B}^{-1}\otimes \ide_Y)\eta_{M\otimes X,B,N\otimes Y}.$$
This $\cB$-balancing isomorphism clearly satisfies  Eq. (\ref{balancedcohone}), and Eq. (\ref{balancedcoh}) follows from commutativity of the outer diagram in
\begin{align*}\resizebox{\hsize}{!}{
\xymatrix@C=0pt{
(M \otimes X\otimes \rT(B\otimes C))\otimes (N\otimes Y)\ar[rd]^{\eta_{M\otimes X,B\otimes C,N\otimes Y}}\ar[rr]^{\eta^{X\otimes Y}_{M,B\otimes C,N}}&& (M\otimes X)\otimes (N\otimes \rT(B\otimes C)\otimes Y)\\
&(M\otimes X)\otimes (\rT(B\otimes C)\otimes N\otimes Y)\ar[ru]^{\ide\otimes \sigma^{-1}_{N,B\otimes C}\otimes \ide}&
\\
(M\otimes X\otimes \rT B \otimes \rT C )\otimes (N\otimes Y)\ar[uu]^{\ide\otimes \mu^T_{B,C}\otimes \ide} \ar[dd]_{\eta_{M\otimes X\otimes \rT B,C,N\otimes Y}} &&(M\otimes X)\otimes (N\otimes\rT B\otimes \rT C\otimes Y)\ar[uu]^{\ide\otimes \mu^T_{B,C}\otimes \ide}\\
&(M\otimes X)\otimes (\rT B\otimes \rT C\otimes N\otimes Y)\ar[uu]^{\ide\otimes \mu^{\rT}_{B,C}\otimes \ide}\ar[dr]^{\ide\otimes \sigma^{-1}_{N,C}\otimes\ide}&\\
(M\otimes X\otimes \rT B)\otimes (\rT C \otimes N\otimes Y)\ar[ur]^{\eta_{M\otimes X,B,\rT C\otimes N\otimes Y}}\ar[rd]^{\ide\otimes \sigma^{-1}_{N,C}\otimes \ide}&&(M\otimes X)\otimes (\rT B\otimes N\otimes \rT C\otimes Y)\ar[uu]^{\ide\otimes \sigma^{-1}_{N,B}\otimes \ide}.\\
&(M\otimes X\otimes \rT B)\otimes (N\otimes \rT C\otimes Y)\ar[ur]^{\eta_{M\otimes X,B,N\otimes \rT C\otimes Y}}&
}}
\end{align*}
The inner diagrams commute using Eqs. (\ref{balancedcoh}), (\ref{augmenteddiag2}), naturality of $\eta$ applied to $\sigma^{-1}_{N,V}$ as well as the definition of $\eta^{X\otimes Y}$.
Hence, we obtain a factorization through $\cM\boxtimes_{\cB}\cN$, which is the endo-functor of $\cM\boxtimes_{\cB}\cN$ denoted by $\rR_{X\otimes Y}$. By completion under biproducts this gives a functor
\begin{align*}\cM\boxtimes \cN\to \Ends (\cM\boxtimes_{\cB}\cN),&&X\otimes Y\longmapsto \rR_{X\otimes Y}.\end{align*}
Using the symmetric closed structure from \cite{Kel}*{Section 6.5} we obtain a functor  $$\rR\colon \cM\boxtimes \cN\boxtimes(\cM\boxtimes_{\cB} \cN)\longrightarrow \cM\boxtimes_{\cB} \cN.$$

We claim that the functor $\rR$ is one of balanced bimodules (in the rightmost $\boxtimes$-product). Indeed, we may use the composite natural isomorphism 
\begin{gather*}
\eta^{\rR}_{M\otimes N,B,X\otimes Y}\colon (M\otimes^{\cM} \rT B\otimes^{\cM}X)\otimes (N\otimes^{\cN} Y) \longrightarrow (M\otimes^{\cM}X)\otimes (\rT B \otimes^{\cN}N\otimes^{\cN} Y),\\
\eta^{\rR}_{M\otimes N,B,X\otimes Y}=\eta_{M\otimes X,B,N\otimes Y}(\ide_{M}\otimes \sigma^{-1}_{X,B}\otimes \ide_{N\otimes Y}).
\end{gather*}
It follows very similarly to before that $\eta^{\rR}$ satisfies Eqs. (\ref{balancedcoh}) and (\ref{balancedcohone}). Hence we obtain a factorization through $(\cM\boxtimes_{\cB} \cN)\boxtimes(\cM\boxtimes_{\cB} \cN)$, denoted by
$\otimes^{\cM\boxtimes_{\cB}\cN}$, as stated. As such, it is unique up to unique natural isomorphism by the universal property. 

To proof that $\otimes^{\cM\boxtimes_{\cB}\cN}$ provides a monoidal structure now follows by use of the functoriality properties of $\boxtimes_{\cB}$ as detailed after Definition \ref{relativetensordef}. 
This way, we obtain the functors
\begin{align*}
\otimes^{\cM\boxtimes_{\cB}\cN}(\otimes^{\cM\boxtimes_{\cB}\cN}\boxtimes \ide_{\cM\boxtimes_{\cB}\cN}), && \otimes^{\cM\boxtimes_{\cB}\cN}(\ide_{\cM\boxtimes_{\cB}\cN}\boxtimes \otimes^{\cM\boxtimes_{\cB}\cN})
\end{align*}
as factorization via the universal property of $\boxtimes_{\cB}$. However, using that $\cM$, $\cN$ are strict monoidal categories, these are factorization of the same balanced functor through canonically isomorphic (iterated) relative tensor products. Hence the functors are isomorphic via a unique natural isomorphism. Uniqueness of the induced natural isomorphism ensures coherence, thus giving a monoidal structure.


If $\cM$ and $\cN$ are rigid, the so is $\cM\boxtimes\cN$. As a monoidal functor, the canonical factorization functor  $\cM\boxtimes\cN\to \cM\boxtimes_{\cB}\cN$ preserves duals. As this functor is essentially surjective, duals exit in the latter category. 
\end{proof}

Note that one can similarly prove a variation of Theorem \ref{relativeproductproperties} where $\cM$ is $\cB$-augmented and $\cN$ is $\overline{\cB}$-augmented as the inverse $\sigma^{-1}$ also satisfies $\otimes$-compatibility.
In particular, given a $\cB$-augmented monoidal category $\cM$, then $\cM\boxtimes_{\cB}\cM^{\oop}$ is a monoidal category. 

\begin{corollary}\label{braidedtensorcor}
If $\cM$ and $\cN$ are braided monoidal categories with functors  
\begin{align*}
\adj{\rF_{\cM}}{\cM}{\cB}{\rT_{\cM}}, && \adj{\rF_{\cN}}{\cN}{\cB}{\rT_{\cN}},
\end{align*}
of braided monoidal categories together with isomorphisms of monoidal functors 
\begin{align*}
\tau^{\cM}\colon \rF_{\cM}\rT_{\cM}\stackrel{\sim}{\Longrightarrow} \ide_{\cB},&&
\tau^{\cN}\colon \rF_{\cN}\rT_{\cN}\stackrel{\sim}{\Longrightarrow} \ide_{\cB},
\end{align*}
then $\cM\boxtimes_{\cB}\cN$ is a braided monoidal category.
\end{corollary}
\begin{proof}
We may equip both $\cM$ and $\cN$ with the structure of a $\cB$-augmented monoidal category using Lemma \ref{augmentedbraided}. It then follows that $\cM\boxtimes_{\cB}\cN$ is a monoidal category using Proposition \ref{relativeproductproperties}. It is directly verified that if $\cM$, $\cN$ are braided with braidings $\Psi^{\cM}$, $\Psi^{\cN}$, then $\Psi^{\cM}\boxtimes \Psi^{\cN}$ defines a braiding on $\cM\boxtimes \cN$, which under application of the functor $\boxtimes_{\cB}\boxtimes\boxtimes_{\cB}$ descents to a braiding on $\cM\boxtimes \cN$ by essential surjectivity.
\end{proof}

Note that for braided fusion categories, a similar construction as obtained in Corollary \ref{braidedtensorcor} has been carried out in \cite{Gre}. To conclude this section, we shall examine some examples of representation-theoretic nature.

\begin{example}\label{tensorBBmod}
Let $B\in \Hopf(\cB)$, and consider $\cM=\lmod{B}(\cB)$ as a $\cB$-augmented monoidal category as in Example \ref{bialgebraexpl}. It was shown in Proposition \ref{Bmodreltensor} that $\cM\boxtimes_{\cB}\cM^{\oop}$ is equivalent to $\lrmod{B}{B}(\cB)$ as a $\cB$-bimodule category. The tensor product on the latter category is given by $(V\otimes W,\triangleright_{V\otimes W},\triangleleft_{V\otimes W} )$, where
\begin{align}
\triangleright_{V\otimes W}&=(\triangleright_V\otimes \triangleright_W)(\ide_B\otimes \Psi_{B,V}\otimes \ide_W)(\Delta_B\otimes \ide_{V\otimes W}),\\
\triangleleft_{V\otimes W}&=(\triangleleft_V\otimes\triangleleft_W)(\ide_V\otimes \Psi_{W,B}\otimes \ide_B)(\ide_{V\otimes W}\otimes\Delta_B).
\end{align}
for $(V,\triangleright_V,\triangleleft_V)$, $(W,\triangleright_W,\triangleleft_W)$ objects of  $\lrmod{B}{B}(\cB)$, and accordingly for morphisms. Recall the equivalence of $\ov{\cB}$-augmented monoidal categories $\rmod{B}(\cB)\simeq \left( \lmod{B}(\cB)\right)^{\oop}=\cM^{\oop}$ from Lemma \ref{oppositeaugmented}. By Theorem \ref{relativeproductproperties}, the category $\cM\boxtimes_{\cB} \cM^{\oop}$ has the structure of a monoidal category which is equivalent, under the equivalence from Proposition \ref{Bmodreltensor}, to the monoidal structure on $\lrmod{B}{B}(\cB)$ described above.
\end{example}


\subsection{Balanced Categorical Bimodules}\label{balancedbimodsect}

We saw in Section \ref{catmodules} that a categorical $\cM$-$\cN$-bimodule over  monoidal categories $\cM$, $\cN$ in $\Cat^c$ can be defined as a monoidal functor $\cM\boxtimes \cN^{\oop}\to \Ends^c(\cV)$. Working with a $\cB$-augmented monoidal categories $\cM$ and $\cN$, we present the following definition of a suitable category of \emph{$\cB$-balanced} bimodules. Note for this that $\cM$ and $\cN$ are $\cB$-bimodules with action induced, via  $\rT_{\cM}\colon \cB\to \cM$, respectively $\rT_{\cN}\colon \cB\to \cN$, from the regular bimodule structure.

\begin{definition} Let $\cM$ and $\cN$ be $\cB$-augmented monoidal categories. 
A \emph{$\cB$-balanced $\cM$-$\cN$-bimodule} is a $\cM$-$\cN$-bimodule $\cV$ as in Proposition \ref{bimoduledata}, together with a natural isomorphism 
$$\beta\colon \triangleleft(\ide_{\cV}\boxtimes\rT_{\cN})\natisomorph \triangleright(\rT_{\cM}\boxtimes\ide_{\cV}),$$
satisfying the coherence condition that the diagrams (\ref{betacoh1}) and (\ref{betacoh2}) commute for any objects $V$ of $\cV$, $X,Y$ of $\cB$, $M$ of $\cM$, and $N$ of $\cN$:
\begin{gather}\label{betacoh1}\vcenter{\hbox{\xymatrix{
V\triangleleft \rT_{\cN}(X\otimes Y)\ar[rr]^{\beta_{V,X\otimes Y}}\ar[d]_{V\triangleleft (\mu^{\rT_{\cN}}_{X,Y})^{-1}} && \rT_{\cM}(X\otimes Y)\triangleright V\ar[d]^{\left(\mu_{X,Y}^{\rT_{\cM}}\right)^{-1}\triangleright V}\\
V\triangleleft \rT_{\cN}(X)\otimes \rT_{\cN}(Y)\ar[d]_{\xi_{V, \rT_{\cN}(X),\rT_{\cN}(Y)}} && \rT_{\cM}(X)\otimes \rT_{\cM}(Y)\triangleright V\ar[d]^{\chi_{\rT_{\cM}(X),\rT_{\cM} (Y),V}}\\
(V\triangleleft \rT_{\cN}(X))\triangleleft \rT_{\cN}(Y)\ar[d]_{\beta_{V,X}\triangleleft \rT_{\cN}(Y)}&&\rT_{\cM}(X)\triangleright(\rT_{\cM}(Y)\triangleright V)\ar[d]^{\rT_{\cM}(X)\triangleright \beta^{-1}_{V,Y}}\\
(\rT_{\cM}(X)\triangleright V)\triangleleft \rT_{\cN}(Y)\ar[rr]^{\zeta_{\rT_{\cM}(X),V,\rT_{\cN}(Y)}} && \rT_{\cM}(X)\triangleright(V\triangleleft \rT_{\cN}(Y));
}}}
\end{gather}
\begin{gather}
\label{betacoh2}
\vcenter{\hbox{
\xymatrix{
M\triangleright ((V\triangleleft \rT_{\cN}(X))\triangleleft N)\ar[rr]^{M\triangleright (\beta_{V,X}\triangleleft N)}\ar[d]_{M\triangleright \xi^{-1}_{V,\rT_{\cN}(X),N}}&&M\triangleright (( \rT_{\cM}(X)\triangleright V)\triangleleft N)\ar[d]^{M\triangleright \zeta_{\rT_{\cM}(X),V,N}}\\
M\triangleright (V\triangleleft \rT_{\cN}(X)\otimes N)\ar[d]_{M\triangleright(V\triangleleft \sigma^{-1}_{N,X})}&&M\triangleright ( \rT_{\cM}(X)\triangleright (V\triangleleft N))\ar[d]^{\chi^{-1}_{M,\rT_{\cM}(X),V\triangleleft N}}\\
M\triangleright (V\triangleleft N\otimes \rT_{\cN}(X))\ar[d]_{M\triangleright \xi_{V,N,\rT_{\cN}(X)}}&&M\otimes\rT_{\cM}(X)\triangleright (V\triangleleft N)\ar[d]^{\sigma_{M,X}\triangleright (V\triangleleft N)}\\
M\triangleright ((V\triangleleft N)\triangleleft \rT_{\cN}(X))\ar[d]_{\zeta^{-1}_{M, V\triangleleft N,\rT_{\cN}(X)}}&&\rT_{\cM}(X)\otimes M\triangleright (V\triangleleft N)\ar[d]^{\chi_{\rT_{\cM}(X),M,V\triangleleft N}}\\
(M\triangleright (V\triangleleft N))\triangleleft \rT_{\cN}(X)\ar[rr]^{\beta_{M\triangleright(V\triangleleft N),X}}&&\rT_{\cM}(X)\triangleright( M\triangleright (V\triangleleft N)).
}}}
\end{gather}
In addition, we require
\begin{align}
\beta_{V,\one}=\ide_{V}.
\end{align}

A \emph{morphism of $\cB$-balanced $\cM$-$\cN$-bimodules}  $\rG\colon (\cV,\beta^{\cV})\to (\cW,\beta^{\cW})$ is a functor $(\rG,\lambda,\rho)\colon \cV\to \cW$ of $\cM$-$\cN$-bimodules such that the diagram
\begin{align}\label{betacoh}
\vcenter{\hbox{
\xymatrix{
\rF(V\triangleleft_{\cV}\rT_{\cN} (X))\ar[rr]^{\rF(\beta^{\cV}_{V,X})}\ar[d]^{\rho_{V,\rT_{\cN} (X)}}&&\rF(\rT_{\cM} (X)\triangleright_{\cV}V)\ar[d]^{\lambda_{\rT_{\cM}(X),V}}\\
\rF(V)\triangleleft_{\cW}\rT_{\cN} (X)\ar[rr]^{\beta^{\cW}_{\rF(V),X}}&&\rT_{\cM}(X)\triangleright_{\cW}\rF(V)
}}}
\end{align}
commutes for all objects $V$ of $\cV$ and $X$ of $\cB$.

A $2$-morphism of $\cB$-balanced $\cM$-$\cN$-bimodules is just a $2$-morphism of $\cM$-$\cN$-bimodules. This way, we obtain the $2$-category $\BiMod_{\cM\text{--}\cN}^{\cB}$ of $\cB$-balanced bimodules over $\cM$ as a $2$-full subcategory of $\BiMod_{\cM\text{--}\cN}$.
Given two $\cB$-balanced $\cM$-$\cN$-bimodules $\cV$, $\cW$, we write $\cHom_{\cM\text{--}\cN}^{\cB}(\cV,\cW)$ for the category for $\cB$-balanced morphisms of $\cM$-bimodules from $\cV$ to $\cW$.
We refer to $\cB$-balanced $\cM$-$\cM$-bimodules simply as $\cB$-balanced $\cM$-bimodules.
\end{definition}

If we interpret the concept of a $\cB$-augmented monoidal category as a categorical analogue of the concept of a $C$-augmented $C$-algebra $R$, over a commutative ring $C$, then $\cB$-balanced bimodules are the categorical analogue of an $R$-bimodule such that the left and right $C$-action coincide.


A large supply of examples of $\cB$-balanced bimodules will be constructed in Proposition \ref{balancingfunctor}.

\begin{proposition}\label{balancedbimondprop}
The datum of a $\cB$-balanced $\cM$-$\cN$-bimodule for $\cB$-augmented monoidal categories $\cM$ and $\cN$ is equivalent to the datum of a monoidal functor $\triangleright \colon \cM\boxtimes \cN^{\oop}\to \Ends^c(\cV)$ together with a structural natural isomorphism 
$$\mu^{\triangleright}\colon \triangleright \boxtimes \triangleright \natisomorph \triangleright (\ide\otimes \ide)$$
such that the functor $\triangleright$ is $\cB$-balanced with balancing $\eta$, and satisfies the compatibility that the diagrams (\ref{balancingaction}) and (\ref{balancingaction2}) commute for any objects $M,M'$ of $\cM$, $N,N'$ of $\cN$, and $X$ of $\cB$:
\begin{gather}\label{balancingaction}
\vcenter{\hbox{\xymatrix{
\triangleright (M\boxtimes N)\triangleright((M'\otimes \rT(X))\otimes N')\ar[rr]^{\triangleright(M\otimes N)(\eta_{M',X,N'})}\ar[d]^{\mu^{\triangleright}_{M\otimes N,(M'\otimes \rT(X))\otimes N'}}&&\triangleright (M\otimes N)\triangleright(M'\otimes (\rT(X)\otimes N'))\ar[d]_{\mu^{\triangleright}_{M\otimes N,M'\otimes (\rT(X)\otimes N')}}\\
\triangleright ((M\otimes M'\otimes \rT(X) )\otimes (N'\otimes N))\ar[rr]^{\eta_{M\otimes M',X,N'\otimes N}}&&\triangleright((M\otimes M')\otimes (\rT(X)\otimes N'\otimes N));
}}}
\end{gather}
\begin{gather}
\label{balancingaction2}
\vcenter{\hbox{\xymatrix{
\triangleright((M\otimes \rT(X))\otimes N)\triangleright(M'\otimes N')\ar[d]^{\mu^\triangleright_{(M\otimes \rT(X))\otimes N,M'\otimes N'}}\ar[rr]^{\eta_{M,X,N}\triangleright(M'\otimes N')}&&\triangleright(M\otimes (\rT(X)\otimes N))\triangleright(M'\otimes N')\ar[d]_{\mu^\triangleright_{M\otimes (\rT(X)\otimes N),M'\otimes N'}}\\
\triangleright((M\otimes \rT(X)\otimes M')\otimes (N'\otimes N))\ar[d]^{\triangleright((M\otimes \sigma^{-1}_{M',X})\otimes (N'\otimes N))}&&\triangleright((M\otimes M')\otimes (N'\otimes \rT(X)\otimes N))\ar[d]_{\triangleright((N\otimes N')\otimes (\sigma_{N',X}\otimes N))}\\
\triangleright((M\otimes M'\otimes \rT(X))\otimes (N'\otimes N))\ar[rr]^{\eta_{M\otimes M',X,N'\otimes N}}&&\triangleright((M\otimes  M')\otimes (\rT(X)\otimes N'\otimes N)).
}}}
\end{gather}

Moreover, a morphism of $\cB$-balanced $\cM$-bimodules $(\rG,\pi)\colon (\cV,\triangleright_{\cV})\to (\cW,\triangleright_{\cW})$ is a morphism of bimodules such that the natural isomorphism $\pi$ in 
\begin{align}\label{pidiagram}
\vcenter{\hbox{
\xymatrix{
\cM\boxtimes\cN^{\oop}\ar[rr]^{\triangleright_{\cV}}\ar[d]_{\triangleright_{\cW}}&& \Ends^c(\cV)\ar@{=>}[lld]_{\pi}\ar[d]^{\rG(-)}\\
\Ends^c(\cW)\ar[rr]^{(-)\rG}&&\Homs^c(\cV,\cW)
}}}
\end{align}
is one of $\cB$-balanced functors.
\end{proposition}
\begin{proof}
Assume given a $\cB$-balanced module $(\cV, \beta)$. Recall Proposition \ref{bimoduledata} to translate this data into a monoidal functor $\triangleright_{\cV}\colon \cM\boxtimes \cN^{\oop}\to \Ends^c(\cV)$, by setting $$\triangleright(M\otimes N)(V)=M\triangleright (V\triangleleft N),$$ for objects $M$ of $\cM$, $N$ of $\cN$, and $V$ of $\cV$. 
To provide a $\cB$-balancing for the action functor $\triangleright$,  we require natural isomorphisms $\eta_{M,X,N,V}\colon M\otimes \rT(X)\triangleright (V\triangleleft N)\isomorph M\triangleright (V\triangleleft \rT(X)\otimes N).$
We define
\begin{align*}
\eta_{M,X,V,N}:=(M\triangleright \xi^{-1}_{V,\rT_{\cN}(X),N})(M\triangleright (\beta^{-1}_{V,X}\triangleleft N))(M\triangleright \zeta^{-1}_{\rT_{\cM}(X),V,N})\chi_{M,\rT_{\cM}(X), V\triangleleft N}.
\end{align*}
It then follows that $\eta$ satisfies Condition (\ref{balancedcoh}) of a $\cB$-balancing isomorphism using Condition (\ref{betacoh1}), naturality, and coherence conditions of the bimodule structure from Eqs. (\ref{modulecoherence2})--(\ref{bimodcoherence2}). 

Now we define
\begin{align*}\mu^\triangleright_{M\otimes N,M'\otimes N',V}&\colon \triangleright(M\otimes N)\triangleright(M'\otimes N')(V)\isomorph \triangleright((M\otimes M')\otimes (N'\otimes N))(V),\\
\mu^\triangleright_{M\otimes N,M'\otimes N',V}&:=\chi^{-1}_{M,M',V\triangleleft N'\otimes N}(M\triangleright(M'\triangleright \xi^{-1}_{V,N',N}))(M\triangleright \zeta_{M',V\triangleleft N',N}).
\end{align*}

Next, we observe that $\eta_{\one,X,V,\one}=\beta_{V,X}^{-1}$, and the very definition of $\eta_{M,X,V,N}$ implies that 
$$\eta_{M,X,V,N}\mu_{M\otimes N,\rT_{\cM}(X)\otimes\one,V}=\mu^\triangleright_{M\otimes N,\one\otimes \rT_{\cN}(X),V}(M\triangleright (\eta_{\one,X,V,\one}\triangleleft N)).$$
This proves the compatibility condition of $\eta$ with the monoidal structure as in diagram (\ref{balancingaction}) in the case where $M'\otimes N'=\one\otimes \one$. The general case follows as $\mu^{\triangleright}$ is compatible with tensor products.
The second condition, Equation (\ref{balancingaction2}), can also first verified if $M\otimes N=\one\otimes \one$. In this case, it precisely reduces to Equation (\ref{betacoh2}). Again, using that $\mu^{\triangleright}$ is compatible with tensor products, the general case follows.

One further checks that given a morphism of $\cB$-balanced bimodules $\rG$, by the compatibility condition of Eq. (\ref{betacoh}) together with compatibility with the bimodule structures, $\rG$ is compatible with the balancing isomorphism $\eta$ in the sense of Eq. (\ref{balancedcoh}). This amounts to commutativity of the diagram
\begin{align}\label{balancedmodcoh}
\vcenter{\hbox{\xymatrix{
\rG(M\otimes \rT_{\cM}(X)\triangleright (V\triangleleft N))\ar[rr]^{\rG(\eta^{\cV}_{M,X,N,V})}\ar[d]_{\lambda_{M\otimes \rT_{\cM}(X),V\triangleright N}}&&\rG(M\triangleright (V\triangleleft \rT_{\cN}(X) \otimes N))\ar[d]^{\lambda_{M,V\triangleright \rT_{\cN}(X) \otimes N}}\\
M\otimes \rT_{\cM}(X)\triangleright \rG(V\triangleleft N)\ar[d]_{M\otimes \rT_{\cN}(X)\triangleright \rho_{V,N}}&&M\triangleright \rG(V\triangleleft \rT_{\cN}(X) \otimes N)\ar[d]^{M\triangleright \rho_{V, \rT_{\cN}(X)\otimes N}}\\
M\otimes \rT_{\cM}(X)\triangleright (\rG(V)\triangleleft N)\ar[rr]^{\eta^{\cW}_{M,X,N,\rG(V)}}&&M\triangleright (\rG(V)\triangleleft \rT_{\cN}(X) \otimes N).
}}}
\end{align}
A natural isomorphism $\pi$ as in the diagram of Eq. (\ref{pidiagram}) is obtained as 
$$\pi_{M,N}=(M\triangleright \rho_{(-),N})\lambda_{M,(-)\triangleleft N}.$$
It is a morphism of $\cB$-balanced functors using that $\rG$ satisfies the condition of Eq. (\ref{betacoh}).

Conversely, given $\triangleright_{\cV}$ as in the statement of the proposition, we define 
$$\beta_{V,X}:=\eta_{\one,X,\one, V}^{-1}.$$
Conditions (\ref{betacoh1})--(\ref{betacoh2}) are now a consequence of Eqs. (\ref{balancingaction}), (\ref{balancingaction2}) and (\ref{balancedcoh}). A morphism $(\rG,\pi)$ satisfying the conditions of the statement of the proposition provides $\lambda_{M,V}=\pi_{M,V,\one}$ and $\rho_{V,N}=\pi_{\one,V,N}$ which are compatible with $\eta$, so in particular, are compatible with $\beta_{V,X}=\eta_{\one,X,\one, V}^{-1}$.

It is readily verified that the two functors described give an equivalence of categories.
\end{proof}

\begin{lemma}\label{regularbalanced}
If $\cM$ is a $\cB$-augmented monoidal category, then the regular bimodule on $\cM$ is a $\cB$-balanced bimodule over $\cM$.
\end{lemma}
\begin{proof}
The natural isomorphism $\sigma$ which is part of the definition of a $\cB$-augmented monoidal category provides a $\cB$-balancing for the regular module. 
\end{proof}

\begin{theorem}\label{bimoduletheorem}The $\Bbbk$-linear bicategories $\BiMod_{\cM\text{--}\cN}^{\cB}$ and $\lmod{\cM\boxtimes_{\cB}\cN^{\oop}}$ are biequivalent.
\end{theorem}
\begin{proof}
Proposition \ref{balancedbimondprop} provides a bifunctor $\BiMod_{\cM\text{--}\cN}^{\cB}\longrightarrow \lmod{\cM\boxtimes\cN^{\oop}}$. We want to show that the image $(\triangleright, \mu^\triangleright)\colon \cM\boxtimes \cN^{\oop}\to \Ends(\cV)$ of a $\cB$-balanced bimodule $\cV$ under this functor factors through $\cM\boxtimes_{\cB}\cN^{\oop}$. As $\triangleright$ is balanced, it follows that there exists a factorization $\triangleright_{\cB}\colon \cM\boxtimes_{\cB}\cN^{\oop}\to \Ends(\cV)$. We want to show that the natural isomorphism $\mu^\triangleright$ induces such a natural isomorphism after passing to the relative tensor product $\cM\boxtimes_{\cB}\cN^{\oop}$ defined in Theorem \ref{relativeproductproperties}. We proceed in two steps, as in the proof of Theorem \ref{relativeproductproperties}. 

First, $\mu^{\triangleright}\colon \triangleright\boxtimes \triangleright \natisomorph \triangleright \otimes$ is an isomorphism of $\cB$-balanced functors which are defined from $\cM\boxtimes \cN^{\oop}\boxtimes \cM\boxtimes \cN^{\oop}$ to $\Ends(\cV)$, and the $\cB$-balancing is defined in the last $\boxtimes$-product. Here, the balancing isomorphisms are given by
\begin{small}
\begin{gather*}
\triangleright(M\otimes N)\triangleright((M'\otimes \rT_{\cM}(X))\otimes N')(V)\xrightarrow{\triangleright(M\otimes N)(\eta_{M',X,N',V})} \triangleright(M\otimes N)\triangleright(M'\otimes (\rT_{\cM}(X)\otimes N'))(V),\\
 \triangleright((M\otimes M'\otimes \rT_{\cM}(X))\otimes(N'\otimes N)(V)\xrightarrow{\eta_{M\otimes M', X, N'\otimes N,V}} \triangleright((M\otimes M')\otimes(\rT_{\cM}(X)\otimes N'\otimes N)(V), 
\end{gather*}
\end{small}
for $\triangleright\boxtimes \triangleright$, respectively, $\triangleright~\otimes$. Condition (\ref{balancingaction}) gives that $\mu^{\triangleright}$ indeed is a morphism of $\cB$-balanced functors, satisfying the diagram (\ref{balancedmor}). Hence, by the universal property of the relative tensor product, there exists an isomorphism of the factorizations through $\cM\boxtimes \cN^{\oop}\boxtimes \cM\boxtimes_{\cB}\cN^{\oop}$. By uniqueness of the factorization, the coherences of $\mu^{\triangleright}$ making $\triangleright$ a monoidal functor also hold for the factorization. 

Second, the induced functors $\triangleright\boxtimes \triangleright, \triangleright~\otimes \colon \cM\boxtimes \cN^{\oop}\boxtimes \cM\boxtimes_{\cB} \cN^{\oop}\longrightarrow\Ends^c(\cV)$ from the first step are also a $\cB$-balanced functors in the first $\boxtimes$-tensor product. Here, we use the images of the balancing isomorphisms
\begin{small}
\begin{gather*}\triangleright((M\otimes \rT_{\cM}(X))\otimes N)\triangleright(M'\otimes N')(V)\xrightarrow{\eta_{M,X,N, \triangleright(M'\otimes N')(V)}}\triangleright(M\otimes (\rT_{\cM}(X)\otimes N))\triangleright(M'\otimes N')(V), \\
\triangleright((M\otimes \rT_{\cM}(X)\otimes M')\otimes (N'\otimes N))\xrightarrow{\eta'_{M\otimes N,X,M'\otimes N',V}}\triangleright ((M\otimes M')\otimes (N'\otimes \rT_{\cM}(X)\otimes N)),
\end{gather*}
\end{small}
after factoring through $\cM\boxtimes\cN^{\oop}\boxtimes_{\cM}\boxtimes_{\cB}\cN^{\oop}$,
where 
\begin{align*}
\eta'_{M\otimes N,X,M'\otimes N',V}:=& \triangleright((M\otimes M')\otimes (\sigma^{-1}_{N',X}\otimes N))\eta_{M\otimes M',X,N'\otimes N,V}\\&\circ \triangleright((M\otimes\sigma^{-1}_{M',X})\otimes (N'\otimes N))(V).
\end{align*}
Condition (\ref{balancingaction2}) gives that $\mu^{\triangleright}$ is a morphism of $\cB$-balanced functors with respect to these $\cB$-balancing isomorphisms. Hence, again by the universal property, we obtain a natural isomorphism on the factorizations through $(\cM\boxtimes \cN^{\oop})\boxtimes (\cM\boxtimes_{\cB} \cN^{\oop})$, which is coherent by uniqueness. Thus, there is a monoidal functor
$(\triangleright, \mu^\triangleright) \colon \cM\boxtimes_{\cB} \cN^{\oop}\longrightarrow \Ends(\cV).$

Conversely, given a monoidal functor $(\triangleright, \mu^\triangleright) \colon \cM\boxtimes_{\cB} \cN^{\oop}\longrightarrow \Ends(\cV),$ we can consider the monoidal functor 
$\triangleright \rT_{\cM,\cN^{\oop}}$, using the universal functor $\rT_{\cM,\cN^{\oop}}\colon \cM\boxtimes\cN^{\oop}\to \cM\boxtimes_{\cB}\cN^{\oop}$ from Definition \ref{reltensordef}, which is monoidal by Theorem \ref{relativeproductproperties}. Hence, this composition is again a monoidal functor. The $\cM$-$\cN$-bimodule thus obtained necessarily satisfies the conditions from Eqs. (\ref{balancingaction}--(\ref{balancingaction2}). Moreover, this assignment upgrades to a bifunctor $\lmod{\cM\boxtimes_{\cB}\cN^{\oop}}\longrightarrow \BiMod_{\cM\text{--}\cN}^{\cB}$ which forms local equivalences with the bifunctor provided in the beginning of the proof.
\end{proof}

\subsection{The Relative Monoidal Center}\label{monoidalcentersect}

This section contains the main construction of the paper, a relative version of the monoidal center.

\begin{definition}
Let $\cM$ be a $\cB$-augmented monoidal category in $\Cat^c$. The \emph{relative monoidal center} of $\cM$ over $\cB$ is defined to be the monoidal category
\[
\cZ_\cB(\cM):=\cHom_{\cM\text{--}\cM}^{\cB}(\cM,\cM),
\]
with composition as tensor product
\[
\rG\otimes \rH:=\rH\rG.
\]
\end{definition}
An object in $\cZ_{\cB}(\cM)$ is a morphism of $\cM$-bimodules $\phi\colon \cM\to \cM$ commuting, as in Equation (\ref{betacoh}), with the $\cB$-balancing isomorphisms. As the $\cB$-balancing isomorphisms on $\cM$ are given by the $\cB$-augmentation, cf. Example \ref{regularbalanced}, the requirement is equivalent to $\phi$ commuting the the $\cB$-augmentation, i.e.
\begin{align}\label{centeraugm}
\phi(\sigma_{M,V})=\lambda^{-1}_{\rT(V),M}\sigma_{\phi(M),V}\rho_{M,\rT(V)}, &&\forall M\in \cM, V\in \cB.
\end{align}

\begin{theorem}\label{centerismonoidal}
If $\cM$ be a $\cB$-augmented monoidal category in $\Cat^c$, then $\cZ_\cB(\cM)$ is a braided strict monoidal category in $\Cat^c$ together with a functor of monoidal categories $\rF\colon \cZ_\cB(\cM)\to \cM$. Moreover, if $\cM$ is rigid, then so is $\cZ_{\cB}(\cM)$.

\end{theorem}
\begin{proof}
By definition, $\cZ_{\cB}(\cM)$ is a strict monoidal category with respect to composition of functors. Given $(\rG, \lambda^{\rG}, \rho^{\rG}), (\rH, \lambda^{\rH}, \rho^{\rH})\colon \cM\to \cM$, then $\rG\rH$ has structural isomorphisms $\lambda^{\rG\rH}_{M,V}=\lambda^{\rG}_{M,\rH(V)}\lambda^{\rH}_{M,V}$, and $\rho^{\rG\rH}_{V,M}=\lambda^{\rG}_{\rH(V),M}\lambda^{\rH}_{V,M}$. 
This monoidal category is braided, using the braiding $\Psi_{\rG,\rH,V}$   defined by the composition
\newline
\begin{minipage}{\linewidth}
\begin{small}
\begin{align*}
\rG\rH(V)=\rG\rH(\one \otimes V)\xrightarrow{\rG(\rho^{\rH}_{\one,V})}\rG(\rH(\one)\otimes V)\xrightarrow{\lambda^{\rG}_{\rH(\one),V}}\rH(\one)\otimes \rG(V)\xrightarrow{\left(\rho^{\rH}_{\one,\rG(V)}\right)^{-1}}\rH(\one\otimes \rG(V))= \rH\rG(V).\end{align*}
\end{small}
\end{minipage}
The hexagon axioms of a braided monoidal category, cf. \cite{EGNO}*{Definition 8.1.1}, follow by the definition of $\lambda^{\rG\rH}$ and $\rho^{\rG\rH}$ as above.
There is a monoidal functor $\rF\colon \cZ_{\cB}(\cM)\to \cM$ mapping $\rG$ to $\rG(\one)$ which is monoidal.
An isomorphism $\mu^{\rF}$ is given by the composition of natural isomorphisms 
\begin{align}\label{centermonoidaliso}
\rG(\one )\otimes \rH(\one)\xrightarrow{\left(\lambda^{\rH}_{\rG(\one),\one}\right)^{-1}}\rH(\rG(\one) \otimes \one)= \rH\rG(\one).
\end{align}
The naturality square as in \cite{EGNO}*{Eq. (2.23)} necessary to make $(\rF,\mu)$ a monoidal functor follows from Eq. (\ref{modulemorph2}).

We will prove that $\cZ_{\cB}(\cM)$ is rigid provided that  $\cM$ is rigid after Proposition \ref{alternativedescription}.
\end{proof}

\begin{example}$~$
\begin{enumerate}
\item If $\cB=\Vect$, and $\cM$ has countable biproducts, as in Example \ref{trivexample}, then $\cZ_{\cB}(\cM)=\cZ(\cM)$ is equivalent to the monoidal center from \cite{Maj2} (in the case of the identity functor).
\item If $\cM$ is a braided $\cB$-augmented monoidal category as in Lemma \ref{augmentedbraided}, then there exists a functor of $\cB$-augmented monoidal categories $\cM\to \cZ_\cB(\cM)$ which is a right inverse to the forgetful functor. It is defined by mapping $M$ to the functor $M\otimes (-)$ of left tensoring by $M$, which is in the center using the braiding of $\cM$.
\end{enumerate}
\end{example}

\begin{remark}
Note that the definition of the relative monoidal center in \cite{Lau} is a slightly different one, which is, in general, not equivalent. There, the goal was to use a  set of assumptions in order to recover the category $\lYD{B}(\cB)$ in the special case where $\cM=\lmod{B}(\cB)$, cf. \cite{Lau}*{Proposition 2.4.7}. Here, the language of $\cB$-balanced functors is used to give a more natural construction with additional properties, for which the Morita dual can be identified, see Section \ref{Moritasection}.
\end{remark}

We now give  an alternative description of the relative monoidal center analogous to the original definition of the monoidal center of \cites{Maj2,JS},.

\begin{definition}\label{Isombraideddef}
Define the category $\Isom_{\cB}(\cM\otimes \ide_{\cM},\ide_{\cM}\otimes \cM)$ to consist of objects which are pairs $(V,c)$ where $V$ is an object of $\cM$ and $c\colon V\otimes \ide_{\cM}\natisomorph \ide_{\cM}\otimes V$ is a natural isomorphism satisfying that
\begin{enumerate}
\item[(i)]  for two objects $M,N$ of $\cM$, the diagram
\begin{align}\label{ctensorcomp}
\vcenter{\hbox{
\xymatrix{
V\otimes M\otimes N\ar[dr]_{c_M\otimes \ide_N}\ar[rr]^{c_{M\otimes N}}&&M\otimes N\otimes V\\
&X\otimes V\otimes Y\ar[ur]_{\ide_M\otimes c_{N}}&
}}}
\end{align}
commutes (tensor product compatibility); 
\item[(ii)]  Any object $(V,c)$ satisfies that for any object $X$ of $\cB$, we have
\begin{align}\label{sigmacomp}
c_{\rT(C)}=\sigma_{V,C}, &&\text{(compatibility with augmentation)}.
\end{align}
\end{enumerate}
A morphisms $\phi\colon (V,c)\to (W,d)$ in $\Isom_{\cB}(\cM\otimes \ide_{\cM},\ide_{\cM}\otimes \cM)$ corresponds to a morphism $\phi\colon V\to W$ in $\cM$ such that the diagram
\begin{align}\label{Isomhoms}
\vcenter{\hbox{
\xymatrix{
V\otimes M\ar[r]^{c_M}\ar[d]_{\phi\,\otimes M}&M\otimes V\ar[d]^{X\otimes \phi}\\
W\otimes M\ar[r]^{d_M}&M\otimes W
}}}
\end{align}
commutes for any object $M$ of $\cM$.
\end{definition}

\begin{proposition}
The category $\Isom_{\cB}(\cM\otimes \ide_{\cM},\ide_{\cM}\otimes \cM)$ is a braided monoidal  category.
\end{proposition}
\begin{proof} It follows directly that $\Isom_{\cB}(\cM\otimes \ide_{\cM},\ide_{\cM}\otimes \cM)$ is a full monoidal subcategory of the category $\Isom(\cM\otimes \ide_{\cM},\ide_{\cM}\otimes \cM)$, which is the monoidal center of \cite{Maj2}, as the compatibility of Eq. (\ref{sigmacomp}) with the augmentation is stable under composition and tensor products. 

We recall that the tensor product of $(V, c^V)$ and $(W, c^W)$ is defined as $(V\otimes W, c^{V\otimes W})$, where $c^{V\otimes W}$ is defined to make the diagram
\begin{align}\label{Isomtensor}
\vcenter{\hbox{
\xymatrix{
V\otimes W\otimes X\ar[dr]_{\ide\otimes c_X^W}\ar[rr]^{c^{V\otimes W}_{X}}&&X\otimes V\otimes W\\
&V\otimes X\otimes W\ar[ru]_{c_X^V\otimes \ide}&
}}}
\end{align} 
commute for any object $X$ of $\cM$. Note that Equation (\ref{sigmacomp}) for $(V\otimes W,c^{V\otimes W})$ translates to diagram in Eq. (\ref{augmenteddiag1}), which holds as $\cM$ is $\cB$-augmented monoidal.

Also recall that a braiding $\Psi_{(V,c^V), (W,c^W)}\colon (V\otimes W,c^{V\otimes W})\isomorph (W\otimes V,c^{W\otimes V})$ is defined by
$\Psi_{(V,c^V), (W,c^W)}=c^V_W.$ This is an isomorphism in $\Isom_{\cB}(\cM\otimes \ide_{\cM},\ide_{\cM}\otimes \cM)$ using $\otimes$-compatibility from Eq. (\ref{ctensorcomp}) and naturality of $c^V$ applied to $c^W_{X}$, for any object $X$ in $\cM$.
The braiding axioms also follow from Eq. (\ref{ctensorcomp}).
\end{proof}

\begin{proposition}\label{alternativedescription}
There is an equivalence of braided monoidal categories between $\cZ_{\cB}(\cM)$ and $\Isom_{\cB}(\cM\otimes \ide_{\cM},\ide_{\cM}\otimes \cM)$.
\end{proposition}
\begin{proof}
The equivalence can be obtained by restriction of the known equivalence of $\cZ(\cM)\simeq \Isom(\cM\otimes \ide_{\cM},\ide_{\cM}\otimes \cM)$, cf. e.g.  \cite{EGNO}*{Proposition 7.13.8}, to the subcategories $\cZ_{\cB}(\cM)$ and $\Isom_{\cB}(\cM\otimes \ide_{\cM},\ide_{\cM}\otimes \cM)$. We provide a sketch of the functors constituting the equivalence here:

Given an object $\phi\colon \cM \to \cM$ in $\cZ_{\cB}(\cM)$, consider the pair $(V_\phi,c^\phi)$, where $V_\phi=\phi(\one)$ and $c^\phi_X$, for an object $X$ of $\cM$, is given by the composition
\begin{align*}
c^\phi_M\colon  V_\phi\otimes M\xrightarrow{\rho^{-1}_{\one,M}} \phi(\one\otimes M)= \phi(M \otimes \one)\xrightarrow{\lambda_{M,\one}}M\otimes V_\phi.
\end{align*}
We have to check that $c^\phi$ thus defined satisfies the tensor compatibility condition from Eq. (\ref{ctensorcomp}). This follows from the diagrams in Eqs. (\ref{modulemorph1}), (\ref{modulemorph2}), and (\ref{modulemorph3}), which, as the isomorphisms $\chi, \xi,\zeta$ are identities for the regular bimodule (in the strict case), make the following diagram commute:
\begin{align*}
\vcenter{\hbox{
\xymatrix{
V_\phi\otimes M\otimes N\ar[rr]^{\rho^{-1}_{\one,M\otimes N}}\ar[d]_{\rho^{-1}_{\one,M}\otimes N}&&\phi(M\otimes N)\ar[rr]^{\lambda_{M\otimes N,\one}}\ar[rrd]_{\lambda_{M,N}}&& M\otimes N\otimes V_\phi\\
\phi(M)\otimes N\ar[rru]_{\rho^{-1}_{M,N}}\ar[rrd]_{\lambda_{M,\one}\otimes N}&&&&M\otimes \phi(N)\ar[u]_{M\otimes \lambda_{N,\one}}\\
&&M\otimes V_\phi\otimes N\ar[rru]_{M\otimes \rho^{-1}_{\one,N}}&&
}}}
\end{align*}
Further, Eq. (\ref{centeraugm}) implies that
\begin{align*}
\sigma_{\phi(\one),X}=\lambda_{\rT (X),\one}\rho^{-1}_{\one, \rT(X)}=c^\phi_{\rT(X)}, &&\forall X\in\cB.
\end{align*}
That is, the compatibility conditions of Eq. (\ref{sigmacomp}) hold.
This constructions extends to a $\Bbbk$-linear functor $\Upsilon$ from $\cZ_{\cB}(\cM)$ to $\Isom_{\cB}(\cM\otimes \ide_{\cM},\ide_{\cM}\otimes \cM)$ by sending a morphism $\theta\colon \phi\Rightarrow \psi$ to $\Upsilon(\theta)=\theta_{\one}$. 

It follows that $\Upsilon$ is a functor of monoidal categories, using the same structural isomorphism $\mu$ as in Eq. (\ref{centermonoidaliso}). By construction, this monoidal functor commutes with the braiding.

Conversely, given an object $(V,c)$ in $\Isom_{\cB}(\cM\otimes \ide_{\cM},\ide_{\cM}\otimes \cM)$, we construct a  morphism $\phi^{V}\colon \cM\to \cM$ of $\cB$-balanced $\cM$-bimodules by setting $\phi^V(M)=M\otimes V$. The structural isomorphisms needed for $\phi^V$ to be morphism  of $\cM$-bimodules are given by 
$\lambda=\ide$ and $\rho_{M,N}=M\otimes c_{N}^{-1}.$ It follows, by construction, that the functor $\phi^V$ is a morphism of $\cB$-balanced bimodules. In particular, it satisfies Eq. (\ref{centeraugm}): 
\begin{align*}
\sigma_{\phi^V(M),X}&=\sigma_{M\otimes V,B}\stackrel{\text{(\ref{augmenteddiag1})}}{=}(\sigma_{M,X}\otimes V)(M\otimes \sigma_{V,X})\stackrel{\text{(\ref{sigmacomp})}}{=}(\sigma_{M,X}\otimes V)(M\otimes c_{\rT(X)})\\
&=(\sigma_{M,X}\otimes V)\rho^{-1}_{M,\rT(X)}=\lambda^{-1}_{\rT(X),M}\phi^V(\sigma_{M,X})\rho^{-1}_{M,\rT(X)}.
\end{align*}
The assignment $(V,c)\mapsto \Phi(V,c):=\phi^V$ extends to a functor $\Phi$ by setting $\Phi(\alpha)_M:=M\otimes \alpha$, which gives a $2$-morphism of $\cB$-balanced bimodules. 

A direct verification shows that the functors $\Upsilon$ and $\Phi$ form an equivalence of categories as claimed.
\end{proof}

We can now prove that $\cZ_{\cB}(\cM)$ is rigid (i.e. has left duals, see e.g \cite{Maj1}*{Definition 9.3.1}), provided that $\cM$ is rigid. Equivalently, we can prove that $\Isom_{\cB}(\cM\otimes \ide_{\cM},\ide_{\cM}\otimes \cM)$ is rigid. Given an object $(V,c)$, the left dual is defined as $(V^*,c^*)$, where
\begin{align}
c^*_M=(\ev_{V}\otimes \ide_{M\otimes V^*})(\ide\otimes c_M^{-1}\otimes \ide)(\ide_{V^*\otimes M}\otimes \coev_V).
\end{align}
Here, $\coev_V\colon \Bbbk \to V\otimes V^*$, $\ev_V\colon V^*\otimes V\to \Bbbk$ are the structural maps making $V^*$ the left dual of $V$ (cf. \cite{Maj1}*{Definition 9.3.1}).
It follows that $(V^*,c^*)$ defines an object in $\Isom(\cM\otimes \ide_{\cM},\ide_{\cM}\otimes \cM)$, and we have to verify Eq. (\ref{sigmacomp}).
If $M=\rT(X)$, then $c_{\rT(X)}^{-1}=\sigma_{V,X}^{-1}$ and we have to show that $c^*_{\rT(X)}=\sigma_{V^*,X}$. This follows, using uniqueness of the inverse, from 
\begin{align*}
c_{\rT(X)}^*\sigma_{V^*,X}^{-1}&=\ide_{V^*\otimes \rT(X)}, & \sigma_{V^*,X}^{-1}c^*_{\rT(X)}&=\ide_{\rT(X)\otimes V^*},
\end{align*}
which is derived from Eq. (\ref{augmenteddiag1}) and naturality of $\sigma_{-,X}$ applied to $\ev_V$ and $\coev_V$.

\begin{example}\label{YDexample}
Let $B\in \BiAlg(\cB)$. We define the category $\lYD{B}(\cB)$ of \emph{Yetter--Drinfeld modules} over $B$ as having objects $V\in \lmod{B}(\cB)$ which, in addition, have a left $B$-comodule structure in $\cB$, such that the compatibility
\begin{equation}\label{YDcond}
\begin{split}
&(m_B\otimes a)(\ide_B\otimes \Psi_{B,B}\otimes \ide_V)(\Delta\otimes \delta)\\&=(m_B\otimes \ide_V)(\ide_B\otimes \Psi_{V,B})(\delta a\otimes \ide_V)(\ide_B\otimes \Psi_{B,V})(\Delta\otimes \ide_V).
\end{split}
\end{equation}
holds. In graphical calculus notation, that is,
\begin{equation}\label{ydpic}\vcenter{\hbox{
\begingroup%
  \makeatletter%
  \providecommand\color[2][]{%
    \errmessage{(Inkscape) Color is used for the text in Inkscape, but the package 'color.sty' is not loaded}%
    \renewcommand\color[2][]{}%
  }%
  \providecommand\transparent[1]{%
    \errmessage{(Inkscape) Transparency is used (non-zero) for the text in Inkscape, but the package 'transparent.sty' is not loaded}%
    \renewcommand\transparent[1]{}%
  }%
  \providecommand\rotatebox[2]{#2}%
  \ifx\svgwidth\undefined%
    \setlength{\unitlength}{83.4452175bp}%
    \ifx\svgscale\undefined%
      \relax%
    \else%
      \setlength{\unitlength}{\unitlength * \real{\svgscale}}%
    \fi%
  \else%
    \setlength{\unitlength}{\svgwidth}%
  \fi%
  \global\let\svgwidth\undefined%
  \global\let\svgscale\undefined%
  \makeatother%
  \begin{picture}(1,0.54781815)%
    \put(0,0.27579781){\color[rgb]{0,0,0}\makebox(0,0)[lb]{\smash{ }}}%
    \put(0,0){\includegraphics[width=\unitlength,page=1]{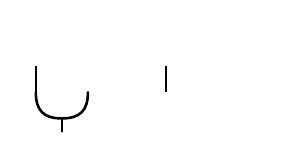}}%
    \put(0.64422368,0.25081165){\color[rgb]{0,0,0}\makebox(0,0)[lb]{\smash{$=$}}}%
    \put(0,0){\includegraphics[width=\unitlength,page=2]{ydcond.pdf}}%
    \put(0.44325035,0.26670016){\color[rgb]{0,0,0}\makebox(0,0)[lb]{\smash{$~$\\  }}}%
    \put(0,0){\includegraphics[width=\unitlength,page=3]{ydcond.pdf}}%
  \end{picture}%
\endgroup%
}}~.\end{equation}
Morphisms in  $\lYD{B}(\cB)$ are those commuting with both the left $B$-action and left $B$-coaction.
The category $\lYD{B}(\cB)$ is braided monoidal (see \cite{BD}*{Section 4} or \cite{Lau2}*{Section 2.1} for details). 
 \end{example}

For a Hopf algebra $H$ in $\cB$, a partial dual was defined and the category $\lYD{H}(\cB)$ was shown to be equivalent to YD modules over its partial dual in \cite{BLS}.

\begin{proposition}\label{YDprop}
Let $\cM$ be the $\cB$-augmented monoidal category $\cM=\lmod{B}(\cB)$ from Example \ref{bialgebraexpl}. Then there is an equivalence of braided monoidal categories
$$\cZ_{\cB}(\cM)\simeq \lYD{B}(\cB).$$
\end{proposition}
\begin{proof}
We will construct an equivalence of braided monoidal categories 
$$\Phi\colon \Isom_{\cB}(\cM\otimes \ide_{\cM},\ide_{\cM}\otimes \cM)\isomorph \lYD{B}(\cB).$$
Given an object $(V,c)$, define $\Phi(V,c)=V$ with its $B$-module structure $\triangleright \colon B\otimes V\to V$. The morphism
$$\delta_c:=c_{B}(\ide_V\otimes 1)\colon V\to B\otimes V,$$
where $B\in \lmod{B}(\cB)$ has the regular $B$-module structure, defines a $B$-coaction on $V$. Indeed,
\begin{align*}
(\ide_B\otimes \delta_c)\delta_c=c_{B\otimes B}(\ide_V\otimes 1\otimes 1)=c_{B\otimes B}(\ide_V\otimes \Delta_B 1)=(\Delta_B\otimes \ide_V)\delta_c,
\end{align*}
where the last equality uses that $\Delta_B$ is a morphism of $B$-modules, which holds by the bialgebra axiom from Eq. (\ref{bialgebraaxiom}).
Moreover, $V$ becomes a Yetter--Drinfeld module in $\cB$. This follows from $c_B$ being a morphism of $B$-modules, combined with the fact that the product map $m_B\colon B\otimes B^{\triv}\to B$ is a morphism of left $B$-modules:
\begin{align*}
&(m_B\otimes \triangleright)(\ide_B\otimes \Psi_{B,B}\otimes \ide_V)(\Delta_B\otimes \delta_c)\\&= c_B(\triangleright\otimes m_B)(\ide_B\otimes \Psi_{B,V}\otimes \ide_B)(\Delta_B\otimes \ide_V\otimes 1)\\
&= c_B(\triangleright\otimes m_B)(\ide_B\otimes \Psi_{B,V}\otimes 1)(\Delta_B\otimes \ide_V)\\
&=(m_B\otimes \ide_V)(c_{B\otimes B^{\triv}}\otimes \ide_B)((\ide\otimes 1)\triangleright\otimes \ide_V)(\ide_B\otimes \Psi_{B,V})(\Delta_B\otimes \ide_V)\\
&=(m_B\otimes \ide_V)(\ide_B\otimes c_{B^{\triv}})(\delta_c\otimes \ide_B)(\triangleright\otimes \ide_V)(\ide_B\otimes \Psi_{B,V})(\Delta_B\otimes \ide_V)\\
&=(m_B\otimes \ide_V)(\ide_B\otimes \Psi_{B,V})(\delta_c\otimes \ide_B)(\triangleright\otimes \ide_V)(\ide_B\otimes \Psi_{B,V})(\Delta_B\otimes \ide_V).
\end{align*}
The last equality uses that $\lmod{B}(\cB)$ is a $\cB$-augmented braided monoidal category as in Example \ref{bialgebraexpl}, with augmentation $c_{B^{\triv}}=\sigma_{V,B}=\Psi_{V,B}$.

A morphism in $\Isom_{\cB}(\cM\otimes \ide_{\cM},\ide_{\cM}\otimes \cM)$, in particular, commutes with the $B$-action and $\lambda_{B}$, and is hence a morphism in $\lYD{B}(\cB)$. Thus, we obtain a functor 
$\Phi$ as claimed. For two objects $(V,c)$ and $(W,d)$ in $\Isom_{\cB}(\cM\otimes \ide_{\cM},\ide_{\cM}\otimes \cM)$, the calculation 
\begin{align*}
\delta_{c\,\otimes d}&\stackrel{(\ref{Isomtensor})}{=}(c_{B} \otimes\ide_W)(\ide_V\otimes d_B)(\ide_{V\otimes W}\otimes 1)\\
&~~=~~ (c_{B} \otimes\ide_W)(\ide_V\otimes m_B\otimes\ide_W)(\ide_V\otimes 1\otimes d_B)(\ide_{V\otimes W}\otimes 1)\\
&~~=~~ (m_B\otimes \ide_{V\otimes W})(\ide_B\otimes \Psi_{V,B}\otimes \ide_W)(\delta_c\otimes \delta_d)
\end{align*}
shows that the functor $\Phi$ is monoidal. Here, we again use that the map $m_B\colon B\otimes B^{\triv}\to B$ is a morphism of left $B$-modules. Moreover, it follows that
\begin{align*}
\Psi_{(V,c),(W,d)}=c_W&=c_W(\ide_W\otimes \,\triangleright)(\ide_W\otimes 1\otimes \ide_V)\\
&=(\triangleright \otimes \ide_V)(\ide_V\otimes {c_{W^{\triv}}})(\delta_c\otimes\ide_W)\\
&=(\triangleright \otimes \ide_V)(\ide_V\otimes \Psi_{V,W})(\delta_c\otimes\ide_W)=\Psi^{\mathbf{YD}}_{V,W},
\end{align*}
using that $\triangleright\colon B\otimes W^{\triv}\to W$ is a morphism of left $B$-modules in $\cB$ and applying naturality to it.
Hence, $\Phi$ is a functor of braided monoidal categories. 

An inverse to $\Phi$ can be constructed using the last calculation which is valid even if $W$ is any $B$-module --- not necessarily coming from an object $(W,d)$. It shows that $c_W$ can be recovered from $\delta_c$. Hence, we obtain an equivalence of categories as stated.
\end{proof}

\begin{example}[\cite{BV}*{Prop. 2.13}]Let $\cC$ be a monoidal category and $H\in \Hopf(\cZ(\cC))$. Then one can define a monoidal structure on $\lmod{H}(\cC)$ using the half-braiding of $H$. There is an equivalence of braided monoidal categories between $\lYD{H}(\cZ(\cC))$ and $\cZ(\lmod{H}(\cC))$. 
Using Proposition \ref{YDprop} this implies that $\cZ_{\cZ(\cC)}(\lmod{H}(\cZ(\cC)))$ is equivalent to $\cZ(\lmod{H}(\cC))$.
\end{example}

\begin{lemma}
Let $B$ be a bialgebra in $\cB$ and consider the $\cB$-augmented monoidal category $\cM=\rcomod{B}(\cB)$ as in Example \ref{comoduleaugmented}. Then there is an equivalence of $\Bbbk$-linear braided monoidal categories $\cZ_{\cB}(\cM)\simeq \rYD{B}(\cB)$.
\end{lemma}
\begin{proof}
We can define a right $B$-module structure on an object $(V,c)$ of $\Isom_{\cB}(\cM\otimes \ide,\ide\otimes\cM)$ by $\triangleleft:=(\varepsilon\otimes \ide_V)c_B$ and proceed dually to the proof of Proposition \ref{YDprop}.
\end{proof}

In particular, if $B$ is a Hopf algebra in $\cB$, then the relative centers of $\lmod{\leftexp{\cop}{B}}(\overline{\cB})$ and $\rcomod{B}(\cB)$ are equivalent by \cite{Lau2}*{Lemma~2.5}.

Note that $\lYD{B}(\cB)$ is in general not $\cB$-augmented. The trivial Yetter--Drinfeld module structure can only be defined on objects of $\cB$ for which the braiding squares to the identity. However, this only holds for \emph{all} objects of $\cB$ if $\cB$ is symmetric monoidal. Hence the monoidal center $\cZ_\cB(\cM)$ for $\cM$ a $\cB$-augmented monoidal category is \emph{not} $\cB$-augmented, in general.

We further note that \cite{Gre2} provides a result in which a different object is referred to as the relative center --- the center of a bimodule category $\cV$ over $\cM$. It is shown there that this center is equivalent to $\cHom_{\cM\boxtimes\cM^{\oop}}(\cM,\cV)$. This category is naturally a categorical module over $\cZ(\cM)$, but not monoidal.


\subsection{Modules over the Relative Monoidal Center}\label{Moritasection}

We observe that there are two natural ways to produce categorical modules over $\cZ_\cB(\cM)$. The following Lemma summarizes the two constructions.

\begin{lemma}\label{centermodules}Let $\cM$ be a $\cB$-augmented monoidal category in $\Cat^c$.
There are $\Bbbk$-linear $2$-functors
\begin{align*}
\cHom_{\cM\text{--}\cM}(\cM,-)\colon \BiMod_{\cM\text{--}\cM}&\longrightarrow \lmod{\cZ_\cB(\cM)},\\
\cV &\longmapsto \cHom_{\cM\text{--}\cM}(\cM,\cV),\\
\cHom_{\cM\text{--}\cM}^{\cB}(\cM,-)\colon \BiMod^\cB_{\cM\text{--}\cM}&\longrightarrow \lmod{\cZ_\cB(\cM)},\\
\cW &\longmapsto \cHom_{\cM\text{--}\cM}^{\cB}(\cM,\cW).
\end{align*}
\end{lemma}
For both, the action of $\cZ_{\cB}(\cM)$ is given by pre-composition of functors.
In fact,  the containment
\begin{align}
\cHom_{\cM\text{--}\cM}^{\cB}(\cM,\cW)\subseteq \cHom_{\cM\text{--}\cM}(\cM,\left.\cW\right|_{\cM\text{--}\cM})
\end{align}
is one of $\cZ_{\cB}(\cM)$-modules. Here, we use the $\Bbbk$-linear $2$-functor  
\begin{align*}
\left.(-)\right|_{\cM\text{--} \cM} \colon\BiMod^\cB_{\cM\text{--}\cM}\longrightarrow\colon \BiMod_{\cM\text{--}\cM},
\end{align*}
where $\left.\cW\right|_{\cM\text{--}\cM}$ is $\cW$ as $\cM$-bimodule category, forgetting about the $\cB$-balanced structure. In fact, we can construct a  right $2$-adjoint to the $2$-functor $\left.(-)\right|_{\cM\text{--}\cM}$. Here, $2$-adjoint refers to an adjoint as a $\Cat^c$-enriched functor, in the sense of \cite{Kel}*{Section 1.11}. For this, we require the following construction, generalizing Definition \ref{Isombraideddef}.

\begin{definition}\label{Isomgeneral}
Let $\cM$, $\cN$ be  monoidal categories together with a monoidal functor $\rG\colon \cN \to \cM$. Further let $\cV$ be an $\cM$-bimodule. Define the category $\Isom(\cV\triangleleft \rG,\rG\triangleright \cV)$ to consist of objects $(V,\phi)$, where $V$ is an object of $\cV$, and $\phi\colon \cV\triangleleft \rG\natisomorph \rG\triangleright \cV$ is a natural isomorphism satisfying the condition that the diagram 
\begin{align}
\label{ctensorcomp2}
\vcenter{\hbox{
\xymatrix{
V\triangleleft \rG(X\otimes Y)\ar[d]_{V\triangleleft (\mu^G)^{-1}_{X,Y}}\ar[r]^{\phi_{X\otimes Y}}&\rG(X\otimes Y)\triangleright V \\
V\triangleleft \rG(X)\otimes \rG(Y)\ar[d]_{\xi_{V,X,Y}}&\rG(X)\otimes \rG(Y)\triangleright V\ar[u]_{\mu^G_{X,Y}\triangleright V}\\
(V\triangleleft \rG (X))\triangleleft \rG(Y)\ar[d]_{\phi_X\triangleleft \rG(Y)}&\rG(X)\triangleright (\rG(Y)\triangleright V)\ar[u]_{\chi^{-1}_{X,Y,V}}\\
(\rG(X)\triangleright V)\triangleleft \rG(Y)\ar[r]^{\zeta_{X,V,Y}}&\rG(X)\triangleright (V\triangleleft \rG(Y))\ar[u]_{\rG(X)\triangleright \phi_{Y}}
}}}
\end{align}
commutes for any objects $X,Y$ of $\cN$. A morphisms $f\colon (V,\phi)\to (W,\psi)$ in $\Isom(\cV\triangleleft \rT,\rT\triangleright \cV)$ is a morphism in $\cV$ such that
\begin{align}
\psi_{X}(f\otimes \rG (X))=(\rG (X)\otimes f)\phi_{X}.
\end{align}
\end{definition}

The above construction of $\Isom(\cV\triangleleft \rG,\rG\triangleright \cV)$ can be extended to a $2$-functor, turning $\cM$-bimodules into $\cB$-balanced bimodules.

\begin{proposition}\label{balancingfunctor}
The assignment $\cV\mapsto \cV_{\cB}:=\Isom(\cV\triangleleft \rT,\rT\triangleright \cV)$ extends to a $2$-functor $(-)_{\cB}\colon \BiMod_{\cM\text{--}\cM}\to \BiMod_{\cM\text{--}\cM}^\cB$.
\end{proposition}
\begin{proof}We construct $\cV_{\cB}:= \Isom(\cV\triangleleft \rT,\rT\triangleright \cV)$ as a $\cB$-balanced bimodule. Since $\cM$ is $\cB$-augmented, this category becomes a $\cM$-bimodule. The left action $X\triangleright (V,\phi)$ is the pair $(X\triangleright V,X\triangleright \phi)$, where $X\triangleright \phi$ is the composition
\begin{align*}
X\triangleright \phi=\chi(\sigma\triangleright V)\chi^{-1}(X\triangleright \phi)\zeta.
\end{align*}
The right action $(V,\phi)\triangleleft X$ is the pair $(V\triangleleft X,\phi\triangleleft X)$, where $\phi\triangleleft X$ is the composition
\begin{align*}
\phi\triangleleft X =\zeta(\phi\triangleleft X)\xi(V\triangleleft \sigma)\xi^{-1}.
\end{align*}
The structural isomorphisms $\chi, \xi, \zeta$ for $\cV_{\cB}$ can be taken to be the corresponding isomorphisms coming from the $\cM$-bimodule structure of $\cV$. We first have to verify that these are morphisms in $\cV_B$. This follows using $\otimes$-compatibility of $\sigma$ and $\phi$, combined with the coherence of $\cV$ as an $\cM$-bimodule. To demonstrate this idea, we prove the property that for any objects $M,N$ of $\cM$, $\chi_{M,N,V}$ is a morphism in $\cV_B$. That is, we have to verify commutativity of the diagram
\begin{align*}
\xymatrix{(M\otimes N\triangleright V)\triangleleft \rT X\ar[rr]^{(M\otimes N\triangleright \phi)_X}\ar[d]_{\chi_{M,N,V}\triangleleft \rT X} &&\rT X\triangleright (M\otimes N\triangleright V)\ar[d]^{\rT X\triangleright \chi_{M,N,V}}\\
(M\triangleright (N\triangleright V))\triangleleft \rT X \ar[rr]^{(M\triangleright (N\triangleright \phi))_X}&&\rT X\triangleright (M\triangleright (N\triangleright V)).
}
\end{align*}
This follows from the calculation
\begin{align*}
(\rT X&\triangleright \chi_{M,N,V})(M\otimes N\triangleright \phi)_X\\
=&(\rT X\triangleright \chi_{M,N,V})\chi_{\rT X,M\otimes N,V}(\sigma_{M\otimes N,X}\triangleright V)\chi^{-1}_{M\otimes N,\rT X,V}(M\otimes N\triangleright \phi_X)\zeta_{M\otimes N,V,\rT X}\\
=&\chi_{\rT X,M,N\triangleright V}\chi_{\rT X\otimes M,N,V}(\sigma_{M,X}\otimes N\triangleright V)(M\otimes \sigma_{N,X}\triangleright V)\chi^{-1}_{M\otimes N,\rT X,V}(M\otimes N\triangleright \phi_X)\zeta_{M\otimes N,V,\rT X}\\
=&\chi_{\rT X,M,N\triangleright V}(\sigma_{M,X}\triangleright (N\triangleright V))\chi^{-1}_{M,\rT X,N\triangleright V}(M\triangleright \chi_{\rT X,N,V})(M\triangleright (\sigma_{N,X}\triangleright V))\chi_{M,N\otimes \rT X,V}\\
&\chi^{-1}_{M\otimes N,\rT X,V}(M\otimes N\triangleright \phi_X)\zeta_{M\otimes N,V,\rT X}\\
=&\chi_{\rT X,M,N\triangleright V}(\sigma_{M,X}\triangleright (N\triangleright V))\chi^{-1}_{M,\rT X,N\triangleright V}(M\triangleright \chi_{\rT X,N,V})(M\triangleright (\sigma_{N,X}\triangleright V))(M\triangleright\chi^{-1}_{N,\rT X,V})\\&
(M\triangleright (N\triangleright\phi_X))(M\triangleright \zeta_{M,V,\rT X})\zeta_{M,N\triangleright V,\rT X}(\chi_{M,N,V}\triangleleft \rT X)\\
=&\chi_{\rT X,M,N\triangleright V}(\sigma_{M,X}\triangleright (N\triangleright V))\chi^{-1}_{M, \rT X, N\triangleright V}(M\triangleright (N\triangleright \phi)_X)\zeta_{M,N\triangleright V,\rT X}(\chi_{M,N,V}\triangleleft \rT X)\\
=&(M\triangleright (N\triangleright \phi))_X(\chi_{M,N,V}\triangleleft \rT X).
\end{align*}
The first equality is the definition of $(M\otimes N\triangleright \phi)$. In the second equality we use Eq. (\ref{augmenteddiag1}) and Eq. (\ref{modulecoherence}). 
The third equality applies naturality of $\chi$ to the morphism $\sigma_{M,X}$ in the second $\otimes$-component, and then uses naturality of $\chi$ applied to $\sigma_{N,X}$ (under application of Eq. (\ref{modulecoherence})). Again under application of the left module coherence and Eq. (\ref{bimodcoherence2}), the forth equality applies naturality of $\chi$ to $\phi_X$. The last two equalities follows by definition of $M\triangleright (N\triangleright \phi)$.


Next, we check that $\cV$ is a $\cB$-balanced $\cM$-bimodule. This follows using the natural isomorphism 
\begin{align*}
\beta_{(V,\phi),X}=\phi_{X}\colon V\triangleleft \rT X\isomorph \rT X\triangleright V.
\end{align*}
The isomorphism $\beta$ defined this way is, by Eq. (\ref{ctensorcomp2}), compatible with $\otimes$ as required in Eq. (\ref{betacoh1}). The compatibility condition (\ref{betacoh2}) holds by the way in which $M\triangleright (\phi\triangleleft N)$ is defined. However, we have to verify that $\beta$ is indeed a morphism in $\cV_{\cB}$. This follows from naturality of $\phi$ applied to $\rT(\Psi_{X,Y})$, using Eq. (\ref{sigmadesent}), and the $\otimes$-compatibility from Eq. (\ref{ctensorcomp2}). Indeed, 
\begin{align*}
&(\rT Y\triangleright \phi_X)(\phi\triangleleft \rT X)_Y
=(\rT Y\triangleright \phi_X)\zeta_{\rT Y,V,\rT X}(\phi_X\triangleleft \rT X)\xi_{V,\rT Y,\rT X}(V\triangleleft\sigma_{\rT X, Y})\xi^{-1}_{V,\rT X,\rT Y}\\
&=\chi_{\rT Y,\rT X,V}((\mu^{\rT}_{Y,X})^{-1}\triangleright V)\phi_{Y\otimes X}(V\triangleleft \mu_{Y,X}^{\rT})
(V\triangleleft\sigma_{\rT X, Y})\xi^{-1}_{V,\rT X,\rT Y}\\
&=\chi_{\rT Y,\rT X,V}
(\sigma_{\rT X, Y}\triangleright V)
(\left(\mu_{X,Y}^{\rT}\right)^{-1}\triangleright V)
\phi_{X\otimes Y}(V\triangleleft \mu^{\rT}_{X,Y})
\xi^{-1}_{V,\rT X,\rT Y}\\
&=\chi_{\rT Y,\rT X,V}
(\sigma_{\rT X, Y}\triangleright V)\chi^{-1}_{\rT X,\rT Y,V}(\rT X\triangleright \phi_Y)\zeta_{\rT X,V,\rT Y}(\phi_{V}\triangleleft \rT Y)\\
&=(\rT X\triangleright \phi)_Y(\phi_{X}\triangleleft \rT Y).
\end{align*}

Now assume given a morphism of $\cM$-bimodules $\rF\colon \cV\to \cW$. We can define an induced morphism $\rF_{\cB}\colon \cV_{\cB}\to \cW_{\cB}$ by declaring that an object $(V,\phi)$ be mapped to the pair $(\rF(V), \phi^{\rF(V)})$, where
$$
\phi^{\rF(V)}=\lambda_{\rT X,V}\rF(\phi_X)\rho^{-1}_{V,\rT X}.
$$
It follows that $\phi^{\rF(V)}$ satisfies the required tensor compatibility from Eq. (\ref{ctensorcomp2}) using that $\phi$ itself satisfies this compatibility, and that $\rF$ is a morphism of bimodules. Functoriality is also readily verified using that the composition of two morphisms of $\cM$-bimodules is again a morphism of  $\cM$-bimodules. Finally, $\rF_{\cB}$ is a morphism of $\cB$-balanced bimodules as by construction of $\phi^{\rF(V)}$ the diagram
\begin{align*}
\xymatrix{
\rF(V\triangleleft \rT X)\ar[rrrrr]^{\rF(\beta_{(V,\phi),X})=\rF(\phi_X)}\ar[d]^{\rho_{V,\rT X}}&&&&&\rF(\rT X\triangleright V)\ar[d]^{\lambda_{\rT X,V}}\\
\rF(V)\triangleleft \rT X\ar[rrrrr]^{\beta_{(\rF V,\phi^{\rF V}),X}=\lambda_{\rT X,V}\rF(\phi_X)\rho^{-1}_{V,\rT X}}&&&&&\rT X\triangleright \rF(V)\\
}
\end{align*}
commutes.
\end{proof}

The functor thus constructed provides a right $2$-adjoint to the functor $\left.(-)\right|_{\cM\text{--}\cM}$ of forgetting the $\cB$-augmented structure:

\begin{theorem}\label{adjunctionthm}
Given a $\cB$-balanced $\cM$-bimodule $\cW$ and a $\cM$-bimodule $\cV$, there exists is an isomorphism of categories in $\Cat^c$
\begin{equation*}
\cHom_{\cM\text{--}\cM}(\left.\cW\right|_{\cM\text{--}\cM},\cV)\cong \cHom_{\cM\text{--}\cM}^{\cB}(\cW,\cV_{\cB})
\end{equation*}
which is natural in $\cW$ and $\cV$.
\end{theorem}
\begin{proof}
Assume given a morphism $\rG\colon \left.\cW\right|_{\cM\text{--}\cM}\to \cV$ of $\cM$-bimodules. The bimodule $\cW$ is $\cB$-balanced. This means, in particular, that there exists natural isomorphisms $\beta_{W,X} \colon W\triangleleft \rT(X) \to \rT(X)\triangleright W$ for any objects $W$ of $\cW$ and $X$ of $\cB$. In order to construct a morphism of bimodules $\rG_{\cB}\colon \cW\to \cV_\cB$ we map an object $W\in \cW$ to the pair $(\rG(W),\phi^{\rG(W)})$, where $\phi^{\rG(W)}=\lambda\rG(\beta_{W,X})\rho^{-1}$. Indeed, $\phi^{\rG(W)}$ satisfies the $\otimes$-compatibility of Eq. (\ref{ctensorcomp2}), which follows, under application of naturality of $\lambda, \rho$ and $\mu^{\rT}$, from Eq. (\ref{betacoh1}). Functoriality of $\rG_\cB$ is clear by construction.

Next, we check that $\rG_{\cB}$ is a morphism of $\cB$-balanced $\cM$-bimodule. 
Indeed, the compatibility of $\beta$ which makes $\cW$ a $\cB$-balanced $\cM$-bimodules with the structural isomorphisms of $\cM$ as a $\cB$-augmented monoidal category imply that $\phi^{\rG(W)}$ satisfies $\otimes$-compatibility of Eq. (\ref{ctensorcomp2}). 
Conversely, given a morphism of $\cB$-balanced bimodules $\rH\colon \cW\to \cV_{\cB}$ we can restrict to a morphism 
$\left.\rH\right|_{\cM\text{--}\cM}\colon \left.\cW\right|_{\cM\text{--}\cM}\to \cV$. 

Starting with a morphism $\rH\colon \cW\to \cV_{\cB}$, we consider $\left(\left.\rH\right|_{\cM\text{--}\cM}\right)_{\cB}$, which maps an object $W$ to the pair $(\rH(W),\lambda\rH(\beta)\rho^{-1})$. But as $\rH$ is a morphism of $\cB$-balanced bimodules, we see that $$\lambda_{\rT X,W}\rH(\beta_{W,X})\rho^{-1}_{W,\rT(X)}=\left(\phi^{\rH(W)}\right)_{X},$$ and hence, $\left(\left.\rH\right|_{\cM\text{--}\cM}\right)_{\cB}=\rH$. 
On the other hand, given a morphism $\rG\colon \left.\cW\right|_{\cM\text{--}\cM}\to \cV$, consider the morphism of $\cM$-bimodules $\left.\rG_{\cB}\right|_{\cM\text{--}\cM}$. It is clear that $\left.\rG_{\cB}\right|_{\cM\text{--}\cM}(W)=\rG(W)$. This shows the claimed isomorphisms of $\Bbbk$-linear categories.

It remains to show that the isomorphism of categories in the statement of the theorem are natural in $\cW$ and $\cV$. To this end, let $\rG\colon \left.\cW\right|_{\cM\text{--}\cM}\to \cV$, $\rF\colon \cV\to \cV'$ be morphisms of $\cM$-bimodules, and $\rH\colon \cW\to \cW'$, $\rG'\colon \cW'\to \cV_{\cB}$ be a morphism of $\cB$-balanced $\cM$-bimodule morphisms. Consider the compositions
\begin{align*}
\begin{array}{ccc}
\vcenter{\hbox{\xymatrix{
\left.\cW\right|_{\cM\text{--}\cM}\ar[rr]^{\rG}\ar[rrd]_{\rF\rG}&&\cV\ar[d]^{\rF}\\
&&\cV'
}}}&,&
\vcenter{\hbox{\xymatrix{
\cW\ar[rr]^{\rG'\rH}\ar[d]^{\rH}&&\cV_{\cB}\\
\cW'\ar[rru]_{\rG'}&&
}}}.
\end{array}
\end{align*}
We want to show that $\rF_{\cB}\rG_{\cB}=(\rF\rG)_{\cB}$ and $\left.(\rG'\rH)\right|_{\cM\text{--}\cM}=\left.\rG'\right|_{\cM\text{--}\cM}\left.\rH\right|_{\cM\text{--}\cM}$. First, 
\begin{align*}
\rF_{\cB}\rG_{\cB}(W)=\rF_{\cB}(\rG (W), \phi^{\rG (W)})=(\rF\rG(W), \phi^{\rF\rG(W)})=(\rF\rG)_{\cB},
\end{align*}
which uses a very similar argument as in the proof of functoriality of $(-)_{\cB}$ in Proposition \ref{balancingfunctor}. Further, for a morphism $f$ in $\cW$, both sides evaluate to $\rF\rG(W)$. The second equality is clear. This concludes the proof of functoriality, and hence the proof of the theorem, as we observe that the adjunction is $\Cat^c$-enriched, in the sense of \cite{Kel}*{Section 1.11}, since we have an isomorphism of categories rather than a bijection of sets. 
\end{proof}

Under the biequivalence from Theorem \ref{bimoduletheorem}, we may reformulate the $2$-adjuction as a collection of \emph{equivalences} of categories
\begin{equation*}
\cHom_{\cM\boxtimes\cM^{\oop}}(\left.\cW\right|_{\rB},\cV)\simeq \cHom_{\cM\boxtimes_{\cB}\cM^{\oop}}(\cW,\cV_{\cB})
\end{equation*}
which is natural in $\cW$ and $\cV$. Here, $\cV$ is a $\cM\boxtimes\cM^{\oop}$-module and $\cW$ is a $\cM\boxtimes_{\cB} \cM^{\oop}$-module. The module $\left.\cW\right|_{\rB}$ is the restriction of the action along the universal functor $\rB\colon \cM\boxtimes \cM^{\\op}\to \cM\boxtimes_{\cB}\cM^{\oop}$.


As a consequence, we obtain a relationship between the two natural ways to introduce categorical modules over the relative monoidal center $\cZ_{\cB}(\cM)$ presented in Lemma \ref{centermodules}. 

\begin{corollary}
In the setup of Theorem \ref{adjunctionthm}, there are isomorphisms of $\cZ_{\cB}(\cM)$-modules
\begin{equation*}
\cHom_{\cM\text{--}\cM}(\cM,\cV)\cong \cHom_{\cM\text{--}\cM}^{\cB}(\cM,\cV_{\cB}),
\end{equation*}
which are natural in $\cV$.
\end{corollary}

\begin{example}
Let $B\in \BiAlg(\cB)$, where $\cB=\lmod{H}$, for $H$ a quasi-triangular Hopf algebra over $\Bbbk$. Then $\cM=\lmod{B}(\Vect)$ is a bimodule over $\cB$ using the induced bimodule structure. It follows that $\cM_{\cB}$ is equivalent to the category $\lmod{B}(\cB)$.
\end{example}

\begin{example}
Given functors of $\cB$-augmented monoidal categories $\rH_1,\rH_2\colon \cM\to \cN$, we can define the bimodule $\leftexp{\rH_1}{\cM}^{\rH_2}$.

A special case is to use the functor $\triv:=\rT\rF\colon \cM\to \cM$. We call the $\cZ_{\cB}(\cM)$-module $\cH_{\cB}(\cM):=\cHom_{\cM\text{--}\cM}^{\cB}(\cM, \leftexp{\reg}{\cM}^{\triv})$ the \emph{relative Hopf center} of $\cM$ over $\cB$. We will  see in Example \ref{Hopfexample} how this name is justified by a characterization using Hopf modules as in \cite{Lau}.
\end{example}

When working with \emph{finite} multitensor tensor categories \cite{EGNO}*{Sections 1.8, 4.1}, stronger statements can be derived.
It follows that the finite multitensor categories $\cZ_{\cB}(\cM)$ and $\cM\boxtimes_{\cB}\cM^{\oop}$ are \emph{categorically Morita equivalent}, in the terminology of \cite{EGNO}*{Section 7.12}, as their $2$-categories of modules are biequivalent.

\begin{corollary}\label{centermorita}
If $\cB$ and $\cM$ are finite $\Bbbk$-multitensor categories, then $\cHom_{\cM\boxtimes_{\cB}\cM^{\oop}}(\cM,-)$ is part of a biequivalence of $2$-categories
between $\lmod{\cM\boxtimes_{\cB}\cM^{\oop}}$ and $\lmod{\cZ_{\cB}(\cM)}$.
\end{corollary}
\begin{proof}The category $\cM\boxtimes_{\cB}\cM^{\oop}$ is a $\Bbbk$-multitensor category by Theorem \ref{relativeproductproperties} and finite by \cite{DSS}*{Theorem 3.3}. 
Further, with the construction of the functors in Lemma \ref{centermodules}, we have that $\cZ_{\cB}(\cM)$ is the image of the regular $\cB$-balanced bimodule $\cM$. This multitensor category is finite by \cite{EGNO}*{Proposition 7.11.6}. Now, the $\cM\boxtimes_{\cB}\cM^{\oop}$-module $\cM$ is faithful as no non-zero object acts by zero on it. Hence, by \cite{EGNO}*{Theorem 7.12.16}, the functor $\cHom_{\cM\boxtimes_{\cB}\cM^{\oop}}(\cM,-)$ is part of an equivalence of categories.
\end{proof}

This result is a version of \cite{EGNO}*{Proposition 7.13.8},  \cite{Ost}*{Section 2} for the relative monoidal center. 


\section{Representation-Theoretic Examples}\label{reptheorysect}

In this section, we apply the general results from the previous section to monoidal categories of representation-theoretic origin. This way, the results of this paper can, in particular, be applied to quantum groups $\Ug$ for generic $q$ (see Section \ref{qgroupssect}) and small quantum groups $u_\epsilon(\fr{g})$ (see Section \ref{smallqgroupssect}), where $\epsilon$ is a root of unity. 


\subsection{Yetter--Drinfeld Module Tensor Actions}\label{YDactionsect}

Let $\cB$ be a braided monoidal category with braiding $\Psi$. 
The category $\lmod{B}(\cB)$ (or $\lcomod{B}(\cB)$) of modules (respectively, comodules) over a bialgebra $B$ in $\cB$ is a monoidal category. Hence, we can again consider (co)algebra objects in this category.

Recall Proposition \ref{YDprop} which states that the relative center $\cZ_{\cB}(\cM)$ is equivalent as a braided monoidal category to the category of Yetter--Drinfeld modules $\lYD{B}(\cB)$. The constructions in Section \ref{catsect} lead to the following statement, for which a more direct proof was given in \cite{Lau2}*{Theorem~2.3}. We give an alternative proof here.

\begin{corollary}\label{mainthm}
Let $B$ be a bialgebra in $\cB$, $\cM=\lmod{B}(\cB)$ and $A$ an algebra in $\lcomod{B}(\cB)$. There is a left categorical action of the center $\cZ_{\cB}(\cM)\simeq \lYD{B}(\cB)$:
\begin{align*}
\triangleright\colon \lYD{B}(\cB)\boxtimes \lmod{A}(\lcomod{B}(\cB))&\longrightarrow \lmod{A}(\lcomod{B}(\cB)).
\end{align*}
An object $(V,a_V,\delta_V)$ of $\lYD{B}(\cB)$ acts on an object $(W,a_W,\delta_W)$ of $\lmod{A}(\lcomod{B}(\cB))$ by
\begin{gather}
(V,a_V,\delta_V)\triangleright (W,a_W,\delta_W):=(V\otimes W,a_{V\triangleright W}, \delta_{V\triangleright W}),\\
a_{V\triangleright W}:=(a_V\otimes a_W)(\ide_B\otimes \Psi_{A,V}\otimes \ide_W)(\delta_A\otimes \ide_{V\otimes W}),\\
\delta_{V\triangleright W}:=(m_B\otimes \ide_{V\otimes W})(\ide_B\otimes \Psi_{V,B}\otimes \ide_W)(\delta_V\otimes \delta_W).
\end{gather}
\end{corollary}
\begin{proof}
We note that the result is a special case of Lemma \ref{centermodules}, under the equivalence from Proposition \ref{YDprop}. For this, we consider the category $\cV=\lmod{A}(\cB)$. Note that these are \emph{not} modules in $\lcomod{B}(\cB)$. We now give $\cV$ a $\cB$-balanced $\cM$-bimodule structure. First, there is the trivial right action, just given by tensoring in $\cB$, with the $A$-action given on the $A$-module factor (the first tensor factor) only. Second, we can obtain a left action of $\cM$ on $\cV$ as follows: Given a left $A$-module $(V, \triangleright_V)$ and a left $B$-module $(W, \triangleright_W)$, we can define a left $A$-module structure on $W\otimes V$ using the $A$-action 
\begin{align}
\triangleright_{W\otimes V}&=(\triangleright_W\otimes \triangleright_V)(\ide_B\otimes \Psi_{A,W}\otimes \ide_V)(\delta_A\otimes \ide_{W\otimes V}).
\end{align}
We now claim that the left $\cZ_{\cB}(\cM)$-modules $\cHom_{\cM\text{--}\cM}^{\cB}(\cM,\cV)$ and $\lmod{A}(\lcomod{B}(\cB))$ are equivalent. Indeed, a $\cM$-bimodule functor $\phi$ is uniquely determined by the image $V=\phi(\one)$ of $\one$, together with a natural isomorphism $c\colon V\otimes \rF\natisomorph \rF\otimes V$ satisfying $\otimes$-compatibility, and compatibility with the augmentation, from Definition \ref{Isomgeneral}. Proceeding similarly to \ref{alternativedescription}, we see that there is an equivalence of $\cB$-balanced bimodules 
$$\cHom_{\cM\text{--}\cM}^{\cB}(\cM,\cV)\simeq \Isom_{\cB}(\cV\triangleleft \rF,\rF\triangleright \cV)=\cV_{\cB}.$$
We now proceed to show that $\Isom_{\cB}(\cV\triangleleft \rF,\rF\triangleright \cV)\simeq \lmod{A}(\lcomod{B}(\cB))$ with a similar argument as used in Proposition \ref{YDprop}.
The morphism $c_{B}$, where $B$ is the regular $B$-module, can be pre-composed with $\ide_V\otimes 1$, where $1$ is the unit of $B$, to give a morphism $\delta\colon V\rightarrow B\otimes V$ in $\cB$. Tensor compatibility in $\cM$ implies that $\delta$ makes $V$ a $B$-comodule. Now, using that $\phi$ is a $\cB$-balanced natural transformation, we find that $V$ is indeed an object in $\lmod{A}(\lcomod{B}(\cB))$. This follows using the fact that $\delta$ is a morphism of $A$-modules, with the induced $A$-module structure on $B\otimes V$.
Up to isomorphism, $\phi$ can be recovered from $\delta$ using that for any $B$-module $W$, the $B$-action is a morphism of $B$-modules $B\otimes W^{\triv}\to W$. Now, using $\otimes$-compatibility, and that $c_{W^{\triv}}=\rT(\Psi_{V,W})$ we see that $c_W=(\triangleright_W\otimes \ide_V)(\ide_B\otimes \Psi_{V,W})(\delta\otimes \ide_W)$.
\end{proof}

We observe that the underlying $B$-comodule structure in the above categorical action of the center is given by the monoidal structure in $\lcomod{B}(\cB)$, where the $A$-action is a generalization of the induced action (which is recovered in the case where $H=\Bbbk$ was simply a field and $\cB=\Vect$). Hence the result of Theorem \ref{mainthm} can be interpreted as a natural generalization of the induced action to comodule algebras, which requires the use of the monoidal center as the category $\lcomod{B}(\cB)$ does not act on $\lmod{A}(\lcomod{B}(\cB))$ in a similar, non-trivial way. 

\begin{example}\label{Hopfexample}$~$
\begin{enumerate}
\item[(i)]
Note that $B$ is always a comodule algebra $B^{\op{reg}}$ with respect to the regular coaction given by $\Delta$. This case gives the category $\lmod{B}(\lcomod{B}(\cB))$, which is also known as the category $\leftexpsub{B}{B}{\mathbf{H}}(\cB)$ of \emph{Hopf modules} over $B$ in $\cB$. A left Hopf module is an object $V$ in $\cB$ which is both  a left $B$-module, with action $\triangleright$, and left a $B$-comodule, with coaction $\delta$, such that the structures satisfy the compatibility condition 
\begin{align}
\delta \triangleright&= (m_B\otimes \triangleright)(\ide_B\otimes \Psi_{B,B}\otimes \ide_V)(\Delta\otimes \delta),&&\Leftrightarrow&& \vcenter{\hbox{
\begingroup%
  \makeatletter%
  \providecommand\color[2][]{%
    \errmessage{(Inkscape) Color is used for the text in Inkscape, but the package 'color.sty' is not loaded}%
    \renewcommand\color[2][]{}%
  }%
  \providecommand\transparent[1]{%
    \errmessage{(Inkscape) Transparency is used (non-zero) for the text in Inkscape, but the package 'transparent.sty' is not loaded}%
    \renewcommand\transparent[1]{}%
  }%
  \providecommand\rotatebox[2]{#2}%
  \ifx\svgwidth\undefined%
    \setlength{\unitlength}{59.04203141bp}%
    \ifx\svgscale\undefined%
      \relax%
    \else%
      \setlength{\unitlength}{\unitlength * \real{\svgscale}}%
    \fi%
  \else%
    \setlength{\unitlength}{\svgwidth}%
  \fi%
  \global\let\svgwidth\undefined%
  \global\let\svgscale\undefined%
  \makeatother%
  \begin{picture}(1,0.52112539)%
    \put(0.27628689,0.26289578){\color[rgb]{0,0,0}\makebox(0,0)[lb]{\smash{ }}}%
    \put(0,0){\includegraphics[width=\unitlength,page=1]{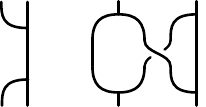}}%
    \put(0.1968943,0.22893528){\color[rgb]{0,0,0}\makebox(0,0)[lb]{\smash{$=$}}}%
    \put(0.7939266,0.19711312){\color[rgb]{0,0,0}\makebox(0,0)[lb]{\smash{$~$}}}%
  \end{picture}%
\endgroup%
}}.
\end{align}
 In \cite{Lau}, the category of Hopf modules is described by a general categorical construction, called the \emph{Hopf center}. The result of a categorical action of the monoidal (or Drinfeld center) on the Hopf center is hence a special case of Theorem \ref{mainthm}. Such a result was first proved in \cite{Lu} in the case of finite-dimensional Hopf algebras over a field $\Bbbk$.
\item[(ii)]We can also view $B$ as a comodule algebra $B^{\triv}$ with respect to the trivial coaction (via the counit $\varepsilon$) on itself. In this case, the category $\lmod{B^{\triv}}(\lmod{B}(\cB))$ consist of simultaneous left $B$-modules and comodules such that the structures commute, i.e. an object $(V,\triangleright,\delta)$ with action $\triangleright$ and coaction $\delta$ satisfying
\begin{align}
\delta \triangleright&=(\ide_B\otimes \triangleright)(\Psi_{B,B}\otimes \ide_V)(\ide_B\otimes \delta).
\end{align}
In this case, there is an action of $\lcomod{B}(\cB)$ on $\lmod{B^{\triv}}(\lmod{B}(\cB))$ and the action of the center factors through this action via the forgetful functor $\lYD{B}(\cB)\to \lcomod{B}(\cB)$.
\end{enumerate}
\end{example}

We can also provide a version of Corollary \ref{mainthm} for module algebras.

\begin{corollary}\label{mainthm2}
Let $B$ a bialgebra in $\cB$, $\cM=\rmod{B}(\cB)$ and $A$ an algebra in $\rmod{B}(\cB)$. There is a left categorical action of the center $\cZ_{\cB}(\cM)\simeq \rYD{B}(\cB)$:
\begin{align*}
\triangleright\colon \rYD{B}(\cB)\boxtimes \lmod{A}(\rmod{B}(\cB))&\longrightarrow \lmod{A}(\rmod{B}(\cB)).
\end{align*}
An object $(V,a_V,\delta_V)$ of $\rYD{B}(\cB)$ acts on an object $(W,a_W,b_W)$ of $\lmod{A}(\rmod{B}(\cB))$, where $a_W\colon A\otimes W\to W$ and $b_W\colon W\otimes B\to W$ are $A$-, respectively $B$-actions, by
\begin{gather}
(V,a_V,\delta_V)\triangleright (W,a_W,b_W):=(V\otimes W,a_{V\triangleright W},b_{V\triangleright W}),\\
a_{V\triangleright W}:=(\ide_{V}\otimes a_W)(\ide_{V}\otimes b_A\otimes \ide_W)(\Psi_{A,V}\otimes \ide_{B\otimes W})(\ide_A\otimes \delta_V\otimes \ide_W),\\
b_{V\triangleright W}:=(b_V\otimes b_W)(\ide_V\otimes \Psi_{W,B}\otimes \ide_B)(\ide_{V\otimes W}\otimes \Delta_B).
\end{gather}
\end{corollary}
\begin{proof}
First note that $a_{V\triangleright W}$ can be used to define a left action of $\cM=\rcomod{B}(\cB)$ on $\cV=\lmod{A}(\cB)$. Similarly to the proof of Corollary \ref{mainthm}, one can show --- by making $\cV$ a $\cB$-balanced $\cM$-bimodule by using the trivial $\cM$-action on the right --- that $\cHom_{\cM\text{--}\cM}(\cM,\cV)$ is equivalent to $\lmod{A}(\rmod{B}(\cB))$. This way, the statement is a corollary of Lemma \ref{centermodules}. For a direct proof, see \cite{Lau2}*{Theorem 2.4}.
\end{proof}

\subsection{Applications to  Braided Drinfeld Doubles}\label{drinapp}

In this section, apply the results of this paper to categories of modules over braided Drinfeld doubles, and explain connections to the construction of comodule algebras and $2$-cocycles over these double from \cite{Lau2}.
For this, let $\cB=\lmod{H}$ for a quasi-triangular $\Bbbk$-Hopf algebra $H$ throughout this section. 

To define the braided Drinfeld double $\Drin_H(C,B)$ (due to \cites{Maj1, Maj3}, where this construction is called the  \emph{double bosonization}), we let $B$ and $C$ be braided bialgebras $\cB$ with a non-degenerate Hopf algebra pairing $\ev\colon C\otimes B\to \Bbbk$. Then $\Drin_H(C,B)$ is a Hopf algebra such that there is a fully faithful monoidal functor
$$\lYD{B}(\cB)\longrightarrow \lmod{\Drin_H(C,B)}.$$
It induces an equivalence $\lYD{B}(\cB)\simeq \lmod{\Drin_H(C,B)}^{C\mathrm{-lf}}$ to the category of modules with \emph{locally finite} $C$-action, i.e. modules $V$ where $\dim(C\triangleright v)<\infty$ for any $v\in V$.
In the case where $H=\Bbbk$, and $B$ is a finite-dimensional Hopf algebra with $C=B^*$ the dual Hopf algebra, this recovers a version of the \emph{quantum double} $\Drin(B)$ of \cite{Dri}, and $\lYD{B}\simeq \lmod{\Drin(B)}$ is an equivalence of braided monoidal categories (see e.g. \cite{Kas}*{Section IX.5} for a direct proof).

In the finite-dimensional case, Theorem \ref{bimoduletheorem} and Corollary \ref{centermorita} provide a new characterization of categorical modules. For this, we use the superscript ${\mathrm{fd}}$ to restrict to finite-dimensional objects.

\begin{corollary}\label{drincor}
If $C,B$ are dually paired Hopf algebras in $\cB=\lmod{H}$, then there is a bifunctor from categorical modules over $\lrmod{B}{B}(\cB)$ to categorical modules over $\lmod{\Drin_H(C,B)}^{C\mathrm{-lf}}$.

If $B$ is a finite-dimensional Hopf algebra in $\cB$ and $B^*$ its dual, then this bifunctor is a biequivalence, i.e. the categories $\lmod{\Drin_H(B^*,B)}^{\mathrm{fd}}$ and $\lrfmod{B}{B}(\cB)$ are categorically Morita equivalent. 
\end{corollary}

We can further described $\lrmod{B}{B}(\cB)$ as modules over a Hopf algebra.

\begin{definition} \label{BBHdef}Let $B$ be a Hopf algebra in $\cB$. We can define a Hopf algebra structure on $B_+\otimes H\otimes B_-$, where $B_+$ and $B_-$ are two copies of the algebra $B$. The algebra structure is given by requiring that $B_+$, $B_-$, and $H$ are subalgebras, such that
\begin{gather}
bd=(R^{(2)}\triangleright d)(R^{(1)}\triangleright b),\qquad hb=(h_{(1)}\triangleright b)h_{(2)},\qquad h d=(h_{(1)}\triangleright d)h_{(2)},
\end{gather}
for $b\in B_+$, $d\in B_-$, and $h\in H$. The coproduct is given on generators by 
\begin{gather}
\Delta(h)=\Delta_H(h), \quad \Delta(b)=b_{(1)}R^{(2)}\otimes (R^{(1)}\triangleright b_{(2)}), \quad \Delta(d)=R^{-(1)}d_{(2)}\otimes (R^{-(2)}\triangleright d_{(1)}).
\end{gather}
Here $b_{(1)}\otimes b_{(2)}$ and $d_{(1)}\otimes d_{(2)}$ denote the coproduct of $B$.
The counit is given by $\varepsilon_B\otimes \varepsilon_H\otimes \varepsilon_B$. The antipode $S$ is given by
\begin{align}S(h)=S_H(h), \qquad
S(b)=S(R^{(2)}(R^{(1)}\triangleright S_B(b)), \qquad S(d)=S(R^{-(1)})(R^{(-2)}\triangleright S^{-1}_B(d)).
\end{align}
The resulting Hopf algebra is denoted by $(B\otimes \leftexp{\cop}{B}) \rtimes H$.
\end{definition}

He note that if $B$ is a self-dual Hopf algebra in $\cB$ (for example, a \emph{Nichols algebra} \cite{AS}), then $(B\otimes \leftexp{\cop}{B}) \rtimes H$ has the same coalgebra structure as $\Drin_H(C,B)$, but the algebra structure is easier.

\begin{proposition}\label{BBHprop}
Let $B$ be a Hopf algebra in $\cB$. Then there is an equivalence of monoidal categories 
$\lrmod{B}{B}(\cB)\simeq \lmod{(B\otimes \leftexp{\cop}{B}) \rtimes H}$.
\end{proposition}
\begin{proof}
Assume given an object $V$ in $\lrmod{B}{B}(\cB)$ with left actions of $B$ and $H$ denoted by $\triangleright$ and right $B$-action $\triangleleft$. Then a left  action $\ov{\triangleright}$ of $(B\otimes \leftexp{\cop}{B}) \rtimes H$ is defined by 
\begin{gather*}
h\,\ov{\triangleright}\,v=h\triangleright v,\qquad
b\,\ov{\triangleright}\,v=b\triangleright v, \qquad d\,\ov{\triangleright}\,v=(R^{-(1)}\triangleright v)\triangleleft S^{-1}_B(R^{-(2)}\triangleright d),
\end{gather*}
for $v\in V$, $b\in B_+$, $c\in B_-$, and $h\in H$. We note that the right action of $B_-$ is obtained from the left $B$ action through the functor $\Phi^{-1}$ in Lemma \ref{oppositeaugmented}.
\end{proof}

Further, the results of this paper relate to the constructions of comodule algebras over the braided Drinfeld double in \cite{Lau2}.
Let $A$ be a left $B$-comodule algebra in $\cB=\lmod{H}$. It was shown in \cite{Lau2}*{Corollary~3.8} that the braided crossed product $A\rtimes \leftexp{\cop}{C}\rtimes H$ is a $\Drin_H(C,B)$-comodule algebra. Therefore, there is a categorical action of $\lmod{\Drin_H(C,B)}$ on $\lmod{A\rtimes \leftexp{\cop}{C}\rtimes H}$, such that the fully faithful functor (obtained by using the pairing $\ev$)
$$\lmod{A}(\lcomod{B}(\cB))\longrightarrow\lmod{A\rtimes \leftexp{\cop}{C}\rtimes H}$$
is one of $\cZ_{\cB}(\cM)$-modules. The action of the former is given by Corollary \ref{mainthm}, and the action on the latter is given by restricting the $\lmod{\Drin_H(C,B})$-action to the subcategory $\lYD{B}(\cB)\simeq \cZ_{\cB}(\cM)$.

As examples for the braided crossed product, we may consider the braided Heisenberg double $\Heis_H(C,B)$ or the twisted tensor product algebra $B\otimes_{\Psi^{-1}}C\rtimes H$ (see \cite{Lau2}*{Section 3.3} for details).

Related to this, there exists two maps of spaces of left $2$-cocycles in non-abelian cohomology \cite{Lau2}*{Section 4.4}:
\begin{align*}
\op{Ind}_B\colon &H_H^2(B,\Bbbk)\longrightarrow H^2(\Drin_{H}(C,B),\Bbbk), &
\op{Ind}_C\colon &H_H^2(\leftexp{\cop}{C},\Bbbk)\longrightarrow H^2(\Drin_{H}(C,B),\Bbbk).
\end{align*}
It follows that if an algebra $A$ is a $2$-cocycle twist of the bialgebra $B$ in $\cB$ by a cocycle $\tau$, then $A\rtimes \leftexp{\cop}{C}\rtimes H$ is a $2$-cocycle twist by $\op{Ind}_B(\tau)$, cf. \cite{Lau}*{Proposition 3.8.4}. As a special case, $\Heis_H(C,B)$ is a $2$-cocycle twist of $\Drin_H(C,B)$ by the $2$-cocycle $\op{Ind}_B(\triv)$ for the trivial $2$-cocycle of $B$. This result generalizes the theorem from \cite{Lu} for finite-dimensional $\Bbbk$-Hopf algebras.

\subsection{Applications to Generic Quantum Groups}\label{qgroupssect}

The construction of $\cZ_{\cB}(\cM)$ is motivated by the representation theory of quantum groups $U_q(\mathfrak{g})$ as discussed in the introduction, see Section \ref{motivation}. We will discuss first applications to this setup. A more detailed study of representation-theoretic applications in this important example will appear elsewhere.

For this section, denote by $\mF$ the field $\Bbbk(q)$ for a generic variable $q$ over $\Bbbk$.
Following Lusztig \cite{Lus} fix a \emph{Cartan datum} $I$. That is, a finite index set $I$ together with a symmetric bilinear form $\cdot$ on the free abelian group $L:=\mZ \langle I\rangle =\mZ\langle g_i\mid i\in I\rangle$, such that $i\cdot i$ is even, and 
$a_{ij}:=2\tfrac{i\cdot j}{i\cdot i}\in \mZ_{\leq 0},$ for all $i\neq j.$
We can use this datum to define a dual $R$-matrix on the category $\lcomod{L}$ by
$
R(g_i,g_j)=q^{i\cdot j}.
$
The braided monoidal category thus obtained is denoted by $\lcomod{L}_q$. 

We now define $E:=\mF\langle e_i\mid i\in I\rangle$ to be the left $L$-comodule where $\delta(e_i)=g_i\otimes e_i$, and $F:=\mF\langle f_i\mid i\in I\rangle$ the left dual $L$-comodule, i.e. $\delta(f_i)=g_{i}^{-1}\otimes f_i$. A duality pairing can be given by
\begin{align}\langle e_i,f_j\rangle=\frac{\delta_{i,j}}{q_i-q_i^{-1}},\end{align}
where $q_i:=q^{i\cdot i/2}$.
Using the dual $R$-matrix, $E$ and $F$ become dually paired Yetter--Drinfeld modules over the lattice $L$. The $L$-actions are given by
\begin{align}
g_i\cdot e_j&= q^{i\cdot j} e_j, &g_i\cdot f_j&= q^{-i\cdot j} f_j.
\end{align}
Hence, we can consider the Nichols algebras (or Nichols--Woronowicz algebras) $U_q(\mathfrak{n}_+):=\cB(E)$, $U_q(\mathfrak{n}_-):=\cB(F)$. See e.g \cite{AS} for details on this construction. We note that $U_q(\mathfrak{n}_+)$ and $U_q(\mathfrak{n}_-)$ are dually paired braided Hopf algebras in $\lcomod{L}_q$ which are primitively generated. That is,
\begin{align}
\Delta (e_i)&=e_i\otimes 1+1\otimes e_i, &\Delta (f_i)&=f_i\otimes 1+1\otimes f_i.
\end{align}
The pairing of the braided Hopf algebra is the unique extension $\ev$ of the pairing of $E$ and $F$ to one of braided Hopf algebras (see \cite{Lus}*{Proposition 1.2.3}).
It is a theorem of \cite{Maj3} that there exists an isomorphism of Hopf algebras between $\Drin_H(U_q(\mathfrak{n}_-),U_q(\mathfrak{n}_+))$ and $U_q(\fr{g})$, where $\fr{g}$ denotes the semi-simple Lie algebra corresponding to the Cartan datum $I$, and $H=U_q(\fr{t})$ is a group algebra generated by $K_i$ for $i\in I$ (cf. also \cite{Lau2}*{Theorem 3.25}).\footnote{In \cite{Maj3} and \cite{Lau2}, $u_q(\fr{n}_-)$ is the \emph{left} dual of $u_q(\fr{n}_+)$. In the present paper, we reverse the roles so that $u_q(\fr{n}_+)$ acts locally finitely in Theorem \ref{quantumcenter}.}

Here, we use the following version of the quantum group $U_q(\mathfrak{g})$. The algebra $U_q(\mathfrak{g})$ is generated by $E_i,F_i,K_i^{\pm 1}$, for $i\in I$, subject to relations
\begin{gather}
K_iE_j=q^{i\cdot j}E_jK_i, \quad K_iF_j=q^{-i\cdot j}F_jK_i,\quad K_i^{\pm 1}K_i^{\mp 1}=1,\quad [E_i,F_j]=\delta_{i,j}\frac{K_i-K_i^{-1}}{q_i-q_i^{-1}},\\
\sum_{k=0}^{1-a_{ij}}(-1)^k\binom{1-a_{ij}}{k}_{q_i}E^{1-a_{ij}-k}_i E_jE_i^k=0,\qquad \sum_{k=0}^{1-a_{ij}}(-1)^k\binom{1-a_{ij}}{k}_{q_i}F^{1-a_{ij}-k}_i F_jF_i^k=0,
\label{qserrerel}
\end{gather}
for $i\neq j\in I$. Here, $\binom{n}{m}_q=
\frac{[n]_q!}{[m]_q![n-m]_q!}$, for $n\geq m$, where $[n]_q=\frac{q^n-q^{-n}}{q-q^{-1}}$. We use the coproduct 
\begin{gather}
\Delta(K_i)=K_i\otimes K_i, \qquad \Delta(E_i)=E_i\otimes K_i+1\otimes E_i,\qquad \Delta(F_i)=F_i\otimes 1+K^{-1}_i\otimes F_i.
\end{gather}
This version of the quantum group appears e.g. in \cite{CP1}*{Section~9.1}. It is the coopposite Hopf algebra of the quantum group of \cite{J}*{Chapter~4}.

\begin{theorem}\label{quantumcenter}
Let $\cB=\lcomod{L}_q$ and $\cM=\lmod{U_q(\mathfrak{n}_-)}(\cB)$. Then the category $\cZ_{\cB}(\cM)$ is equivalent to the full subcategory of $U_q(\fr{g})$-modules $V$ s.t.
\begin{enumerate}
\item[(i)] $V$ is a \emph{weight module}, i.e. $V=\bigoplus_{i\in L}V_i$, where $K_j\cdot v_i=q^{i\cdot j}v_i$ for any $v_i\in V_i$;
\item[(ii)] $V$ is \emph{locally finite} over $U_q(\fr{n}_+)$, i.e. $\dim(U_q(\fr{n}_+)\cdot v)<\infty$ for any $v\in V$.
\end{enumerate}
That is, $V$ is a $U_q(\mathfrak{n}_+)$-locally finite weight module over $\Ug$. We denote the full subcategory of $\lmod{\Ug}$ of such modules by $\lmod{\Ug}^{U_q(\fr{n}_+)\mathrm{-lfw}}$.
\end{theorem}
\begin{proof}
Recall that $\cZ_\cB(\cC)\simeq \lYD{U_q(\mathfrak{n}_-)}(\cB)$ by Proposition \ref{YDprop}. Given a Yetter--Drinfeld module over $U_q(\fr{n}_-)$ with action $f_i\cdot v$, and coaction $\delta(v)=v^{(-1)}\otimes v^{(0)}$ for $v\in V$, we define a $U_q(\fr{g})$ action by
\begin{gather*}
K_i\cdot v=q^{i\cdot |v|}v, \qquad F_i\cdot v=f_i\cdot v, \qquad E_i\cdot v=q^{-i\cdot |v^{(0)}| }\ev(e_i\otimes v^{(-1)})\cdot v^{(0)},
\end{gather*}
where $v\mapsto |v|\otimes v$ is the $L$-grading of $V$. The action of $F_i$ is locally finite by construction. One checks that this action preserves the quantum group relations. Further, any morphism in $\lYD{U_q(\fr{n}_-)}(\cB)$ is a morphism of the corresponding $U_q(\fr{g})$-module. As the pairing is non-degenerate, any $U_q(\fr{n}_+)$-locally finite $U_q(\fr{g})$-module, with action of the $K_i$ induced by an $L$-grading, arises this way. This uses that the dual $R$-matrix of $\Bbbk L$ induces a non-generate Hopf algebra pairing of $\Bbbk L$ with the group algebra $K$ generated by the $K_i$ (using that $q$ is generic).
\end{proof}

Note that the above conditions on $U_q(\fr{g})$ modules are two out of the three conditions used to define the category $\cO_q$ for quantum groups in \cite{AM}. Modules in $\cZ_{\cB}(\cM)$ are not necessarily finitely generated. Note that, however, unlike $\cZ_{\cB}(\cM)$, $\cO_q$ is not a monoidal category as the requirement of finite generation is not stable under taking tensor products over $\mF$.

The category $\cM\boxtimes_{\cB}\cM^{\oop}$ is, by Proposition \ref{Bmodreltensor} and Example \ref{tensorBBmod}, equivalent as a monoidal category to $\lrmod{U_q(\mathbf{n}_-)}{U_q(\mathbf{n}_-)}(\lcomod{L}_q)$. 
This category can be identified as the category $\lmod{T_q(\fr{g})}^{\op{w}}$ of \emph{weight modules} (that is, modules satisfying condition (i) from Theorem \ref{quantumcenter}) over the following Hopf algebra:

\begin{definition}\label{Tqdef}
Consider the Hopf algebra $T_q(\fr{g})$ generated by $x_i, y_i, K_i, K_i^{-1}$ for $i\in I$, subject to the relations
\begin{gather}
x_iy_j=q^{i\cdot j}y_jx_i,\qquad
K_ix_j=q^{-i\cdot j}x_jK_i,\qquad
K_iy_j=q^{-i\cdot j}y_jK_i,\qquad K_i^{\pm 1}K_i^{\mp 1}=1\label{bosonT}\\
\sum_{k=0}^{1-a_{ij}}(-1)^k\binom{1-a_{ij}}{k}_{q_i}x^{1-a_{ij}-k}_i x_jx_i^k=0,\qquad \sum_{k=0}^{1-a_{ij}}(-1)^k\binom{1-a_{ij}}{k}_{q_i}y^{1-a_{ij}-k}_i y_jy_i^k=0,\label{serrerel}
\end{gather}
with coproducts defined on generators by
\begin{align}\label{Tcoprod}
\Delta(x_i)&=x_i\otimes 1+K_i^{-1}\otimes x_i, &\Delta(y_i)&=y_i\otimes 1+K_i\otimes y_i, &\Delta(K_i)=K_i\otimes K_i.
\end{align}
\end{definition}

It was shown in \cite{Mas}*{Example~5.4} that $U_q(\fr{g})$ is related to a version of $T_q(\fr{g})$ via so-called \emph{Doi twist}. That is, a two-sided twist of the algebra structure by a $2$-cocycle. 

The constructions of this paper, see Section \ref{Moritasection}, imply that given a categorical module over $\lmod{T_q(\fr{g})}^{\op{w}}$ of weight modules over $T_q(\fr{g})$, we can construct a categorical module over the category $\lmod{\Ug}^{U_q(\fr{n}_+)\mathrm{-lfw}}$. Further, this constructions is bifunctorial.


\subsection{Applications to Small Quantum Groups}\label{smallqgroupssect}

In this section, we explain how to obtain the small quantum group $u_\epsilon(\fr{g})$ as a braided Drinfeld double, and apply the results of this paper to this example as a main application. 

Let $(I,\cdot)$ be a Cartan datum (as in the previous section) which is assumed to be of finite type. Denote the associated Lie algebra by $\fr{g}$. Assume that $\Bbbk$ is a field of characteristic zero and $\epsilon\in \Bbbk$ a primitive $l$-th root of unity, where $l\geq 3$ is an odd integer (assume that $l$ is coprime to $3$ if $\fr{g}$ contains a $G_2$-factor). Note that $\epsilon^2$ is also a primitive $l$-th root of unity. We set $\epsilon_i:=\epsilon^{i\cdot i/2}$.

Let $K=\langle k_1,\ldots, k_n\rangle$ be the abelian group generated by $k_i$ such that $k_i^l=1$. Consider the group algebra $\Bbbk K$ as a quasi-triangular Hopf algebra with universal $R$-matrix given by 
\begin{align}
R=\frac{1}{|I|^l}\sum_{\mathbf{i},\mathbf{j}\,\in\, \mZ_l\langle I\rangle} \epsilon^{-\mathbf{i}\cdot \mathbf{j}}k_{\mathbf{i}}\otimes k_{\mathbf{j}}.
\end{align}

Similarly to Section \ref{qgroupssect}, we can define a $\Bbbk K$-module $E=\Bbbk\langle e_i~\mid~i\in I \rangle$, with action given by $k_i\triangleright e_j=\epsilon^{i\cdot j}e_j$. 
The universal $R$-matrix induces a coaction on $E$ given by
\begin{align}
\delta(e_i)=R^{(2)}\otimes (R^{(1)}\triangleright e_i)=k_i\otimes e_i.
\end{align}
Consider the dual $\Bbbk K$-module $F=\Bbbk\langle f_i~\mid~i\in I \rangle$, with action given by $k_i\triangleright f_j=\epsilon^{-i\cdot j}f_j$. Following \cite{Lus}, the pairing $\langle ~,~\rangle\colon E\otimes F\to \Bbbk$, $\langle f_i,e_j \rangle=\delta_{i,j}(\epsilon_i-\epsilon_i^{-1})^{-1}$ uniquely extends to a pairing 
$
\ev\colon T(E)\otimes T(F)\longrightarrow \Bbbk.
$
of primitively generated Hopf algebras in $\lmod{K}_q$. The quotients by the left and right radical of the pairing $\ev$ yields braided Hopf algebras $\cB(E)$, $\cB(F)$, so-called \emph{Nichols algebras} \cite{AS}, such that the induced non-degenerated pairing $\ev$ on $\cB(E)\otimes \cB(F)$ is one of Hopf algebras. It is well-known, see e.g. \cite{AS3}*{Section~3}, that $\cB(E)$ can be identified with $u_\epsilon(\fr{n}_+)$, and $\cB(F)$ with $u_\epsilon(\fr{n}_-)$. In addition to the quantum Serre relations of Eq. (\ref{qserrerel}), with $e_i, f_i,\epsilon_i$ instead of $E_i,F_j,q_i$,
these algebras satisfy the nilpotency relations 
$e_i^l=0$, $f_i^l=0$, for all $i\in I$. 
The following theorem is derived similarly to \cite{Lau2}*{Theorem 3.25}. Details on the version of the small quantum group $u_\epsilon(\fr{g})$ (in the De Concini--Kac--Procesi form) used here can be found in \cite{BG}*{Chapter III.6}.\footnote{We use the coopposite Hopf algebra of the one appearing there.}

\begin{theorem}\label{smallthm}
There is an isomorphism of Hopf algebras between $\Drin_{\Bbbk K}(\cB(E),\cB(F))$ and $u_\epsilon(\fr{g})$.
\end{theorem}
\begin{proof}
The braided Drinfeld double is the Hopf algebra generated by $e_i,f_i,k_i^{\pm 1}$ subject to the relations
\begin{gather*}
k_ie_j=\epsilon^{i\cdot j}e_jk_i, \quad k_if_j=\epsilon^{-i\cdot j}f_jk_i,\quad k_i^{\pm 1}k_i^{\mp 1}=1,\quad e_if_j-\epsilon^{i\cdot j}f_je_i=\delta_{i,j}\frac{1-k_i^{-2}}{\epsilon_i-\epsilon_i^{-1}},\\
\sum_{k=0}^{1-a_{ij}}(-1)^k\binom{1-a_{ij}}{k}_{q_i}e^{1-a_{ij}-k}_i e_je_i^k=0,\qquad \sum_{k=0}^{1-a_{ij}}(-1)^k\binom{1-a_{ij}}{k}_{q_i}f^{1-a_{ij}-k}_i f_jf_i^k=0,
\end{gather*}
with coproducts given by
\begin{gather*}
\Delta(k_i)=k_i\otimes k_i, \qquad \Delta(e_i)=e_i\otimes 1+k^{-1}_i\otimes e_i,\qquad \Delta(f_i)=f_i\otimes 1+k^{-1}_i\otimes f_i.
\end{gather*}
One now checks that the mapping $\phi(k_i)=K_i$, $\phi(e_i)=K_i^{-1}E_i$, $\phi(f_i)=F_i$ extends to an isomorphism of Hopf algebras   $\phi\colon \Drin_{\Bbbk K}(\cB(E),\cB(F)) \to u_\epsilon (\fr{g})$.
\end{proof}

\begin{corollary}
Denote $\cB=\lmod{K}_\epsilon$ and $\cM=\lmod{u_\epsilon(\fr{n}_-)}(\cB)$. Then there is an equivalence of monoidal categories between $\cZ_\cB(\cM)$ and $\lmod{u_\epsilon(\fr{g})}$.
\end{corollary}

We can now apply Theorem \ref{centermorita} to the $\cB$-augmented monoidal category $\cM=\lmod{u_\epsilon(\fr{n}_-)}^{\mathrm{fd}}(\cB)$.

\begin{corollary}\label{smallqcor}
The monoidal categories $\lmod{u_\epsilon(\fr{g})}^{\mathrm{fd}}$ and $\cM\boxtimes_\cB\cM^{\oop}$ are categorically Morita equivalent. The latter category is equivalent as a monoidal category to $\lrfmod{u_\epsilon(\fr{n}_-)}{u_\epsilon(\fr{n}_-)}(\cB)$.
\end{corollary}

\begin{definition}\label{tepdef}
Denote by $t_\epsilon(\fr{g})$ the algebra generated by $x_i,k_i,y_i$ for $i\in I$, subject to the relations obtained from Eqs. (\ref{bosonT})--(\ref{serrerel}) by replacing $q_i$ with $\epsilon_i$, and the additional relations
\begin{align}
x_i^l&=0, &y_j^l&=0, &\forall i\in I,
\end{align}
and Hopf algebra structure given using the same coproduct formulas as in Eq. (\ref{Tcoprod}).
\end{definition}

\begin{corollary}
The monoidal categories $\cM\boxtimes_\cB\cM^{\oop}$ and $\lmod{t_\epsilon(\fr{g})}$ are equivalent.
\end{corollary}

\appendix
\section{Existence of Relative Tensor Products}\label{appendixA}

\begin{proof}[Proof of Theorem \ref{reltensorexistence}]
We prove the theorem in two steps. First, we show the existence of a $\Bbbk$-linear category $\cV\otimes_{\cC}\cW$, which is then completed under finite colimits to give $\cV\boxtimes_{\cC}\cW$.

In the first step, we work in the $2$-category of small $\Bbbk$-linear categories $\Cat$. That is, we do not require the existence of finite colimits (or biproducts) at this stage. The $2$-category $\Cat$ is cartesian closed, with tensor product $\otimes$ and cocomplete (see \cite{Wol3}) as it is the category of $\Vect$-categories. We also require the existence of conical pseudo-colimits in $\Cat$ (see \cite{Kel3}*{Section 5}). Indeed, it was shown in \cite{Str} that such conical pseudo-colimits reduce to indexed colimits, and hence exist.

We fix some notation. Denote by $\bfD$ the $2$-category described by
\begin{gather*}
\bfD:=\Big(\xymatrix{1\ar@/^/@<1.5ex>[rrr]^{l_1}\ar[rrr]^{m}\ar@/_/@<-1.5ex>[rrr]_{r_1}&&&2\ar@<1ex>[rrr]^{l_2}\ar@<-1ex>[rrr]_{r_2}&&&3}\Big),\\
l_2 r_1=r_2l_1, \qquad \chi\colon r_2 m\natisomorph r_2r_1, \qquad \xi\colon l_2 m\natisomorph l_2l_1.
\end{gather*}
Given a right $\cC$-module $\cV$ and a left $\cC$-module $\cW$, we can define a strict $2$-functor $\Gamma\colon \bfD\to \Cat$ by the diagram
\begin{gather*}
\Gamma:=\Big(\xymatrix{\cV\otimes \cC\otimes \cC\otimes \cW\ar@/^/@<1.5ex>[rrr]^{l_1=\triangleleft \otimes \ide_{\cC}\otimes\ide_{\cW}}\ar[rrr]^{m=\ide_{\cV} \otimes\, \otimes^{\cC}\,\otimes \ide_{\cW}}\ar@/_/@<-1.5ex>[rrr]_{r_1=\ide_{\cW} \otimes \ide_{\cC}\otimes \triangleright}&&&\cV\otimes \cC\otimes \cW\ar@<1ex>[rrr]^{l_2=\triangleleft \otimes \ide_{\cW}}\ar@<-1ex>[rrr]_{r_2=\ide_{\cW} \otimes \triangleright}&&&\cV\otimes \cW}\Big),\\
\Gamma(\chi)=\chi\otimes \ide_{\cW}, \qquad \Gamma(\xi)=\ide_{\cV}\otimes \xi.
\end{gather*}
where by slight abuse of notation, $r_1$, $l_1$, $r_2$, $l_2$ and $m$ also denote their images under $\Gamma$.

In short, the relative tensor product $\cV\otimes_\cC \cW$ of a right $\cC$-module $\cV$ and a left $\cC$-module $\cW$ can be defined as the (conical) pseudo-colimit (see \cite{Kel3}*{Eq. (5.7)}) over $\Gamma$ in $\Cat$ 
which exists by the above considerations. 

We shall now spell out the data and universal properties involved in the definition of $\cV\otimes_\cC \cW$ in detail. Let $\cD$ be a $\Bbbk$-linear category and assume the existence of a diagram
\begin{align}\label{cone}
\vcenter{\hbox{
\xymatrix{\cV\otimes \cC\otimes \cC\otimes \cW\ar@/_6pc/[rrrrrrrd]_{\rF_1}\ar@/^/@<1.5ex>[rrr]^{l_1=\triangleleft \otimes \ide_{\cC}\otimes\ide_{\cW}}\ar[rrr]^{m=\ide_{\cV} \otimes\, \otimes^{\cC}\,\otimes \ide_{\cW}}\ar@/_/@<-1.5ex>[rrr]_{r_1=\ide_{\cW} \otimes \ide_{\cC}\otimes \triangleright}&&&\ar@<2ex>@{<=}[dl]_{\rho_{l_1}}\ar@<6ex>@{<=}[dl]_{\rho_{m}}\ar@<8ex>@{<=}[dl]^{\rho_{r_1}}\cV\otimes \cC\otimes \cW\ar@/_3pc/[rrrrd]_{\rF_2}\ar@<1ex>[rrr]^{l_2=\triangleleft \otimes \ide_{\cW}}\ar@<-1ex>[rrr]_{r_2=\ide_{\cW} \otimes \triangleright}&&&\cV\otimes \cW\ar[dr]^{\rF_3}\ar@<2ex>@{<=}[dl]^{\rho_{r_2}}\ar@{<=}[dl]_{\rho_{l_2}}&\\
&&&&&&&\cD,}}
}
\end{align}
i.e. there are natural isomorphisms
\begin{align*}
\rho_{l_1}\colon \rF_1\natisomorph \rF_2 l_1, &&\rho_{r_1}\colon \rF_1\natisomorph \rF_2 r_1, &&\rho_{m}\colon \rF_1\natisomorph \rF_2 m, &&\rho_{l_2}\colon \rF_2\natisomorph \rF_3 l_2, &&\rho_{r_2}\colon \rF_2\natisomorph \rF_3 r_2,
\end{align*}
satisfying 
\begin{align}\label{coherence1}
(\rho_{r_2}\circ_0\ide_{l_1})\rho_{l_1}&=(\rho_{l_2}\circ_0\ide_{r_1})\rho_{r_1}, \\ \rF_3(\chi\otimes \ide_{\cW})(\rho_{r_2}\circ_0\ide_{m})\rho_{m}&=(\rho_{r_2}\circ_0\ide_{r_1})\rho_{r_1},\label{coherence2}\\
\rF_3(\ide_{\cV}\otimes\xi )(\rho_{l_2}\circ_0\ide_{m})\rho_{m}&=(\rho_{l_2}\circ_0\ide_{l_1})\rho_{l_1}.\label{coherence3}
\end{align}
Note that the data of such a diagram is the same as a pseudo $2$-natural transformation $\Delta \cI \Longrightarrow \Fun(\Gamma(-),\cD)$, where $\Delta \cI$ is the constant $2$-functor from the (opposite of) the $2$-category $\bfD$ to $\Cat$, see \cite{Kel3}*{Section~5}. (In the case of the simple diagram $\bfD$, (\ref{coherence1})--(\ref{coherence3}) are all coherences for the morphisms of the form $\rho$.) 

Together with modifications, the pseudo $2$-natural transformations $\cI \Longrightarrow \Fun(\Gamma(-),\cD)$ form a $\Bbbk$-enriched category, which we denote by $\Cone(\Gamma,\cD)$. 
That is, $\Cone(\Gamma,\cD)$ has diagrams of the form (\ref{cone}) as objects. Given two such diagram, where in the second one the functors and natural isomorphisms are equipped with a dash, e.g. $\rF_1'$, $\rho_{r_1}'$, a morphisms in $\Cone(\Gamma,\cD)$ consists of natural transformations $\mu_i\colon \rF_i\Longrightarrow \rF_i'$, such that the following diagrams commute (for $i=1,2$):
\begin{align}\label{conemorphism}
\vcenter{\hbox{
\xymatrix{\rF_i\ar@{=>}[r]^{\rho_{l_i}}\ar@{=>}[d]_{\mu_i} &\rF_{i+1} {l_i}\ar@{=>}[d]^{\mu_{i+1}\circ_0 l_i}\\
\rF_i'\ar@{=>}[r]^{\rho_{l_i}'} &\rF_{i+1}' l_i,
}}} &&\vcenter{\hbox{
\xymatrix{\rF_i\ar@{=>}[r]^{\rho_{r_i}}\ar@{=>}[d]_{\mu_i} &\rF_{i+1} r_i\ar@{=>}[d]^{\mu_{i+1}\circ_0 r_i}\\
\rF_i'\ar@{=>}[r]^{\rho_{r_i}'} &\rF_{i+1}' r_i,
}}}
&&\vcenter{\hbox{
\xymatrix{\rF_1\ar@{=>}[r]^{\rho_{m}}\ar@{=>}[d]_{\mu_1} &\rF_{2} m\ar@{=>}[d]^{\mu_{2}\circ_0 m}\\
\rF_1'\ar@{=>}[r]^{\rho_{m}'} &\rF_{2}' m.
}}}
\end{align}

The relative tensor product $\cV\otimes_\cC \cW$ comes together with the data of a \emph{counit} in the language of \cite{Kel3}, i.e. a pseudo $2$-natural transformation $\sigma\colon \Delta\cI\Longrightarrow \Fun(G(-), \cV\otimes_\cC \cW)$. That is, there exists a diagram as in (\ref{cone}) for $\cD=\cV\otimes_\cC \cW$, where we denote the appearing functors by $\Psi_1$, $\Psi_2$, and $\Psi_3$; and the natural isomorphisms by $\sigma_x$, for $x=r_1,l_1, r_2, l_2,m$. Given a functor $\rF\colon \cV\otimes_\cC \cW\to \cD$, the counit provides a diagram of the form (\ref{cone}). The universal property of $\cV\otimes_\cC \cW$ now asserts that this assignment gives an \emph{isomorphism} of categories
$$\Cone(\Gamma,\cD)\cong \Fun(\cV\otimes_\cC \cW,\cD).$$

Next, we claim that there is an \emph{equivalence} of $\Bbbk$-enriched categories 
$$\Cone(\Gamma,\cD)\simeq \Fun^{\bal{\cC}}(\cV\otimes \cW, \cD).$$
An object of $\Cone(\Gamma,\cD)$ already comes with a functor $\rF:=\rF_3\colon \cV\otimes \cW\to  \cD$, and a $\cC$-balancing isomorphism can be defined using $\eta:=\rho_{r_2}\rho_{l_2}^{-1}$. We need to verify the coherence diagram (\ref{balancedcoh}). This follows from commutativity of the diagram
\begin{align*}\xymatrix{
\rF_3(((V\triangleleft C)\triangleleft D)\otimes W)\ar[d]^{\rho_{l_2}^{-1}\circ_0 \ide_{l_1}}&&\rF_3((V\triangleleft (C\otimes^{\cC} D))\otimes W)\ar[ll]_{\rF_3(\xi\otimes \ide)}\ar[d]^{\rho_{l_2}^{-1}\circ_0 \ide_m}\\
\rF_2((V\triangleleft C)\otimes D\otimes W)\ar[d]^{\rho_{r_2}\circ_0 \ide_{l_1}}\ar[dr]^{\rho_{l_1}^{-1}}&&\rF_2(V\otimes (C\otimes^{\cC} D)\otimes W)\ar[dl]^{\rho_{m}^{-1}}\ar@{=}[dd]\\
\rF_3((V\triangleleft C)\otimes (D\triangleright W))\ar[d]^{\rho_{l_2}^{-1}\circ_0 \ide_{r_1}}&\rF_1(V\otimes C\otimes D\otimes W)\ar[dl]^{\rho_{r_1}}\ar[dr]^{\rho_{m}}&\\
\rF_2(V\otimes C\otimes (D\triangleright W))\ar[d]^{\rho_{r_2}\circ_0 \ide_{r_1}}&&\rF_2(V\otimes (C\otimes^{\cC} D)\otimes W)\ar[d]^{\rho_{r_2}\circ_0 \ide_{m}}\\
\rF_3(V\otimes (C\triangleright (D\triangleright W)))&&\ar[ll]^{\rF_3(\ide\otimes \chi)}\rF_3(V\otimes ((C\otimes^{\cC} D)\triangleright W)),
}
\end{align*}
for objects $V$ in $\cV$, $W$ in $\cW$ and $C,D$ in $\cC$. Commutativity of the smaller, enclosed, diagrams follows from Eqs. (\ref{coherence1})--(\ref{coherence3}). This implies commutativity of the outer diagram, which implies Eq. (\ref{balancedcoh}). A morphism $\mu$ in $\Cone(\Gamma,\cD)$ clearly gives a natural transformation $\mu_1\colon \rF\to \rF'$, which commutes with the balancing as required in Eq. (\ref{balancedmor}) using Eq. (\ref{conemorphism}). Hence, we have constructed a functor
\[
\Upsilon \colon \Cone(\Gamma,\cD)\longrightarrow  \Fun^{\bal{\cC}}(\cV\otimes \cW, \cD),
\]
which is $\Bbbk$-linear by construction.

Conversely, we want to construct a $\Bbbk$-linear functor
\[
\Omega \colon \Fun^{\bal{\cC}}(\cV\otimes \cW, \cD)\longrightarrow  \Cone(\Gamma,\cD).
\]
Assume given a $\cC$-balanced functor $\rF$ with balancing isomorphism $\eta$, we define
\begin{align*}
\rF_3&=\rF, & \rF_2&=\rF(l_2), & \rF_3&=\rF((l_2\circ_0 \ide_{l_1})l_1),\\
\rho_{l_2}&=\ide, & \rho_{r_2}&=\eta, & \rho_m&=\rF(\xi^{-1}\otimes\ide), \\
\rho_{l_1}&=\ide, & \rho_{r_1}&=\eta\circ_0 \ide_{l_1}.
\end{align*}
Equation (\ref{coherence1}) follows as it is an equality by definition, and Equation (\ref{coherence2}) follows from Equation (\ref{balancedcoh}) for the balancing isomorphism $\eta$, while Equation (\ref{coherence3}) is just the invertibility condition of $\rF(\xi\otimes \ide)$. Thus, we obtain an object in $\Cone(\Gamma,\cD)$. Next, giving a morphism $\phi\colon (\rF,\eta)\Rightarrow (\rF', \eta')$ of $\cC$-balanced functors, we define
\begin{align*}
\mu_1&=\phi, & \mu_2&=\phi\circ_0 \ide_{l_2}, & \mu_3&=\phi\circ_0 \ide_{(l_2\circ_0\ide)l_1}.
\end{align*}
It is clear that the diagrams in (\ref{conemorphism}) commute. For $r_1, l_1, r_2, l_2 $ this follows from compatibility of $\eta$ and the identities with $\phi$, while for $m$ this follows from the compatibility of $\phi$ with the $\cC$-module structure of $\cV$. This assignment is clearly $\Bbbk$-linear and functorial, providing a functor $\Omega$ as required.

Next, we show that $\Upsilon$, $\Omega$ give an equivalence of categories. It is clear that $\Upsilon\Omega\cong \ide$. We apply $\Omega\Upsilon$ to an object $(\rF,\rho)$ of $\Cone(\Gamma,\cD)$ and denote the image by $(\ov{\rF},\ov{\rho})$. Then we define an isomorphism $\mu\colon (\rF,\rho)\to (\ov{\rF},\ov{\rho})$ by
\begin{align*}
\mu_1&=(\rho_{l_2}\circ_0\ide_{l_1})\rho_{l_1}, & \mu_2&=\rho_{l_2}, & \mu_3&=\ide.
\end{align*}
It is then readily verified that all diagrams in Eq. (\ref{conemorphism}) commute, using that
\begin{align*}
\ov{\rF_3}&=\rF_3, & \ov{\rF_2}&=\rF_3(l_2), & \ov{\rF_3}&=\rF_3((l_2\circ_0 \ide_{l_1})l_1),&\ov{\rho_{l_1}}&=\ide,\\
\ov{\rho_{l_2}}&=\ide, & \ov{\rho_{r_2}}&=\rho_{r_2}\rho_{l_2}^{-1}, & \ov{\rho_m}&=\rF_3(\xi^{-1}\otimes\ide),  &\ov{\rho_{r_1}}&=\rho_{r_2}\rho_{l_2}^{-1}\circ_0 \ide_{l_1}.
\end{align*}
Note that the square involving $\rho_m$ commutes by Eq. (\ref{coherence3}). As all $\mu_i$ are invertible, the claim follows.

To this point, we have shown that $\cV\otimes_{\cC}\cW$ exists in $\Cat$. By \cite{Kel}*{Theorem 5.35} we can consider the closure $\ov{\cV\otimes_{\cC}\cW}$ of $\cV\otimes_{\cC}\cW$ under finite colimits (which form a small class of indexing types). Then there is an equivalence of $\Bbbk$-linear categories
\begin{align*}
\Fun(\cV\otimes_{\cC}\cW,\cD)\simeq \Fun^c(\ov{\cV\otimes_{\cC}\cW},\cD)
\end{align*}
which is induced by composition with the inclusion functor $\cV\otimes_{\cC}\cW\hookrightarrow \ov{\cV\otimes_{\cC}\cW}$. Hence, we denote $\cV\boxtimes_{\cC}\cW=\ov{\cV\otimes_{\cC}\cW}$ which is a category in $\Cat^c$ and claim that it satisfies the universal property defining the relative tensor product of $\cV$ and $\cW$ over $\cC$ as in Definition \ref{relativetensordef}.

Indeed, since $\cV\boxtimes \cW$ has finite coproducts, there exists a functor $\rT$
\begin{align*}
\xymatrix{\cV\otimes \cW\ar[d]^{\rT'}\ar@{^{(}->}[r]^{\iota}& \cV\boxtimes \cW\ar[d]^{T}\\
\cV\otimes_{\cC} \cW\ar@{^{(}->}[r]^{\iota_{\cC}}& \cV\boxtimes_{\cC} \cW
}
\end{align*}
making the diagram commute. Now, given a $\cC$-balanced functor $\rF\colon \cV\boxtimes \cW\to \cD$, we can restrict to a $\cC$-balanced functor $\rF'=\rF\iota\colon \cV\otimes \cW\to \cD$ which, under equivalence, corresponds to a functor $\rG'\colon \cV\otimes_{\cC} \cW\to \cD$ such that $\rF'=\rG'\rT'$. Such a functor extends to a functor $\rG\colon \cV\boxtimes_{\cC} \cW\to \cD$ such that $\rG'=\rG\iota_{\cC}$, and hence $\rF'=\rG\iota_{\cC}\rT'=\rG\rT\iota$, and thus $\rF\cong \rG\rT$. This way, we obtain an equivalence 
\begin{align*}
\Fun^c(\cV\boxtimes_{\cC}\cW,\cD)\simeq \Fun^{\bal{\cC}}(\cV\boxtimes\cW,\cD)
\end{align*}
induced by the functor $\rT$. Since $\cV\boxtimes_{\cC} \cW$ is defined using two universal properties which satisfy naturality, it is natural in both components $\cV$ and $\cW$.

Note that given an equivalence, after choosing an inverse equivalence with the required natural isomorphisms, these can be altered to obtain an adjoint equivalence \cite{Mac}*{Section IV.4}.
\end{proof}


%
%
%
%


\bibliography{biblio2}
\bibliographystyle{amsrefs}

\end{document}